\documentclass[a4paper,11pt,reqno]{article}  
\usepackage{graphicx,color,multirow, here}
\usepackage{mathtools}
\usepackage[utf8]{inputenc} 
\usepackage[T1]{fontenc}
\usepackage[english]{babel}
\usepackage{amsmath, amsfonts, amsthm,amssymb,here,dsfont, hyperref}
\usepackage{natbib}
\usepackage{comment, enumitem}

\usepackage{caption}
\usepackage{subcaption}

\usepackage{makecell}

\usepackage{tikz}
\usetikzlibrary{arrows.meta, positioning, fit, calc}
\allowdisplaybreaks

\newcommand{\argmax}[1]{\underset{#1}{\operatorname{arg}\!\operatorname{max}}\;}
\newcommand{\argmin}[1]{\underset{#1}{\operatorname{arg}\!\operatorname{min}}\;}

\newtheorem{theo}{Theorem}[section]

\newtheorem{proposition}[theo]{Proposition}
\newtheorem{corollary}[theo]{Corollary}
\newtheorem{lemma}[theo]{Lemma}
\newtheorem{remark}[theo]{Remark}

\newtheorem{ass}{Assumption}

\newcommand{\R}{\mathbb{R}}
\newcommand{\w}{\widehat}

\newcommand{\wt}{\widetilde}
\newcommand{\wtX}{\widetilde{X}}
\newcommand{\one}{\mathds{1}}
\newcommand{\E}{\mathbb{E}}
\renewcommand{\P}{\mathbb{P}}

\newcommand{\cF}{\mathcal{F}}

\newcommand{\cN}{\mathcal{N}}
\newcommand{\cR}{\mathcal{R}}
\newcommand{\cY}{\mathcal{Y}}

\newcommand{\dd}{\mathrm{d}}

\newcommand{\cG}{\mathcal{G}}
\newcommand{\Nt}{N_{\rm test}}

\makeatletter

\@addtoreset{equation}{section}
\makeatother

\usepackage[textwidth=1.8cm, textsize=tiny]{todonotes}

\title{Who Drives the System? Classifying Mean-Field Particle Systems from Trajectory Data}

\author{Christophe Denis \thanks{SAMM, Université Paris 1}, Charlotte Dion-Blanc \thanks{LPSM, Sorbonne Université, Université Paris Cité, CNRS, F-75005 Paris, France.} \thanks{CMAP, Ecole Polytechnique, IP Paris, Palaiseau, France}, Yating Liu \thanks{CEREMADE, Universit\'e Paris Dauphine}}

\date{\today}

\begin{document}

\maketitle

\begin{abstract}
We study the supervised classification problem for interacting particle systems (IPS) from trajectory observations. 
The IPS belongs to one of $K$ classes, 
each characterized by a distinct interaction drift. Given a learning dataset, the goal is then to predict the class label of a newly observed system based on the trajectory of a single particle.
This setting raises several statistical challenges. First, the particles within each system interact and are therefore dependent. Second, although the full particle system is Markovian, the trajectory of a single particle is not. To address these difficulties, we exploit the McKean--Vlasov limit of the IPS, which describes the dynamics of a typical particle as the number of particles tends to infinity. 

We propose a plug-in classification procedure based on estimating the interaction drift associated with each class from discretely observed trajectory data. Hence, for each class, we provide a new nonparametric estimator of the drifts by minimizing a ridge-regularized least-squares contrast over a B-spline basis. 
In particular, our theoretical findings reveal that the convergence rate of the resulting classifier is of order $N^{-1/6+\varepsilon}$ for any $\varepsilon>0$, where $N$ denotes the number of particles in each system. Numerical experiments illustrate the performance of the proposed method and show that it outperforms an end-to-end neural network baseline that does not exploit the underlying particle-system structure.
\end{abstract}



\section{Introduction}\label{sec:intro}

Interacting particle systems (IPS) provide a natural framework for modeling large populations of agents whose dynamics depend both on their individual state and on interactions with the rest of the system. Such models arise in a wide range of fields, including statistical physics \cite{HuangLiuPickl2020,ChaintronDiez2022,bertoli2024stability}, biology and neuroscience \cite{erny2021conditional,locherbach2022metastability, carrillo2022noise}, financial mathematics \cite{Cardaliaguet2018Mean, Carmona-Delarue-I}, and more recently machine learning \cite{transformers}. In many applications, several types of interacting systems may coexist, each characterized by different interaction mechanisms. A natural statistical problem is therefore to identify the type of a system from the observation of its dynamics. Moreover, the need for a reliable classification algorithm for paths generated by interacting particle systems (IPS) is also motivated by the assessment of generative models,  as in the recent work of \cite{pham} where the authors use a classifier to determine whether generated data can be distinguished from real data. Another example where supervised classification could be needed, is for neurons. Indeed, we know that a population of neurons can change its dynamic regime depending on the cognitive
task or on a pathological condition. Then, it could be crucial to be able to know the state of the organism looking only at one or a small group of neurons.

In this work, we fill a gap in the literature by studying the supervised classification problem for interacting particle systems observed discretely over a fixed time interval. More precisely, we assume that the system belongs to one of $K$ possible classes, each class being characterized by a different drift function. For each class, we observe two particle systems of size $N$ at discrete time points. The objective is to classify a new system based on the observation of a single particle trajectory. Our approach relies on a plug-in principle: we first estimate the  drift function associated with each class and then construct a classifier based on these estimators, mimicking the oracle one. The main challenge is that the classifier must be learned from dependent data, since particles within each training system interact and are therefore not independent. As a result, classical learning procedures designed for independent samples cannot be applied directly.

\paragraph{Model.}
We consider a system of $N$ interacting particles $(X_t^{[Y],n})_{t\in[0,1],\,1\le n\le N}$. Each particle's position evolves according to the  
simple general model which is the result of pairwise interactions
with all remaining particles. Its dynamic is described by the following equation, for $1\le n\le N $,
\begin{equation}\label{eq:sys}
\dd X_t^{[Y],n}
=
\frac{1}{N}\sum_{n'=1}^N 
b_Y^{*}\!\left(X_t^{[Y],n},X_t^{[Y],n'}\right)\dd t
+ \dd W_t^{n},
\qquad X_0^{[Y],n}=x_0 \,\in\R,
\end{equation}
where $(W^n)_{n\ge1}$ are independent Brownian motions and 
$Y\in\{1,\ldots,K\}$ denotes the class label of the system. 
Conditionally on $Y=k$, the interaction among particles is described by the function $b_k^{*}$. 

Such systems are exchangeable and, as the number of particles $N$ grows, their behavior is well approximated by the solution of a McKean–Vlasov equation 
\begin{equation}\label{eq:mkv}
\dd \xi_t^{[Y]}
=\left(\int_{\R}
b_Y^{*}\big(\xi_t^{[Y]},x\big)\mu_t^{[Y]}(\dd x)\right)\dd t
+ \dd W_t, \quad \xi_0^{[Y]}=x_0,
\end{equation}
where for every $t\in[0,1]$ and for each label $Y=k, \,k\in\mathcal{Y}$, $\mu_t^{[k]}=\text{Law}(\xi_t^{[k]})$ denotes the probability distribution of $\xi_t^{[k]}$, and $W=(W_t)_{t\in[0,1]}$ is an $(\mathcal{F}_t)$-standard Brownian motion, independent of the label $Y$.

The approximation of the dynamics of a typical particle in \eqref{eq:sys} by the McKean--Vlasov equation \eqref{eq:mkv} is justified by the propagation of chaos property; see, \emph{e.g.}, \cite{mckean1966class,meleard2006asymptotic,jourdainmalrieu}. The statistical objective of this paper is to predict the label $Y$ from the observation of a single particle trajectory generated by the interacting system \eqref{eq:sys}, without requiring access to
the full particle system. To construct the classifier, we therefore rely on the McKean--Vlasov limit \eqref{eq:mkv} as an approximate description of the dynamics of a typical particle.

\paragraph{Related literature.}
Multiclass classification is a well-established research area, with supervised classification of temporal data attracting considerable attention in recent years, on can cite \cite{ramsaySilverman, wang2016functional, NEURIPS2022_times_series}. Among the most closely related works to ours are those that study supervised classification within the framework of stochastic differential equations (SDEs). Several contributions have addressed this problem, including~\cite{cadre2013, DDM2019, gadat2020, DDMVC2022, zhao2026plug}. In these works, the goal is to construct empirical classifiers that discriminate between classes based on the drift of a diffusion process. The training sample is typically assumed to consist of $N$ independent realizations of time-homogeneous diffusion processes. Several types of classification procedures have been investigated. Empirical risk minimization procedures are considered in~\cite{cadre2013, DDM2019}, while plug-in classifiers are studied in~\cite{DDMVC2022} for univariate diffusion processes and in~\cite{zhao2026plug} for the multivariate setting. In particular, rates of convergence are established under regularity assumptions on the drift functions, such as Lipschitz continuity.

However, these approaches cannot be directly applied to our setting, since the particles of an interacting particle system (IPS) are not independent. In particular, the strong dependencies among particles introduce significant statistical challenges. Moreover, although the full particle system, viewed as an $\mathbb{R}^N$-valued diffusion process, is Markovian, the trajectory of an individual particle is generally not. Consequently, recent statistical methods developed for Markov processes, such as those proposed in~\cite{flamary}, cannot be directly applied to our setting.

Our proposed classification procedure relies on nonparametric estimators of the drift functions. 
In recent years, there is  a rich literature  on statistical inference for interacting particle systems and McKean–Vlasov equations and a significant part of this literature focuses on the estimation of the drift function. Parametric and nonparametric approaches have been developed under various observation schemes, including discrete-time observations and long time horizons; see for instance \cite{AMORINOMKV,CGCL2025,amorino2024polynomial}. Kernel-based estimators of the full drift function have been proposed in \cite{hoffmann2021}, where the number of particles tends to infinity while the observation time remains fixed. Other works investigate least-squares approaches for particular interaction structures \cite{belopodol} or learning strategies based on neural networks \cite{beloDeep}. These contributions mainly address estimation problems and often rely on either ergodicity of the dynamics or large-sample asymptotics.

\paragraph{Our contribution.}
In this work, we propose a classification procedure $\w{g}$ for particles originating from an interacting particle system (IPS), assuming that two independent training samples from each class are available. We consider the setting in which the particles of the IPS are observed at discrete time points. The performance of a classifier is evaluated through its misclassification risk, ${\mathcal{R}}(\w{g}) = \mathbb{P}\left(\w{g}(Z^{[Y]}) \neq Y\right)$, 
where $Z^{[Y]}$ denotes the observed trajectory of a particle from an IPS associated with class $Y$, whose dynamics are given by \eqref{eq:sys}.

Exploiting the propagation of chaos property, we rely on the associated McKean--Vlasov limit equation~\eqref{eq:mkv} to obtain an explicit characterization of the Bayes classifier $g^*$. This characterization reveals the relationship between the coefficients of the McKean--Vlasov equation and the optimal classification rule. Motivated by this result, we propose a plug-in classifier based on ridge-type estimators of drift functions $b_k^*$. These estimators are obtained by minimizing an $L_2$ contrast over a constrained bivariate $B$-spline space.

Compared with empirical risk minimization, plug-in methods provide a more explicit characterization of the mechanisms underlying the differences between classes through the estimation of their drift functions. Moreover, by estimating the class-conditional distributions, they provide a direct probabilistic interpretation of the resulting posterior class probabilities, which can be further used for uncertainty quantification, decision-making under different loss functions, or calibration.
However, the theoretical analysis of plug-in classifiers in this setting raises several challenges. These include the dependence among particles, the structure of the interaction function, and the interplay between the sizes of the particle systems used for training and prediction and the time-discretization step $\Delta$. In addition, it requires deriving convergence rates for the drift functions over a non-compact interval, which leads to additional technical difficulties. To assess the performance of a classifier applied to a particle from an IPS, we introduce the notion of mean-field consistency, which requires the risk of $\w{g}$ to converge to the Bayes risk as the size $\Nt$ of the particle system containing the test particle tends to infinity.

To the best of our knowledge, this is the first work to study the problem of supervised classification for particles originating from interacting particle systems, and to provide theoretical guarantees for a classifier that explicitly exploits the underlying particle-system structure.
Our main contributions are summarized as follows:

\begin{enumerate}
\item[$i)$] We derive a general upper bound on the excess risk of the plug-in classifier, explicitly characterizing its dependence on the training system size $N$, the prediction system size $\Nt$, the time-discretization step $\Delta$, and the estimation error of the drift functions.

\item[$ii)$] We establish a convergence rate of order $N^{-1/6}$ for the estimation error of the bivariate Lipschitz drift functions as the size $N$ of the training particle systems increases.

\item[$iii)$] We show that the excess risk of the plug-in classifier based on the ridge estimators of the drift functions converges at the rate
$\max\left(\Nt^{-1}, N^{-1/6+\varepsilon}\right)$,
for any $\varepsilon > 0$.
\end{enumerate}

Finally, we complement our theoretical results with numerical experiments that assess the performance of the proposed classifier and compare it with an end-to-end neural network classifier that does not explicitly exploit the structure of the underlying particle system.

\paragraph{Organization of the paper.}
Section~\ref{sec:maths} introduces the interacting particle system and its main properties. Section~\ref{sec:classif} presents the plug-in classification strategy that relies on preliminary estimators of the drift functions. Then, Section~\ref{sec:inference} studies the estimation of the drift function and derives the corresponding convergence rates. In Section~\ref{sec:classiffinal} we obtain the convergence rate for the resulting classifier. Finally, Section~\ref{sec:num} provides numerical experiments that illustrate the effectiveness of the proposed method.

\section{Mathematical model}\label{sec:maths}

In this section, we introduce the model under consideration and formulate the main assumptions underlying our analysis. Section~\ref{subsec:ModelAssump} presents the model and its assumptions, whereas Section~\ref{subsec:PropChaos} examines the propagation of chaos property.

\subsection{Model and assumptions}
\label{subsec:ModelAssump}

We place ourselves in a filtered probability space $(\Omega, \mathcal{F}, (\mathcal{F}_t)_{t\in[0, 1]}, \P)$ that satisfies the usual conditions with a fixed time horizon $T=1$. Let $\mathcal{P}(\R)$ denote the set of the probability measures on $\R$ and let $\mathcal{P}_p(\R)$ denote the set of the probability distribution on $\R$ with $p$-th finite moment, $p\geq1$. 
The model considered in this paper is the interacting particle system
$
(X^{[Y],1},\ldots,X^{[Y],N})
=
(X_t^{[Y],1},\ldots,X_t^{[Y],N})_{t\in[0,1]},
$
defined by Equation~\eqref{eq:sys}, where \(Y\) is uniformly distributed on \(\mathcal{Y}:=\{1,\ldots,K\}\). For each \(k\in\mathcal{Y}\), the functions \(b_k^*:\R^2\to\R\) is measurable, and \(W^1,\ldots,W^N\) are independent standard Brownian motions, independent of \(Y\). Hence, the probability measure $\P$ can be decomposed by $\P=\frac{1}{K}\sum_{k=1}^{K} \P_k$ with $\P_k\coloneqq \P (\,\cdot\, | Y=k)$.
For every $(x,\mu)\in\R\times \mathcal{P}(\R)$, we write $b_k^*[x, \mu]:= \int_\R b_k^*(x,z)\mu(\dd z),\, k\in\mathcal{Y}$.\footnote{For simplicity, we retain the notation $b_k^*$ but distinguish the inputs using square brackets and round brackets: we write $b_k^*[\,\cdot\,,\,\cdot\,]$ when the input is $(x, \mu)\in\R\times {\mathcal{P}(\R)}$ and write $b_k^*(\,\cdot\,,\,\cdot\,)$ when the input consists of two values $x,y\in\R$.}
We define the stochastic process $\xi^{[Y]}$ as the solution to the McKean-Vlasov SDE \eqref{eq:mkv}, which describes the dynamics of a representative particle in the mean-field limit.

\begin{ass}\label{ass:bsup_et_lip_in_x_and_y}
The drift functions $(x,y)\mapsto b_k^*(x,y), \,k\in\mathcal{Y}$ are bounded
Lipschitz continuous, which means, there exists a constant $L>0$ such that for every $ x_1, x_2, y_1, y_2\in \R$ and for every $k\in\mathcal{Y}$,
\begin{equation}\label{eq:lip_in_x_in_b-ne}
\big|b_k^*( x_1, y_1 )-b_k^*( x_2,y_2)\big|\leq L\big[ |x_1-x_2|+|y_1-y_2|\big].
\end{equation}
We denote 
$$b^*_{\rm \max}:= \max_{k \in \cY}\|b_k^*\|_\infty.$$
\end{ass}

Assumption~\ref{ass:bsup_et_lip_in_x_and_y} guarantees the existence and uniqueness of a strong solution to the particle system \eqref{eq:sys}, see e.g. \cite[Lemma 3.2]{lacker2018mean}. 
Moreover, let us recall the definition of $p$-Wasserstein distance $\mathcal{W}_p, \,p\geq1$, see e.g. Equation (5.4) \cite{Carmona-Delarue-I},
\begin{equation}\label{eq:def-wasserstein}
\forall\,\mu, \nu\in\mathcal{P}_p(\R),\quad \mathcal{W}_p(\mu, \nu) := \inf_{\pi \in \Pi(\mu, \nu)} \left(\int_{\R \times \R} |x-y|^p \dd\pi(x,y)\right)^{\frac{1}{p}}, 
\end{equation}
where $\Pi(\mu, \nu)$ in \eqref{eq:def-wasserstein} is the set of probability measures on $\R \times \R$ with marginals $\mu, \nu$. Assumption \ref{ass:bsup_et_lip_in_x_and_y} also implies that the drift function $b_k^*[\cdot,\cdot]$ is Lipschitz continuous, which means, 
\begin{equation}\label{eq-lip-wasserstein}
\forall\,x,y\in\R,\,\forall\,\mu, \nu\in\mathcal{P}_1(\R),\quad \big|\,b_k^*[x,\mu]-b_k^*[y,\nu]\,\big|\leq L\big[\,|x-y|+\mathcal{W}_1(\mu, \nu)\,\big],
\end{equation}
by applying the Kantorovich duality of the Wasserstein distance, see  \cite[Corollary 5.4]{Carmona-Delarue-I}. The above inequality~\eqref{eq-lip-wasserstein} remains valid if we replace $\mathcal{W}_1$ with $\mathcal{W}_p$ for any $p\geq 1$, assuming $\mu, \nu\in\mathcal{P}_p(\R)$, since in this case we have $\mathcal{W}_1(\mu, \nu)\leq \mathcal{W}_p(\mu, \nu)$, see, e.g., Remark 6.6 in \cite{villani2008optimal}.
Consequently, Assumption \ref{ass:bsup_et_lip_in_x_and_y} guarantees the existence and strong uniqueness of the solution $\xi^{[k]}=(\xi^{[k]}_t)_{t\in[0,1]}$ to the McKean-Vlasov Equation~\eqref{eq:mkv} given $Y=k\in\mathcal{Y}$, see e.g.~\cite[Theorem 1.1]{sznitman1991topics}.

In the sequel, $\mathfrak{C}$ denotes a generic positive constant, whose value may change from line to line, and which depends only on the model parameters $K, L, x_0, b_{\max}^*.$ More generally, $C_{\lambda_1}, \cdots C_{\lambda_{p}}$ denotes a positive constant depending on the parameters $\lambda_1, \cdots, \lambda_p$ whose value
may also vary from line to line.

\subsection{Propagation of chaos}
\label{subsec:PropChaos}

We introduce some preliminary results of the McKean-Vlasov equation and the particle system. Let $\|X\|_p= \E[|X|^p]^{1/p}$ denote the $L^p$-norm of a random variable $X$. 

An important property of the McKean-Vlasov equation is known as "propagation of chaos".
We present here a version with respect to the Wasserstein distance and the total variation distance $ d_{\mathrm{TV}} $, adapted to the setting of this paper and provide its proof later in Section \ref{subsec:proof-sec-2}. The proof follows the same idea as in \cite[Theorem 3.3]{lacker2018mean}, \cite{sznitman1991topics} and in \cite[Theorem 2.14]{Lacker2023Hierarchies}. 

\begin{proposition}[Propagation of chaos property] \label{prop:chaos} 
Assume that Assumption~\ref{ass:bsup_et_lip_in_x_and_y} holds. Given \( Y = k \in \mathcal{Y} \), let \( \xi^{[k]} = (\xi_t^{[k]})_{t \in [0,1]} \) be the unique solution to the McKean--Vlasov Equation~\eqref{eq:mkv}, and let \( \mu^{[k]} \) denotes its law, with marginal distributions \( (\mu_t^{[k]})_{t \in [0,1]} \). Let \( (X_t^{[k],1}, \ldots, X_t^{[k],N})_{t \in [0,1]} \) be the particle system defined by~\eqref{eq:sys} and let $\mu_t^{[k], N}=\frac{1}{N}\sum_{n=1}^{N}\delta_{X_t^{[k],n}}$, $t\in[0,1]$. Then,
\begin{enumerate}[label=(\alph*)]
\item For every $p\geq 1$, there exists a constant $C>0$ depending on $p$ and $\mathfrak{C}$ such that 
\begin{equation}
\sup_{t\in[0,1]}\E \left[ \mathcal{W}_p^p\left(\mu_t^{[k], N}, \mu_t^{[k]} \right)\right]\leq C N^{-\frac{1}{2}}.\nonumber
\end{equation}
Moreover, if the Brownian motion \( W \) in \eqref{eq:mkv} is chosen to be the same as the first Brownian motion \( W^{1} \) in \eqref{eq:sys}, then 
\begin{equation}
\E\left[\sup_{t\in[0,1]}\left|  X_t^{[k],1}- \xi_t^{[k]}\right|\right] \leq \mathfrak{C} N^{-\frac{1}{2}}. \nonumber
\end{equation}
\item For a fixed $\ell\in\{1, ..., N\}$, we have
\begin{equation}\label{eq:chaos-TV}
d_{\mathrm{TV}}\big(\mathrm{Law}(X^{[k], 1}, ..., X^{[k], \ell}), (\mu^{[k]})^{\otimes \ell}\big)\leq \mathfrak{C} \frac{\ell}{N},
\end{equation}
where $d_{\rm TV}$ denotes the total variation distance defined  for two probability measures $\mu$ and $\nu$  on a measurable space $E$ by 
\begin{equation}\label{eq:total-variation-distance}
d_{\mathrm{TV}}(\mu, \nu) = \frac{1}{2}\sup \left\{ \int f \, \mathrm{d}\mu - \int f \, \mathrm{d}\nu \:\Bigg|\: f: E \to [-1,1] \:\mathrm{ measurable} \:\right\}.
\end{equation}
\end{enumerate}
\end{proposition}
The proof of this result is postponed in Section~\ref{subsec:proof-sec-2}.

We are now ready to present the supervised classification task.

\section{Classification procedure}\label{sec:classif}
In this section, we describe our classification procedure. The statistical goal is  given in Section~\ref{subsec:Stasetting}. Section~\ref{subsec:meanFielCons} introduces the 
notion ofmean-field consistency while our proposed plug-in classifier is formally described in Section~\ref{subsec:PluginClass}. Finally, the consistency of our plug-in strategy is established in Section~\ref{subsec:excessriskdrift}.

\subsection{Statistical setting}
\label{subsec:Stasetting}

We consider the multiclass classification problem in which the goal is to predict the label $Y\in\cY$ from the trajectory of a particle $Z^{[Y]}$ generated by an interacting particle system \eqref{eq:sys}. A classifier is a measurable function \[g: \mathcal{C}([0,1], \R) \rightarrow \cY,\] and its performance is evaluated through the misclassification risk
\begin{equation}\label{eq:def-misclassification-risk}
\cR_{(Z^{[Y]})}(g) = \mathbb{P}\left(g\left(Z^{[Y]}\right) \neq Y\right).    
\end{equation}

\paragraph*{Discrete observations. }
In practice, the particle trajectories defined by \eqref{eq:sys} are observed only at discrete time points. For a regular grid of length $M+1$ on $[0,1]$ with a time step $\Delta=1/M$, we denote by $\bar{Z}^{[Y]}=({Z}^{[Y]}_{0},{Z}^{[Y]}_{\Delta} \dots, {Z}^{[Y]}_{M\Delta})$ the discrete observation of $Z^{[Y]}$.

\paragraph*{Learning sample.}
For each class $k \in \cY$, we observe a learning sample, denoted by $\mathcal{D}_N^{[k]}$, composed of  two independent particle systems of size $N$, both distributed according to $\mathbb{P}_k:=\mathbb P(\,\cdot\,\mid Y=k)$
\begin{equation}\label{eq:data-discrete}
\mathcal{D}^{[k]}_N:=\left\{\left(\bar{X}^{[k], 1}, \ldots, \bar{X}^{[k], N}\right)\; \text{and}\; \left(\bar{\widetilde{X}}^{[k], 1}, \ldots, \bar{\widetilde{X}}^{[k], N}\right)\right\}. 
\end{equation}
The whole learning sample $ \mathcal{D}_N:= \left\{ \mathcal{D}^{[1]}_N, \ldots, \mathcal{D}^{[K]}_N\right\}$  is the concatenation of the samples $(\mathcal{D}_N^{[k]})_{k \in \cY}$. Hence, we have $2N$ particle trajectories that are available per class. This independence is essential for the nonparametric estimation problem studied later, where the two samples are used separately for estimation purposes; see the comment following Theorem~\ref{theo:riskbhatMkk}.

\paragraph*{Statistical goal.}

Let \(\bar{Z}^{[Y]}\) be a new discrete observation over the time interval \([0,1]\) of a single particle \(Z^{[Y]}\) from an interacting particle system of size \(\Nt\), whose dynamics are described by Equation~\eqref{eq:sys}. We assume that this new observation is independent of the learning sample \(\mathcal{D}_N\).
Our goal is then to build, based on the learning sample $\mathcal{D}_N$, a classifier $\w{g}:\R^{M+1}\rightarrow \cY$ such that $\w{g}(\bar{Z}^{[Y]})$ provides an accurate prediction of its associated label $Y$.
Similar to Equation~\eqref{eq:def-misclassification-risk}, its misclassification risk is defined by $\cR_{(\bar Z^{[Y]})}(\widehat{g}) = \mathbb{P}\left(\widehat{g}\left(\bar Z^{[Y]}\right) \neq Y\right)$.

\subsection{Mean-field consistency}
\label{subsec:meanFielCons}

To evaluate the theoretical performance of $\w{g}$, let us introduce
\begin{equation*}
\mathcal{R}_{\Nt}^{*}:=  \inf_{g \in \cG} \cR_{({Z}^{[Y]})}(g),
\end{equation*}
the optimal risk based on the continuous observation of a particle $Z^{[Y]}$ from an interacting particle system of size $\Nt$. 
 An empirical classifier $\w{g}$ 
 should satisfy
\begin{equation*}\label{eq:mean-field-consis-1}
\E\left[\cR_{(\bar{Z}^{[Y]})}(\w{g})\right] -  \mathcal{R}_{\Nt}^{*} \rightarrow 0, \;\; {\rm as} \;\; N,\Nt \rightarrow +\infty, \;\; {\rm and} \;\; \Delta \rightarrow 0,   
\end{equation*}
where the expectation is taken under the distribution of $\mathcal{D}_N$.
Let us consider $\xi^{[Y]}$ a solution of Equation~\eqref{eq:mkv} obtained in the mean field regime ({\it e.g.} $\Nt \rightarrow +\infty$). Leveraging the result of Proposition~\ref{prop:chaos}, we have that
\begin{equation*}
\mathcal{R}^*_{\Nt} \xrightarrow[\Nt \rightarrow \infty]{} \mathcal{R}^*  
\end{equation*}
where 
\begin{equation}\label{eq:bayes}
\mathcal{R}^*: = \inf_{g \in \mathcal{G}} {\cR_{(\xi^{[Y]})}(g)},
\end{equation}
see Remark~\ref{rem:excessriskNtest}.
In view of this observation, it is natural to introduce the following notion of consistency. We say that $\w{g}$ is {\it mean-field consistent} if
\begin{equation}\label{eq:mean-field-consis-2}
\mathbb{E}\left[\cR_{(\bar{Z}^{[Y]})}(\w{g})\right] \rightarrow \mathcal{R}^*, \;\; {\rm as} \;\; N, \Nt \rightarrow + \infty \;\; {\rm and} \;\; \Delta \rightarrow 0.
\end{equation}
Let us study this convergence using the following decomposition
\begin{equation} \label{eq:decomposition-error}
\E[\cR_{(\bar{Z}^{[Y]})}(\w{g})] -  \mathcal{R}^* =
\left(\E[ \cR_{(\bar{Z}^{[Y]})}(\w{g})]- \E[ \cR_{(\bar{\xi}^{[Y]})}(\w{g})]\right)+ \left(\E[ \cR_{(\bar{\xi}^{[Y]})}(\w{g})]- \cR^* \right).
\end{equation}
The first term on the right-hand side measures the difference in misclassification risk between a particle $\bar{Z}^{[Y]}$ from the finite interacting particle system~\eqref{eq:sys} and its mean-field counterpart $\xi^{[Y]}$ defined by the McKean--Vlasov equation~\eqref{eq:mkv}.
This term is controlled by leveraging the propagation of chaos properties in total variation distance. 
The second term is the excess risk of the classifier in the mean-field model with respect to the optimal risk $\cR^*$. Its analysis involves the time-discretization error, the approximation of the marginal distributions, and the estimation error of the drift functions $(b_k^*)_{k \in \cY}$, which is investigated in Section~\ref{sec:inference}.
This decomposition leads to the following result.
\begin{proposition}\label{prop:excessriskTerm13}
Grant Assumption~\ref{ass:bsup_et_lip_in_x_and_y}.
Let $\bar{Z}^{[Y]}$ be the new observation of a particle coming from an IPS of size $\Nt$ with unknown label $Y$, and let $\bar{\xi}^{[Y]}$ be a discrete observation of a solution of the McKean–Vlasov Equation~\eqref{eq:mkv} with the same label as $\bar{Z}^{[Y]}$. Let $\w{g}$ be a classifier built on the learning sample $\mathcal{D}_N$. Then we have 
\begin{equation*}
\left|\E[\cR_{(\bar{Z}^{[Y]})}(\w{g})]-\cR^*\right|\leq  \mathfrak{C}\Nt^{-1} +  \left(\E[ \cR_{(\bar{\xi}^{[Y]})}(\w{g})]- \cR^* \right).
\end{equation*}
\end{proposition}
The proof of this result is given in Section~\ref{proofs:classif}.
In the following section, the aim is to propose a classifier $\w{g}$ and control the mean-field excess risk $\E[\cR_{(\bar{\xi}^{[Y]})}(\w{g})] -  \mathcal{R}^*$. 

\subsection{Plug-in classifier}
\label{subsec:PluginClass}

In view of the mean-field consistency property, we first characterize the optimal classifier associated with the McKean--Vlasov limit. We define
\begin{equation*}
g^* \in \argmin{g\in \cG}
\cR_{(\xi^{[Y]})}(g),
\end{equation*}
so that $\cR^{*}=\cR_{(\xi^{[Y]})}(g^*)$. The optimal classifier $g^*$ is characterized by
\begin{equation}\label{eq:def-pi}
 g^*\big(\xi^{[Y]}\big) \in \argmax{k \in \mathcal{Y}} \pi_k^*\big(\xi^{[Y]}\big), \;\; {\rm with} \;\; \pi_k^*\big(\xi^{[Y]}\big) =\P\big(Y=k\,\big|\,\xi^{[Y]}\big).
\end{equation}

Proposition \ref{prop:bayes}
provides a closed-form expression of the conditional probabilities $\pi_k^*$. In particular, it generalizes the results of \cite{DDMVC2022} obtained in the framework where the observation is solution of a time homogeneous diffusion process.

\begin{proposition}\label{prop:bayes}
Grant Assumption \ref{ass:bsup_et_lip_in_x_and_y}. 
Let $\xi^{[Y]}=(\xi^{[Y]}_t)_{t\in[0,1]}$ denote the unique strong solution to Equation \eqref{eq:mkv}, and let $(\mu^{[k]}_t)_{t\in[0, 1]}$ denote its marginal distribution if it comes from class $k$.
We define
\begin{equation}
F^*_k(\xi^{[Y]}):=\int_{0}^{1}{{b_k^{*}}\big[\xi^{[Y]}_s, \mu^{[k]}_s\big] \dd \xi^{[Y]}_s}-\frac{1}{2}\int_{0}^{1}b^{*2}_{k}\big[\xi^{[Y]}_s, \mu^{[k]}_s\big]\dd s, \quad k \in \cY.
\end{equation}
Then, for each $k \in \cY$, the conditional probability $\pi^{*}_k$ defined in~\eqref{eq:def-pi} satisfies 
\begin{equation}\label{eq:pi}
\pi^*_k(\xi^{[Y]})= {\rm softmax}\left(F^*(\xi^{[Y]})\right)_k, \;\; \P-\text{ a.s.},
\end{equation}
where $F^*(\xi^{[Y]}) = \left(F_1^*(\xi^{[Y]}), \ldots, F_K^*(\xi^{[Y]})\right)$ and, for $z=(z_1,\ldots,z_K)\in\mathbb{R}^K$, the $k$-th component of the softmax function is defined by 
$
\mathrm{softmax}(z)_k
=
\frac{e^{z_k}}{\sum_{j=1}^K e^{z_j}}.
$
\end{proposition}
Consequently, the optimal classifier can equivalently be written as
\begin{equation}\label{eq:gstarfinal}
g^*\big(\xi^{[Y]}\big) \in \argmax{k \in \cY} F_k^*\big(\xi^{[Y]}\big).
\end{equation}

\paragraph{Plug-in classifier.}
In view of the expression of $g^*$ given by Equation~\eqref{eq:gstarfinal},
we consider a plug-in strategy. First, for each $k \in \cY$, based on $\mathcal{D}_{N}^{[k]}$, we construct an estimator $\w{b}_k$ of $b_k^*$. An explicit construction of this estimator is provided in Section~\ref{sec:inference}. Second, for a new observation $\bar{Z}^{[Y]}$, we approximate $F_k^*(\bar{Z}^{[Y]})$ by replacing $b_k^*$ with $\w{b}_k$, approximating the marginal distribution \(\mu^{[k]}\) with its empirical counterpart constructed from the training sample, and estimating continuous-time integrals with their discrete approximations. This leads to the estimated score \(\widehat F_k(\bar Z^{[Y]})\), defined as follows,
\begin{equation}\label{eq:Fhathat}
{\w{F}}_k(\bar{Z}^{[Y]}):=\frac{1}{N}\sum_{n=1}^N \sum_{m=0}^{M-1} \w{b}_k\left({Z}^{[Y]}_{m\Delta},\wt{X}_{m\Delta}^{[k],n}\right)\left(Z^{[Y]}_{(m+1)\Delta}-Z^{[Y]}_{m\Delta}\right)-\frac{\Delta}{2}\sum_{m=0}^{M-1}  \left(\frac{1}{N}\sum_{n=1}^N \w{b}_k(Z^{[Y]}_{m\Delta}, \wt{X}^{[k],n}_{m\Delta})\right)^2.
\end{equation}
The corresponding conditional probabilities can then be estimated by
\begin{equation}\label{eq:def-pi-hat}
\widehat\pi_k(\bar Z^{[Y]}) = \operatorname{softmax}\big(\widehat F(\bar Z^{[Y]})\big)_k. 
\end{equation}
Since the softmax function preserves the ordering of its coordinates, we define the plug-in classifier by
\begin{equation}\label{eq:def-g-hat} 
\widehat g(\bar Z^{[Y]}) \in \argmax{k\in\mathcal Y} \widehat F_k(\bar Z^{[Y]}). 
\end{equation}
Figure \ref{fig:classification_proc}  summarizes the procedure.

\begin{figure}[ht]
    \centering   
    \begin{tikzpicture}[
        >=Latex,
        font=\sffamily\small, 
        box/.style={
            draw=black, 
            thick, 
            rounded corners=6pt, 
            align=center, 
            text width=4.2cm, 
            inner sep=6pt,
            fill=white
        },
        theorybox/.style={box, draw=black, text=black!70, dashed}, 
        modelbox/.style={box, draw=black, very thick}, 
        phasetitle/.style={font=\bfseries, anchor=south west, inner xsep=0pt},
    ]

    \node[modelbox, xshift=-0.8cm] (ps) {Training Data\\({IPSs} of size $N$ from $K$ classes \eqref{eq:sys})};
    \node[theorybox, right=4cm of ps] (mkv) {McKean--Vlasov\\limit equation \eqref{eq:mkv}};
    
    \node[modelbox, below=1.2cm of ps] (plugin) {Plug-in classifier $\widehat{g}$\\\eqref{eq:Fhathat}-\eqref{eq:def-g-hat}};
    \node[theorybox, right=4cm of plugin] (opt) {Optimal classifier $g^*$\\{\eqref{eq:gstarfinal}}};

    \draw[-Latex, thick] (ps.east) -- node[above, align=center, font=\footnotesize] {Propagation\\of chaos} node[below, font=\footnotesize] {$N \rightarrow \infty$} (mkv.west);
    \draw[-Latex, thick] (ps.south) -- node[left, font=\footnotesize] {Construct} (plugin.north);
    \draw[-Latex, thick, draw=black!50] (mkv.south) -- node[right, align=left, font=\footnotesize, text=black] {{Characterized} in\\Prop. \ref{prop:bayes}} (opt.north);
    
    \draw[Latex-Latex, thick, dashed, draw=black!50] (plugin.east) -- 
        node[above, align=center, text=black!70, font=\footnotesize] {Mean-field excess risk \\$\E[\cR_{(\bar{\xi}^{[Y]})}(\w{g})] -  \mathcal{R}^*$} (opt.west);
    
    \node[fit=(ps) (mkv) (plugin) (opt)] (phase1_nodes) {};
    \node[phasetitle] at ($(phase1_nodes.north west) + (0, 0.2)$) (p1title) {Phase I: Training \& Theoretical Mean-Field Analysis};
    \node[draw=black!, thick, rounded corners=8pt, fit=(phase1_nodes) (p1title), inner sep=10pt] (frame1) {};

    \node[modelbox, below=2cm of frame1.south, text width=3.5cm, xshift=0.5cm] (model_inf) {\textbf{Model}\\Plug-in classifier $\widehat{g}$};
    
    \node[modelbox, draw=black, text width=4.5cm, left=0.8cm of model_inf] (input) {\textbf{Input}\\{One particle trajectory $\bar{Z}^{[Y]}$}\\from a new IPS of size $\Nt$};
    \node[modelbox, draw=black, text width=3.5cm, right=0.8cm of model_inf] (output) {\textbf{Output}\\Predicted label {$\widehat{g}(\bar{Z}^{[Y]})$}};

    \draw[-Latex, thick, draw=black] (input.east) -- (model_inf.west);
    \draw[-Latex, thick, draw=black] (model_inf.east) -- (output.west);

    \node[fit=(input) (model_inf) (output)] (phase2_nodes) {};
    \node[phasetitle, text=black] at ($(phase2_nodes.north west) + (0, 0.2)$) (p2title) {Phase II: Prediction};
    \node[draw=black, thick, rounded corners=8pt, fit=(phase2_nodes) (p2title), inner sep=10pt] (frame2) {};
    \end{tikzpicture}
    \caption{Classification procedure: data-driven construction and its mean-field benchmark.}
    \label{fig:classification_proc}
\end{figure}
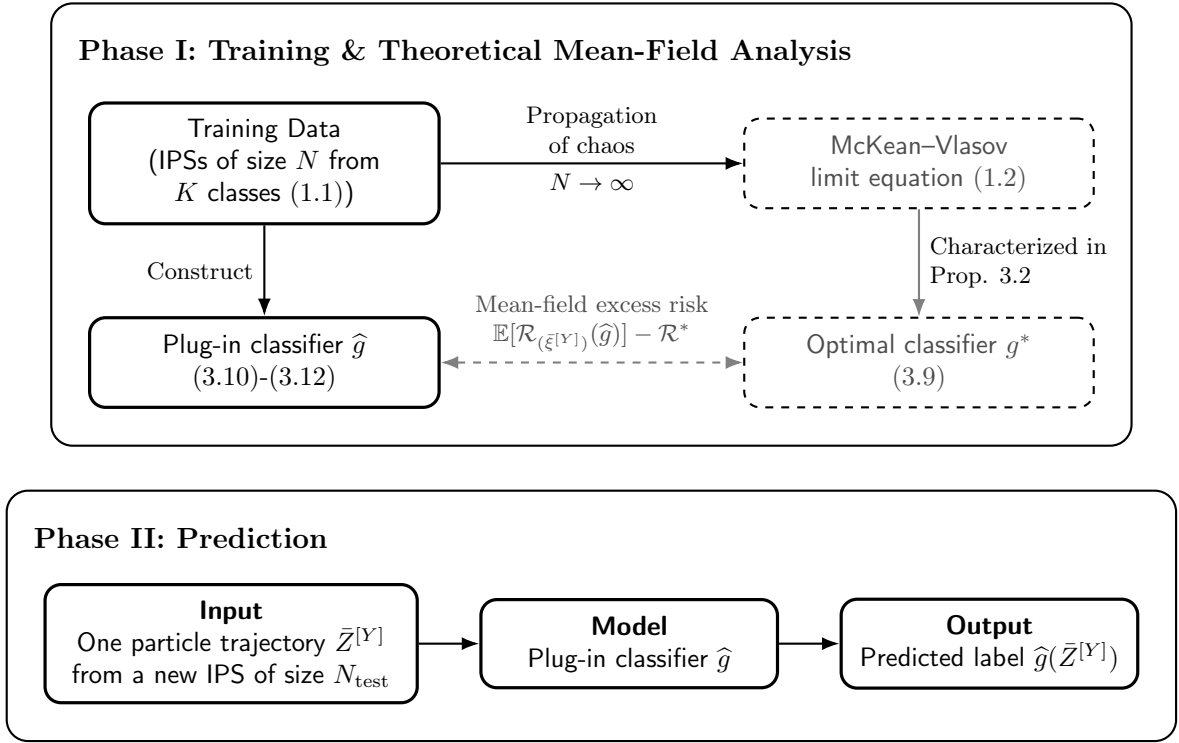


\subsection{Mean-field excess      risk}\label{subsec:excessriskdrift}

Proposition~\ref{prop:excessriskTerm13} highlights that the study of mean-field consistency relies on controlling the mean-field excess risk 
\begin{equation*}
\E\left[\mathcal{R}_{(\bar{\xi}^{[Y]})}(\w{g}) - \mathcal{R}^*\right]
\end{equation*}
where $\bar{\xi}^{[Y]}=(\xi^{[Y]}_{m\Delta})_{m=0, \ldots, M}$ is a discrete-time observation of a solution to Equation~\eqref{eq:mkv}.
The objective of this section is to provide a bound on the mean-field excess risk. To this end, for each class $k \in \mathcal{Y}$, we introduce the following semi-norm
defined for a measurable function $h: \mathbb{R}^2 \rightarrow \mathbb{R}$ as 
\begin{equation*}
\left\|h \right\|^2_{M,k}:=\mathbb{E}\left[\frac{1}{M} \sum_{m=0}^{M-1}
\left(\frac{1}{N}\sum_{n'=1}^N h(\xi^{[Y]}_{m\Delta}, \wtX_{m\Delta}^{[k],n'})\right)^2\right] =  \dfrac{1}{K}\sum_{k'=1}^K \|h\|^2_{M,k,k'},    
\end{equation*}
where for each $k'\in\cY$, we define
\begin{equation}\label{eq:normMkk}
\|h\|_{M,k,k'}^2:=\E\left[ 
\frac{1}{M} \sum_{m=0}^{M-1}
\left(\frac{1}{N}\sum_{n'=1}^N h(\xi^{[k']}_{m\Delta}, \wtX_{m\Delta}^{[k],n'})\right)^2
\right].
\end{equation}

In order to establish a bound on the excess risk of the classifier, we first need to ensure that the estimator $\widehat{b}_k$ is bounded. We therefore consider a thresholded version of $\widehat{b}_k$, which we continue to denote by $\widehat{b}_k$,
\begin{equation}\label{eq:bhattronc}
\w{b}_k(\cdot,\cdot):=\w{b}_k(\cdot,\cdot) \one_{\{|\w{b}_k(\cdot,\cdot)| \leq  \log^{1/2}(N)\}} + \text{sgn}\left(\w{b}_k\right)  \log^{1/2}(N) \one_{\{|\w{b}_k(\cdot,\cdot)| > \log^{1/2}(N)\}}. 
\end{equation}

Note that, since $\Vert b_k^{*}\Vert_{\infty}\leq b_{\max}^*$, we have $\left\|b_k^*\right\|_{\infty} \leq \sqrt{\log(N)}$ for $N$ sufficiently large. Hence, the truncation only acts on unusually large values of the estimator and does not affect its asymptotic properties. 
We can now state the resulting bound on the excess risk of $\w{g}$.
\begin{theo}\label{theo:excessrisk}
Grant Assumption~\ref{ass:bsup_et_lip_in_x_and_y}. Let $\widehat g$ be the classifier defined in~\eqref{eq:def-g-hat}, with
$\widehat b_k$, $k\in\mathcal Y$, given by~\eqref{eq:bhattronc}.
Then, for $N$ sufficiently large,
$$
\E\left[\mathcal{R}_{(\bar{\xi}^{[Y]})}(\w{g}) - \mathcal{R}^*\right] \leq 
\mathfrak{C} \left(\sqrt{\Delta} + N^{-1/2}+ \log^{1/2}(N) \sum_{k=1}^K 
\E\left[ \|\w{b}_k -{b}_k^*\|_{M,k}\right]
\right).
$$
\end{theo}
The above result shows that the excess risk of the classifier $\w{g}$ under the mean-field regime is bounded by a sum of three terms. The first term reflects the time discretization error. The second term results from the empirical approximation of the marginal distribution. 
The last term corresponds to the estimation error of the drift functions. Note that the logarithmic factor is due to the bound imposed on $\left\|\w{b}_k\right\|_{\infty}$.
This result highlights that if the estimators $\w{b}_k$, $k \in \mathcal{Y}$, satisfy
$\sqrt{\log(N)} \,\E\big[ \|\w{b}_k -{b}_k^*\|_{M,k}\big]\rightarrow 0$,
then the mean-field excess goes to $0$ as $M,N \rightarrow +\infty$. In view of Proposition~\ref{prop:excessriskTerm13}, this implies that $\w{g}$ is mean-field consistent.
Hence, deriving an explicit convergence rate requires controlling the estimation error of the drift functions. In the next section, we introduce explicit estimators of the drift functions and establish rates of convergence for the resulting classifier $\w{g}$.

\begin{remark}
It is interesting to note that a similar bound is established in Theorem~1 of~\cite{DDMVC2022} for classical SDEs driven by standard Brownian motion, without interactions. In that setting, the term of order $N^{-1/2}$ comes from the estimation of the unknown distribution of the label $Y$
whereas in the present work this distribution is assumed to be known.
By contrast, in Theorem~\ref{theo:excessrisk}, the term of order $N^{-1/2}$ stems from the empirical approximation of the mean-field distribution.
\end{remark}

\section{Estimation of the drift functions}
\label{sec:inference}

In this section, we present the construction of the estimators $\w{b}_k$
of the bivariate drift functions $b_k^*$, $k \in \mathcal{Y}$. More specifically, we adapt the procedure developed in~\cite{DDM2019} in the new setting of IPS. Particularly, the procedure relies on the minimization of a $L_2$ contrast over a constraint space of functions. 
The estimation procedure is formally described in Section~\ref{subsec:NonparamEstimator} and its theoretical properties are stated in Section~\ref{subsec:TheoPropertiesEstimator}.

\subsection{Nonparametric drift estimator}
\label{subsec:NonparamEstimator}

We first introduce the considered space of functions.

\paragraph*{Constraint space of approximation.}

Let $A > 0$. We consider a space function generated by the $B$-spline basis on 
$[-A,A]^2$, which is an extension of the univariate $B$-spline basis.
We refer to~\cite[Chapters 14 and 15]{Gyorfi-2002a} for more details.

Let  $D > 0$ be the dimension parameter, and $R \geq 1$ the degree. The sequence of knot vectors \({\bf u} = (u_{-R}, \ldots, u_{R+D})\) is defined by
$
u_d = -A + 2d \frac{A}{D}, ~ d = 0, \ldots, D, 
$
with \( u_{-R} = \ldots = u_{-1} = -A \) and \( u_{D+1} = \ldots = u_{R+D} = A \). The tensor product spline basis \((B_{j_1, j_2})_{-R \leq j_1, j_2 \leq D-1}\) is then defined by
\[
B_{j_1, j_2}(x, y) = B_{j_1}(x) B_{j_2}(y),
\]
where \( B_{j_1} \) and \( B_{j_2} \), for \( -R \leq j_1, j_2 \leq D-1 \), are the univariate \( B \)-spline basis functions of order \( R \), defined by the knot vector \(\bf u\) (see, e.g., \cite[Definition 14.2]{Gyorfi-2002a}). 
The basis \((B_{j_1,j_2})_{-R \leq j_1, j_2 \leq D-1}\) inherits properties similar to those of the univariate \( B \)-spline basis. In particular, they are positive, null under 
$[u_{j_1}, u_{j_1+R+1}) \times [u_{j_2}, u_{j_2+R+1}) $ and for all $(x,y)\in \R^2$,
\begin{equation*}
\sum_{j_1,j_2=-R}^{D-1} B_{j_1,j_2}(x, y) = 1.
\end{equation*}
Moreover, for each $h \in \mathcal{S}_{D,R}$, $h(x,y) = 0$, for all $(x,y) \notin [-A,A]^2$, and $h \in \mathcal{C}^{R-1}$ on $(-A,A)^2$. 
The corresponding spline approximation space is then defined as
\begin{equation}\label{eq:SDR}
\mathcal{S}_{D,R} := \left\{ \sum_{j_1,j_2=-R}^{D-1} a_{(j_1,j_2)} B_{j_1,j_2} : {\bf a} = (a_{(j_1,j_2)}) \in \mathbb{R}^{(D+R)^2}, \; \|{\bf a}\|_2^2 \leq (D+R)^2 A^2 \log(N) \right\}.
\end{equation}
M
Hence the space $\mathcal{S}_{D,R}$ consists of linear combination of $B-$splines, with a imposed on the $\ell^2$-norm of the coefficients. This constraint ensures that the coefficients are not too large and yields some important properties.  
In particular, it is well known that, without this constraint, the definition of the estimator will depend on the invertibility of the Gram matrix (see e.g. \cite{comteregression}). However, this is no longer the case here. As detailed below, this constraint guarantees the invertibility of the Gram matrix used to define the estimator.
In particular, as in~\cite[Proposition 3]{DDMVC2022}, the following result gives the approximation error of the set $\mathcal{S}_{D,R}$ for the drift functions $b_k^*, k \in \mathcal{Y}$. 
\begin{proposition}\label{prop:splineapprox}
Grant Assumption~\ref{ass:bsup_et_lip_in_x_and_y}. If $A\geq 1$, for $R \geq 1$, and each $k \in  \mathcal{Y},$ there exists $\mathfrak{b}_k \in \mathcal{S}_{D,R}$ such that,
\begin{equation}\label{eq:biais}
\left|\mathfrak{b}_k(x,y) - b^*_k(x,y)\right| \leq \dfrac{4A(R+1)}{D}L,
\;\; {\forall (x,y) \in (-A,A)^2},
\end{equation}
where $L$ is the Lipschitz constant of $b_k^*$ given in Assumption \ref{ass:bsup_et_lip_in_x_and_y}.
\end{proposition}
The proof of this result is given in Section~\ref{proof:propsplineapprox}.
In particular we can see that, as expected, this error decreases as $D$ increases. 

\paragraph{Ridge estimator of the drift functions.}

Let $k \in \mathcal{Y}$. To estimate the bivariate function $b_k^*$, we consider the learning sample $\mathcal{D}_N^{[k]}$ defined in Equation~\eqref{eq:data-discrete} composed of two independent IPS
$\overline{X}^{[k],n}_{1\leq n \leq N}$, and $\overline{\widetilde{X}}^{[k],n}_{1 \leq n  \leq N}$.
The choice of the contrast relies on the following observation.
Conditional on the event $\{Y=k\}$, for each $1 \leq n \leq N$, the increments
\begin{equation*}
U_{m\Delta}^{[k],n} := \frac{X_{(m+1)\Delta}^{[k],n} - X_{m\Delta}^{[k],n}}{\Delta}, \quad 0 \leq m \leq M - 1,
\end{equation*}
can be decomposed in $\displaystyle
\frac{1}{N} \sum_{n' = 1}^N b_k^*(X_{m\Delta}^{[k],n}, {X}_{m\Delta}^{[k],n'}) 
$ plus additional terms to control.
Hence, our estimation problem can be viewed as a regression problem {\it w.r.t.} the variables $(U^{[k,n]}_{m\Delta})_{0\leq m \leq M-1}$, $n=1, \ldots,N$.
More precisely, for each $k \in \mathcal{Y}$, 
the \( B \)-spline-based estimator $\wt{b}_k$ of $b^*_k$ is defined as a minimizer of the following empirical least squares contrast, over the constrained set $\mathcal{S}_{D,R}$:
\begin{equation}\label{eq:btilde}  
\wt{b}_k = \argmin{h \in \mathcal{S}_{D,R}}{\gamma^{[k]}_{N,M}(h)}, \quad 
\gamma^{[k]}_{N,M}(h):= \dfrac{1}{NM}\sum_{m=0}^{M-1} \sum_{n=1}^N \left(U_{m\Delta}^{[k],n} - \frac{1}{N} \sum_{n'=1}^N h(X_{m\Delta}^{[k],n}, {\wt{X}_{m\Delta}^{[k],n^{'}}})\right)^2.
\end{equation}
Moreover, according to the definition of the space of approximation $\mathcal{S}_{D,R}$ given in \eqref{eq:SDR} we can  also write
\begin{equation}\label{eq:bspline-estim}
   \wt{b}_k(x, y) := \sum_{j_1 = -R}^{D-1} \sum_{j_2 = -R}^{D-1} \w{a}_{(j_1, j_2)}^{[k]} B_{j_1}(x) B_{j_2}(y), 
\end{equation}
where the coefficients of the decomposition are given by the matrix \( \widehat{\mathbf{A}}^{[k]} = (\widehat{a}_{(j_1, j_2)}^{[k]})_{-R \leq j_1, j_2 \leq D-1} \). This is the solution of the constrained optimization problem:
\begin{equation}\label{eq:bspline-coeff}
    \w{\mathbf{A}}^{[k]}=\!\!\argmin{\substack{\mathbf{A}^{[k]}\in \R^{D+R}\otimes\mathbb{R}^{D+R}\:\text{s.t.}\\\Vert \mathbf{A}^{[k]}\Vert_{F}^2\leq (D+R)^2 A^2 \log(N)}} \!\!\frac{1}{NM}\sum_{m=0}^{M-1}\sum_{n=1}^{N}\Big(U_{m\Delta}^{[k],n}-\frac{1}{N}\sum_{n'=1}^{N}\sum_{j_1,j_2=-R}^{D-1}a^{[k]}_{(j_1,j_2)}B_{j_1}\big(X_{m\Delta}^{[k], n}\big)B_{j_2}\big(\wt{X}_{m\Delta}^{[k], n'}\big)\Big)^2
\end{equation}
Rewriting the increment vector \( \mathbf{U} \) as:
\[
\mathbf{U} = \big(U_{0\Delta}^{[k], 1}, \ldots, U_{0\Delta}^{[k], N}, U_{1\Delta}^{[k], 1}, \ldots, U_{1\Delta}^{[k], N}, \ldots, U_{(M-1)\Delta}^{[k], 1}, \ldots, U_{(M-1)\Delta}^{[k], N}\big) \in \mathbb{R}^{MN},
\]
and defining the matrix \( \mathbf{B} :=\mathbf{B}_D= \big( \mathbf{B}_{(j_1, j_2), (m, n)} \big)_{-R \leq j_1, j_2 \leq D-1, \; 0 \leq m \leq M-1, \; 1 \leq n \leq N} \) of size \( MN \times (D+R)^2 \) as:
\begin{equation}\label{eq:B}
\mathbf{B}_{(j_1, j_2), (m, n)} = B_{j_1}\big(X_{m\Delta}^{[k], n}\big) \left( \frac{1}{N} \sum_{n'=1}^N B_{j_2}\big(\wt{X}_{m\Delta}^{[k], n'}\big) \right),
\end{equation}
rewritten as a vector of dimension \( (D+R)^2 \):
\[
\w{\mathbf{a}} = \big(\w{a}_{(j_1, j_2)}\big)_{-R \leq j_1, j_2 \leq D-1} = \big(\w{a}_{(1, 1)}, \w{a}_{(1, 2)}, \ldots, \w{a}_{(D+R, D+R)}\big) \in \R^{(D+R)^2},
\]
is the solution of the ridge optimization problem:
\[
\w{\mathbf{a}} = \argmin{\|\mathbf{a} \|_2^2 \leq (D+R)^2 A^2 \log(N)} \big\| \mathbf{U} - \mathbf{B} \mathbf{a} \big\|_2^2.
\]
Finally, rather than $\wt{b}$ give in Equation~\eqref{eq:btilde}, we consider a thresholded version also denoted $\w{b}$ defined as follows
\begin{equation}\label{eq:estimatorbhatfinal}
\w{b}_k(\cdot):=\wt{b}(\cdot) \one_{|\wt{b}(\cdot)| \leq  \log^{1/2}(N)} + \text{sgn}(\wt{b})  \log^{1/2}(N) \one_{\{|\wt{b}(\cdot)| > \log^{1/2}(N)\}},
\end{equation}
analogously as the formula \eqref{eq:bhattronc}.
Next, we focus on analyzing the risk of the estimator.

\subsection{Convergence rate of the drift estimator}
\label{subsec:TheoPropertiesEstimator}

This section is devoted to the study of the rates of convergence of the estimators $\w{b}_k, k \in \cY$.
We assume that the regularity parameter $R$ is fixed. However, we assume that the dimension parameter depends on $N$ and satisfies $D:=D_N \rightarrow +\infty$ as the sample size $N$ tends to infinity. Furthermore, we consider $A:=A_N= \log(N)$. Hence, for each $k \in \mathcal{Y}$, the estimators $\w{b}_k$ are non negative on $[-A_N, A_N]$.
The choice of $A_N = \log(N)$ is motivated by the followings.
First, from Proposition~\ref{prop:splineapprox}, on $[-A_N,A_N]$,
we have that there exists $\mathfrak{b}_k \in \mathcal{S}_{D_N, R}$ such that
\begin{equation*}
\left|\mathfrak{b}_k(x,y)-b_k^*(x,y) \right| \leq \mathfrak{C} \dfrac{\log(N)}{D_N} \rightarrow 0, \;\; N \rightarrow +\infty,    
\end{equation*}
provided that $\log(N) =o(D_N)$. Hence, the set $\mathcal{S}_{D_N,R}$ provides a good approximation of the drift functions ${b}^*_k$ on $[-A_N,A_N]$. Besides, with $A_N = \log(N)$, we have that (see Lemma~\ref{lem:ineq-concentration-Xt}, in Section~\ref{proofs:technical})
\begin{equation*}
 \sup_{t \in (0,1)} \mathbb{P}\left(\left|\xi_t^{[k]}\right|\geq A_N\right)   \leq \dfrac{\mathfrak{C}}{N}. 
\end{equation*}
Therefore, leveraging this result, one can show that the global bias error on $\mathbb{R}^2$, which is related to the approximation of $b_k^*$ by an element of $\mathcal{S}_{D_N,R}$, is a $o(1)$ 
provided that $\log(N) =o(D_N)$.
In particular, we establish the following result.
\begin{theo}\label{theo:riskbhatMkk}
Grant Assumption \ref{ass:bsup_et_lip_in_x_and_y}.
For each $k \in \mathcal{Y}$, the estimator $\w{b}_k$ of $b_k^*$ given in Equation~\eqref{eq:estimatorbhatfinal} satisfies
\begin{equation*}
\E\left[ \|\w{b}_k-{b}^*_k\|^2_{M,k,k}\right] 
\leq 
\mathfrak{C} \left( \frac{\log^2(N)}{D_N^2} + \log^{3/2}(N)\sqrt{\frac{(D_N+R)^2}{N}}+ \Delta  \right).
\end{equation*}
\end{theo}
The proof of this result is given in Section~\ref{proof:secdrift}. 
The bound in Theorem~\ref{theo:riskbhatMkk} decomposes into three terms: a bias term, an estimation error term, and a discretization error term.
The proof of this result relies on the following steps. For each $k \in \mathcal{Y}$, we first obtain a control on the empirical norm that stands for the empirical counterpart of the $\left\|\cdot\right\|_{M,k,k}$ based on the learning sample $\mathcal{D}_N^{[k]}$. 
Then, a concentration argument is used to extend the control on the empirical norm to the semi-norm $\left\|\cdot\right\|_{M,k,k}$.
This step relies on the Bernstein inequality established 
in~\cite{hoffmann2021}. 
More precisely, we consider a second system and adopt a two-sample strategy to apply standard concentration results developed in~\cite{hoffmann2021}.
If only one system is observed, the definition of the norm in Equation~\eqref{eq:normMkk} would change, with $\tilde{X}$ replaced by $X$ in the second coordinate to estimate the empirical measure. Controlling this new norm would then be challenging, since the nonparametric estimator would depend on the same sample.
An alternative strategy might be possible, for instance by following \cite{belopodol} and using concentration techniques based on $U$-statistics. However, this more technical approach would likely require additional assumptions on the initial condition of the model.


Now, let us leverage the result obtained in Theorem~\ref{theo:riskbhatMkk} to get a rate of convergence with respect to the integrated norm defined for each $k, k'$ and some measurable function $h$ as
\begin{equation*}
\Vert h \Vert_{M, k, k', \star}^2\coloneqq \E\left[ \frac{1}{M}\sum_{m=0}^{M-1} \left(\int_{\R} h(\xi_{m\Delta}^{[k']}, y)\mu_{m\Delta}^{[k]}(dy)\right)^2\right]. 
\end{equation*}
Let us emphasize that in this expression the measure $\mu$ is associated with the limit equation~\eqref{eq:mkv} for class $k$, whereas the first argument $\xi$ comes from a limit particle but from class $k'$ that may be different from $k$.
To this extent, we establish first the following result that links the integrated norm to the norm $\left\|\cdot \right\|_{M,k,k'}$ given in Equation~\eqref{eq:normMkk}.
\begin{lemma}
\label{lem:boundIntegrated}
Grant Assumption~\ref{ass:bsup_et_lip_in_x_and_y}.
For all $k,k'\in\mathcal{Y}$, it holds that for $D_N \leq \sqrt{N}$, 
\begin{equation*}
\left|\E\left[\| \w{b}_k-b_k^* \|_{M, k, k', \star}^2\right] - \E\left[\| \w{b}_k-b_k^* \|_{M, k, k'}^2\right]\right|
\leq \mathfrak{C} \dfrac{\log(N)D_N}{\sqrt{N}}.
\end{equation*}
\end{lemma}
Combining Lemma~\ref{lem:boundIntegrated} for $k'=k$, with Theorem~\ref{theo:riskbhatMkk} yields the convergence rate of our drift estimator in term of integrated norm, which does not depend on the learning sample. 
\begin{corollary}
\label{coro:RateIntegratedNorm}
Grant Assumption~\ref{ass:bsup_et_lip_in_x_and_y}.
Choosing $\Delta =  O(1/N)$, and $D_N  \propto N^{1/6}$, we then obtain, up to some logarithmic factors, that for each $k \in \mathcal{Y}$
\begin{equation*}
\E\left[ \|\w{b}_k-{b}^*_k\|_{M,k,k, \star}\right] \propto  N^{-1/6}.   
\end{equation*}
\end{corollary}
Finally, it is important to note that result of Theorem~\ref{theo:riskbhatMkk} and Corollary~\ref{coro:RateIntegratedNorm}  holds for the estimation of the function $b_k^*$ on $\mathbb{R}^2$ and is not only restricted to a compact set.  

Let us remind the reader that, at this point, the only assumption of regularity made on the drift function is that it is Lipschitz in the two coordinates. It is well known that in the usual nonparametric regression framework~\cite{tsybakov2008nonparametric,Gyorfi-2002a}, the optimal rate of convergence for bivariate function is of order $N^{-1/4}$. In our specific setting, we obtain a rate of convergence of order $N^{-1/6}$. This difference is due to the order of the variance term and illustrates the difficulty of the problem of estimating interaction function of an IPS. A specific comparison with the existing literature is provided in Remark~\ref{rem:mark}.

\begin{remark}[Comparison]\label{rem:mark}
To the best of our knowledge, the most recent nonparametric estimator related to the present framework is the one proposed in \cite{belopodol}. In this work, the authors consider an interaction function of the form $b(X_t^{n,N},X_t^{n',N})=\phi(X_t^{n,N}-X_t^{n',N})$ with $\phi$ bounded and propose an estimator of $\phi$. Similarly to~\cite{denis2021ridge}, their approach relies on the minimization of a least-squares contrast over a compact approximation space. The authors assume that the initial distribution is Gaussian and consider the continuous observations framework. In particular, they derive risk bounds with respect to the norm $\|\cdot\|_\star$, which can be viewed as a continuous counterpart of the norm $\left\|\cdot\right\|_{M,k,k, \star}$, that depends on the unknown measure does not readily translate into an $L_2$-norm bound.  In this setting, the variance term has the same order as in Theorem~\ref{theo:riskbhatMkk}.

Another point of comparison is the result obtained in~\cite{DDMVC2022} in the framework of univariate homogeneous diffusions. In this work, the authors propose a projection estimator of the drift function based on $B$-spline approximations. They establish a risk bound in which the variance term appears without the square root. However, this improvement is mainly due to the fact that the risk is measured with respect to a norm that is equivalent to the $L_2$-norm, which is not the case in our framework.
\end{remark}

\section{Convergence rates of plug-in classifiers}\label{sec:classiffinal}

In this section we derive the convergence rate for the proposed classifier $\w{g}$ based on the ridge estimators of the drift functions $b_k^*, k \in \mathcal{Y}$. 
In view of Theorem~\ref{theo:excessrisk}, the control of the excess risk relies on the control of the error $\E\left[ \|\w{{b}}_k -{b}_k^*\|_{M,k}\right]$, for each $k \in \cY$. However, it requires that for each $k \in \cY$, the estimator $\w{b}_k$ is consistent under the distribution $\mathbb{P}$ of $\xi^{[Y]}$.
Since $\w{b}_k$ are estimators of $b_k^*$ built from $\mathcal{D}_N^{[k]}$, we can only expect at first the consistency of $\w{b}_k$ {\it w.r.t.} $\mathbb{P}_k$, the distribution of $\xi^{[Y]}$ conditional on $\left\{Y=k\right\}$.
Hence, for each $k \in \cY$, we have to show that consistency {\it w.r.t.} 
$\mathbb{P}_k$ implies consistency {\it w.r.t.} $\mathbb{P}$. This is the purpose of the next result that relies on a martingale argument.

\begin{proposition}\label{prop:changementnorme}
Under Assumption~\ref{ass:bsup_et_lip_in_x_and_y}, 
for every couple $(k,k')\in \cY^2$ with $k \neq k'$, we have
\begin{equation*}
\E\left[\|\w{{b}}_k -{b}_k^*\|^2_{M,k,k'}\right] \leq C\,\exp(\sqrt{2c\log(N)})\left(\E \left[\left\Vert \w{b}_k-{b}_k^* \right\Vert^2_{M, k, k}\right]
+\frac{\log(N)}{\sqrt{N}} \right),
\end{equation*} 
where $C, c > 0$ and $C$ depends on $b^*_{\max}$.
\end{proposition} 
Since for each $\varepsilon > 0$, and $c> 0$, we have that
$\exp(\sqrt{2c\log(N)})= o(N^\varepsilon)$,
Proposition~\ref{prop:changementnorme} implies that for each $k \in \mathcal{Y}$, and $\varepsilon > 0$,
\begin{multline*}
\mathbb{E}\left[\left\|\w{b}_k-b_k^*\right\|^2_{M,k}\right]
= \dfrac{1}{K}\sum_{k' = 1}^K\mathbb{E}\left[\left\|\w{b}_k-b_k^*\right\|^2_{M,k,k'}\right] \leq \\ \dfrac{1}{K} \left(\mathbb{E}\left[\left\|\w{b}_k-b_k^*\right\|^2_{M,k,k}\right] + C(K-1)N^{\varepsilon}\left(\mathbb{E}\left[\left\|\w{b}_k-b_k^*\right\|^2_{M,k,k}\right] + \log(N) N^{-1/2} \right)\right) 
\end{multline*}
thus 
\begin{equation}\label{eq:normeMKcontrol}
\mathbb{E}\left[\left\|\w{b}_k-b_k^*\right\|^2_{M,k}\right]
 \leq 
C N^{\varepsilon}\left(\mathbb{E}\left[\left\|\w{b}_k-b_k^*\right\|^2_{M,k,k}\right] + \log(N) N^{-1/2}\right).
\end{equation}
Therefore, combining Equation~\eqref{eq:normeMKcontrol} with Proposition~\ref{prop:excessriskTerm13}, Theorem~\ref{theo:excessrisk}, and Theorem~\ref{theo:riskbhatMkk}, we establish the following result.
\begin{corollary}
\label{coro:excessRiskModel}
Let $D_N \propto N^{1/6}$, and $\Delta= O(1/N)$. 
Under Assumption~\ref{ass:bsup_et_lip_in_x_and_y}, for all $\varepsilon  > 0$, and $N$ large enough, the following holds
\begin{equation*}
\E\left[\left| \mathcal{R}_{(\bar{Z}^{[Y]})}(\w{g}) -  \mathcal{R}^* \right|\right]
\leq \mathfrak{C} \max\left(N^{-1/6+\varepsilon}, \Nt^{-1}\right).
\end{equation*}
\end{corollary}

The above result highlights that the convergence rate is determined by the interplay between the propagation of chaos error and the estimation error of the drift estimators, yielding an overall rate slower than $N^{-1/6}$.
In particular, if $\Nt  > N^{1/6}$,
the first term dominates, meaning that the estimation error is the leading contribution to the excess risk. Besides, the exponent $\varepsilon$ arises from the change of norm, as emphasized by Proposition~\ref{prop:changementnorme}.
Note that the obtained rate of convergence is slower than the usual optimal  nonparametric rate in the supervised classification framework, which is of order $N^{-1/4}$ (see \emph{e.g.}~\cite{yang}). It reflects the additional difficulty of estimating interaction function of an IPS (see Section~\ref{subsec:TheoPropertiesEstimator} for other remarks). 

Finally, let us note that the choice $\Delta =O(1/N)$ is not optimal regarding Theorem~\ref{theo:excessrisk}, and Theorem~\ref{theo:riskbhatMkk}. Indeed, this condition can be weakened to $\Delta= O(1/N^{1/3})$. However, the condition $\Delta =O(1/N)$ is classical and widely considered in the literature as in~\cite{AMORINOMKV}.

\begin{remark}[The trade-off induced by the change of class]
This is a key aspect of our analysis. Indeed, the need to control the estimation error of the drift associated with class $k$ along a (limit) trajectory generated from another class $k \neq k'$ is intrinsic to the plug-in approach. Intuitively, the exponential factor appearing in Proposition~\ref{prop:changementnorme} reflects the discrepancy between the probability measures
$\mathbb{P}_k$ and $\mathbb{P}_{k'}$, for $k \neq k'$.

More precisely, as shown in the proof of Proposition~\ref{prop:changementnorme} (see Section~\ref{subsubsec:proof-changementnorme}), the control of a martingale term based on Lemma~2.1 of~\cite{vandeGeer1995} yields for $k \neq k'$
\begin{equation*}
\|\w{b}_k-b_k^*\|^2_{M,k,k',\star} \lesssim \exp(a) \|\w{b}_k-b_k^*\|^2_{M,k,k, \star} + \exp(-a^2/2c),
\end{equation*}
where 
\begin{equation*}
c =\max_{k,k' \in \mathcal{Y}}\int_0^1 (b_k^*[\xi_s, \mu^{[k]}_s]-b_{k'}^*[\xi_s, \mu^{[k']}_s])^2ds,
\end{equation*}
and the inequality holds for all $a > 0$.

This inequality illustrates the trade-off induced by the parameter $a$. When $c$ is close to zero, meaning that the drift functions are nearly indistinguishable, the second term is negligible for any fixed $a > 0$. Consequently, the convergence rate {\it w.r.t.} $\left\|\cdot\right\|_{M,k,k', \star}$ is essentially the same as that obtained for $\left\|\cdot\right\|_{M,k,k, \star}$. However, this regime also corresponds to a difficult classification problem since the classes are barely distinguishable.

Conversely, when $c$ is larger, the classification problem becomes easier because the drift functions are better separated. In this case, parameter $a$ should be chosen sufficiently large so that the residual term $\exp(-a^2/2c)$ is negligible, while avoiding an excessive increase of the multiplicative factor $\exp(a)$, which would deteriorate the convergence rate. Hence, the choice $a = \sqrt{2c\log(N)}$ provides a trade-off between these two competing effects.
\end{remark}

\section{Numerical illustration }\label{sec:num}

This section presents numerical experiments illustrating the proposed classification procedure for McKean–Vlasov particle systems. We consider two settings: a two-label example and a five-label example (Section~\ref{sec:Simulation-setting}). In both settings, the training data are simulated with a time step of $0.01$ using the particle method introduced in~\cite{liu2024particle-method}. For each example, we evaluate the empirical classification accuracy for different training sample sizes and time discretization levels. We also compare the performance of the proposed plug-in classifier with that of the simulated optimal classifier and a direct neural-network classifier.

All simulations are implemented in Python. The code used to reproduce the numerical experiments is available at~\href{https://github.com/YatingLiu2023/Classification-Mean-Field-Particle-Systems}{https://bit.ly/4yA9Yfa}.

\subsection{Simulation setting}\label{sec:Simulation-setting}

\paragraph*{Scenario A: two-class classification example}
%
We consider a two-class classification problem, referred to as Scenario A, where the drift functions 
\begin{eqnarray*}
\begin{cases}
b^{\,*}_{1}(x,y)&=\theta_0
\left[\frac{1}{4}+\frac{3}{4}\cos^2(x)+ \cos^2(y)\right] \\
b^{\,*}_{2}(x,y)&= 2 (x-y)\exp(- (x-y)^2)-8 (x-y)\exp(-2 (x-y)^2).
\end{cases}
\end{eqnarray*}
The parameter $\theta_0$ in $b^{*}_{1}$ controls the separation between the trajectories in these two particle systems. We investigate two representative regimes:
\begin{enumerate}
    \item[(i)]$\theta_0 = 1$, which produces well separated trajectories, and
    \item[(ii)]$\theta_0 = 0.5$, for which the trajectories show substantial overlap, resulting in a more challenging classification task.
\end{enumerate}

Figure~\ref{fig:Atwo-regimes-true-drifts} illustrates the true drift functions $b^{*}_{1}$ (for $\theta_0 = 1$ and $\theta_0 = 0.5$) and $b^{*}_{2}$. Figure~\ref{fig:Atwo-regimes-path} displays representative trajectories under regimes A(i) and A(ii).

\begin{figure}[H]
\centering
\begin{subfigure}[t]{0.32\textwidth}
    \centering
\includegraphics[width=\linewidth]{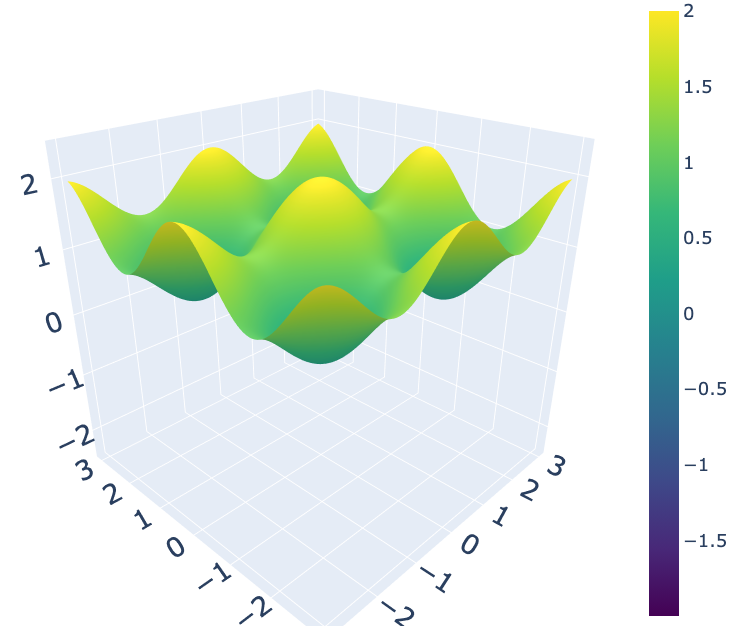}
    \caption{A(i) True drift $b_1^*$ with $\theta_0 = 1$}
\end{subfigure}
\begin{subfigure}[t]{0.32\textwidth}
    \centering
    \includegraphics[width=\linewidth]{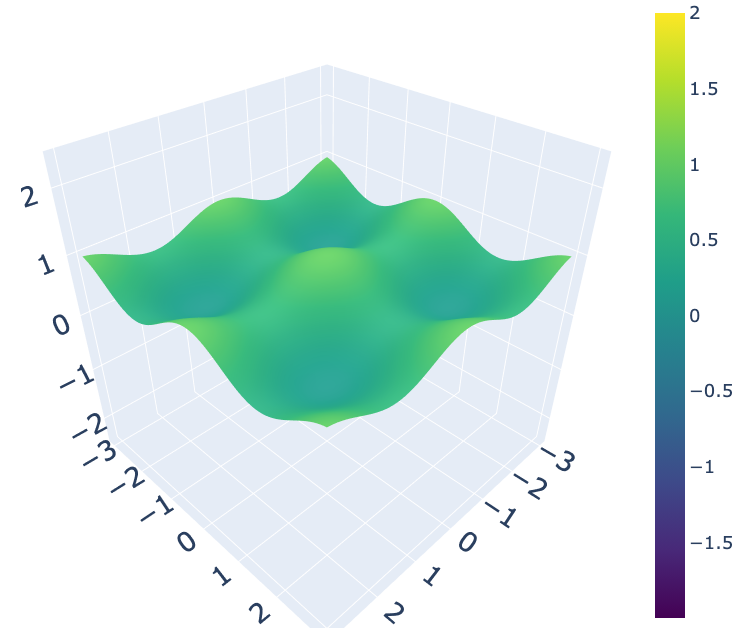}
    \caption{A(ii) True drift $b_1^*$ with $\theta_0 = 0.5$}
\end{subfigure}
\begin{subfigure}[t]{0.32\textwidth}
    \centering
    \includegraphics[width=\linewidth]{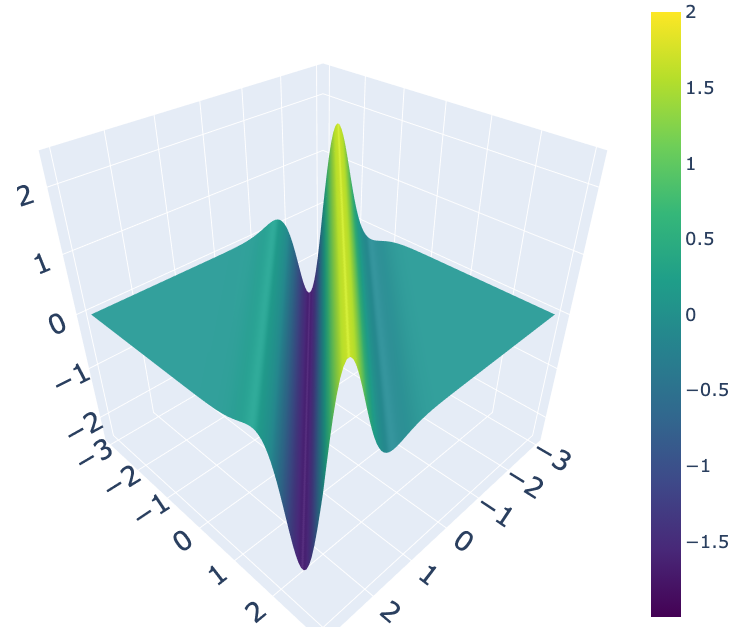}
    \caption{A - True drift $b_2^*$}
\end{subfigure}
\caption{True drift functions of Scenario A}
\label{fig:Atwo-regimes-true-drifts}
\end{figure}

\begin{figure}[H]
\centering
\begin{subfigure}[t]{0.4\textwidth}
    \centering
    \includegraphics[width=\linewidth]{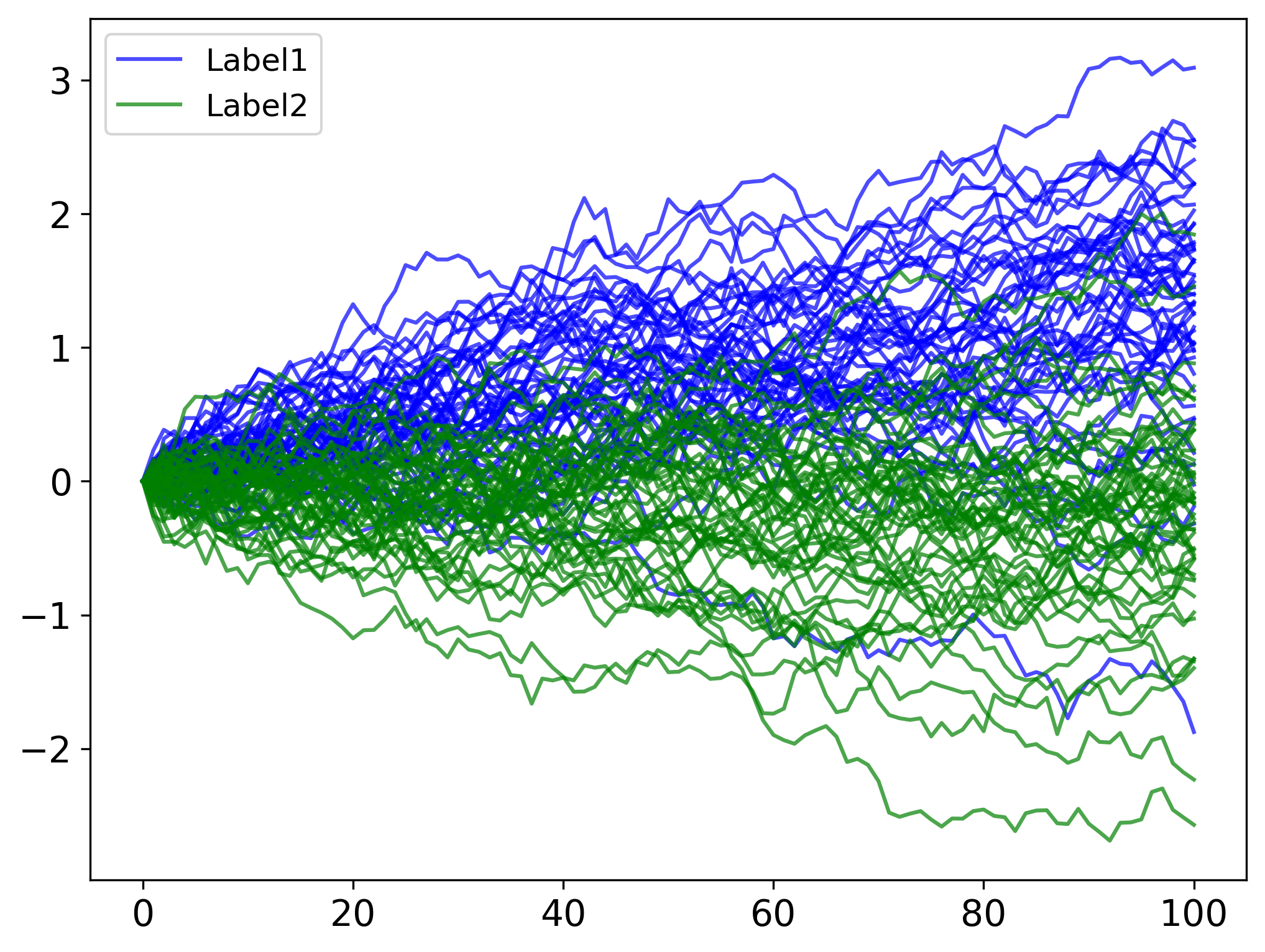}
\end{subfigure}
\quad 
\begin{subfigure}[t]{0.4\textwidth}
    \centering
\includegraphics[width=\linewidth]{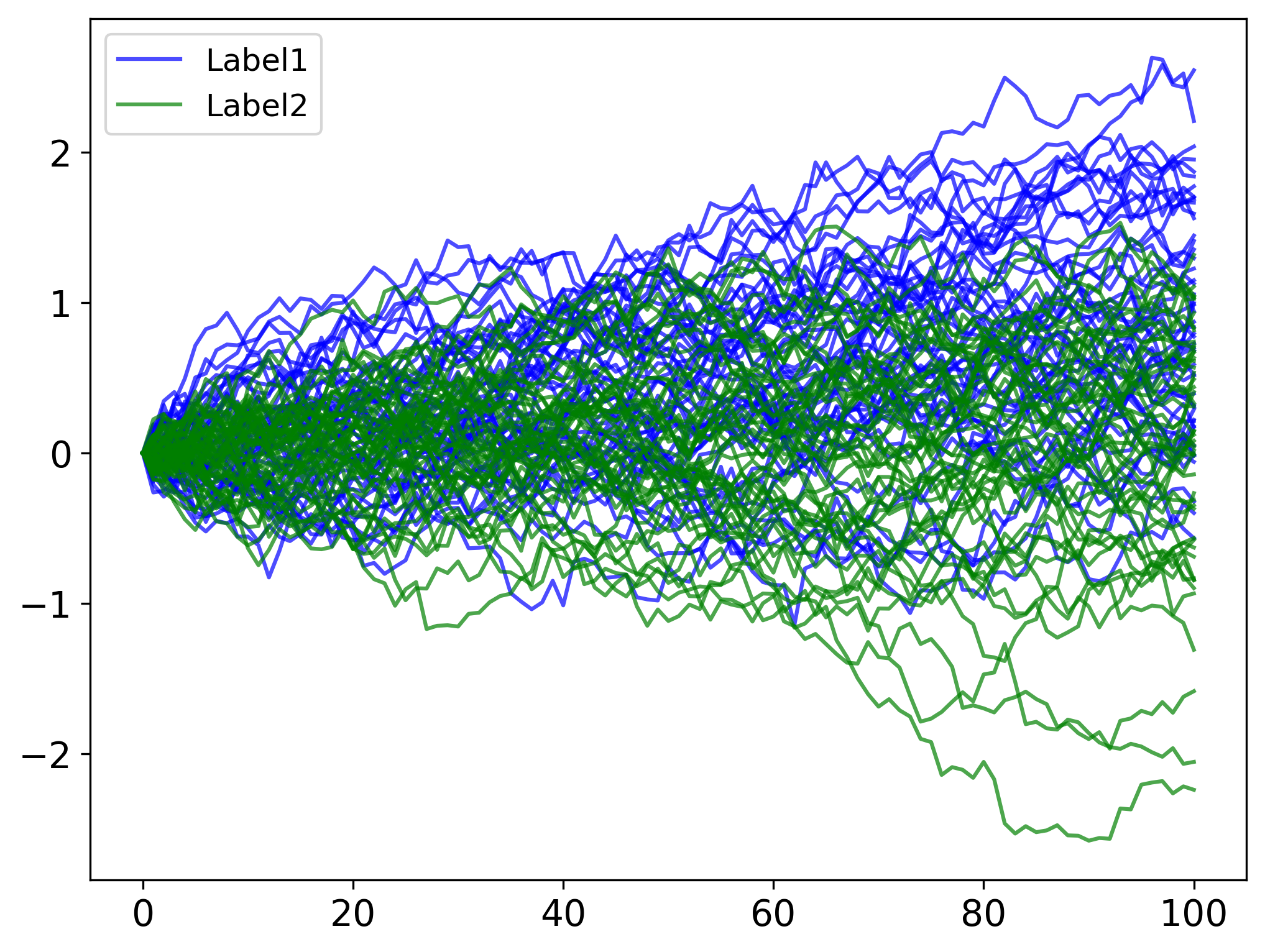}
\end{subfigure}
\caption{Representative trajectories under regimes A(i) $\theta_0 = 1$ (left) and A(ii) $\theta_0 = 0.5$ (right).}
\label{fig:Atwo-regimes-path}
\end{figure}

\paragraph*{Scenario B: Five-class classification}
We next consider a five-class classification problem inspired by \cite{DDMVC2022}. 
The corresponding interacting particle systems \eqref{eq:sys} are defined with drift functions $b_k^*$, for $k=1,\dots,5$,
\begin{equation}\label{eq:def-bk-5labels}
    b^{\,*}_{k}(x,y)=\theta_k\left[\frac{1}{4}+\frac{3}{4}\cos^2(x)+ \cos^2(y)\right].
\end{equation}
The parameter $\theta_k$ determines the class. 
In this example, we choose
\[
(\theta_1, \theta_2, \theta_3, \theta_4, \theta_5) = (-6,-3,0,3,6).
\]
Figure~\ref{fig:5-label-paths} illustrates 50 sample paths generated under the different labels. 
\begin{figure}[H]
    \centering
    \includegraphics[width=0.5\linewidth]{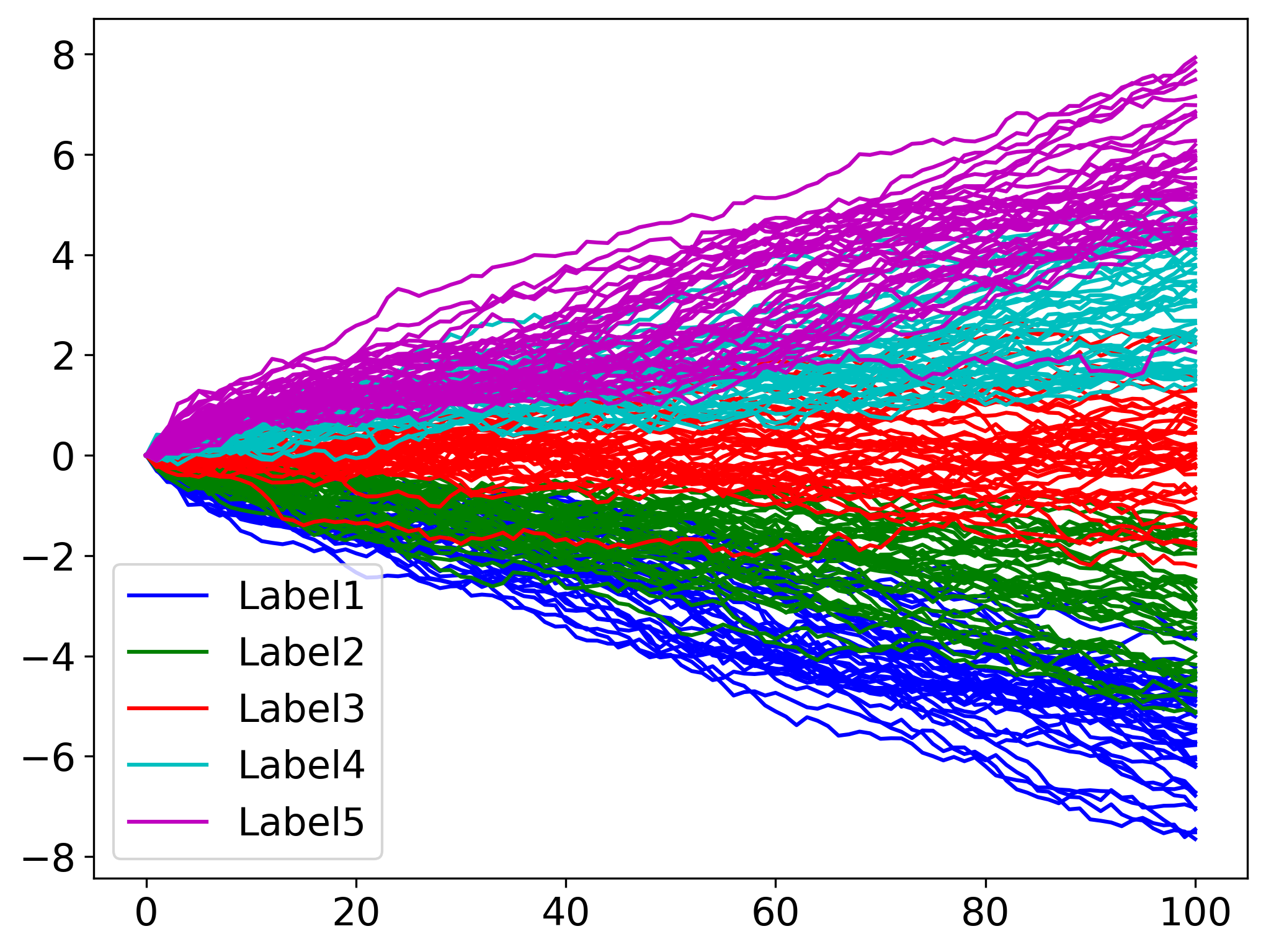}
\caption{Representative trajectories for each of the five labels.}
    \label{fig:5-label-paths}
\end{figure}

\paragraph*{Approximation of the optimal classifier}

To provide a benchmark for the classification performance, we approximate the optimal classifier $g^*$ defined in \eqref{eq:gstarfinal}. More precisely, we consider 
\begin{equation}\label{eq:approximate-bayes}
  \widehat{g}^{\text{optimal}}(Z) \in \argmax{k \in \mathcal{Y}}\wt{F}_k(Z), 
\end{equation}
where $\widetilde{F}_k$ is a numerical approximation of the discretized score defined later in~\eqref{eq:Fhat}. 
The approximation is computed using with 
$N_{\text{ref}} = 5000$ and $M_{\text{ref}} = 100$. The resulting empirical classification accuracies of this approximate optimal classifier are \textbf{0.808} for Scenario~A(i), \textbf{0.696} for Scenario~A(ii), and \textbf{0.906} for Scenario~B.

\subsection{Implementation and benchmark methods}

\paragraph{Implementation details.}

For Scenario A, the B-spline basis is constructed over the rectangular domain 
$[-6,6] \times [-6,6]$, chosen to cover the region explored by the simulated trajectories. The dimension $D$ of the B-spline basis is a tuning parameter. In the experiments reported below, we set $D=2$ and use B-splines of order $2$. For Scenario~B, the B-spline basis is constructed on the larger domain $[-12,12]\times[-12,12]$, reflecting the wider range of the simulated trajectories. We set $D=20$, while keeping the spline order equal to $2$. The choice of the B-spline basis dimension is further investigated in Section~\ref{subsubsec:selection_b_spline_dim}.


\paragraph{Benchmark method.}

We compare our classification rule with a direct neural-network baseline, referred to as \emph{direct NN}. The baseline is a fully connected feedforward neural network with ReLU activations. The network outputs class logits and is trained using the Adam optimizer with cross-entropy loss. For each class $k$, we use $\big(\bar{X}^{[k],1},\ldots,\bar{X}^{[k],N}\big)$ from $\mathcal{D}^{[k]}_N$ as the training sample, while the particle system $\big(\bar{\widetilde{X}}^{[k],1},\ldots,\bar{\widetilde{X}}^{[k],N}\big)$ is used as the validation sample for hyperparameter tuning.

For each repetition, we perform a random hyperparameter search over 20 configurations. The search includes the number of hidden layers in \{1,2,3,4\}, the number of neurons per hidden layer in \{16, 32, $\dots$, 128\}, and the learning rate (log-uniform in $[10^{-4},10^{-2}]$). The configuration achieving the highest validation accuracy is selected. The selected configuration is then reinitialized and trained on the training system, with the validation system used for early stopping. 

\subsection{Experimental results}

The numerical experiments are conducted with $N_{\mathrm{train}}\in\{100,1000\}$ training particles per class and $M_{\mathrm{train}}\in\{10,100\}$ observation steps, corresponding to observation time steps $\Delta_{\mathrm{train}}=1/M_{\mathrm{train}}$. The test system contains $N_{\mathrm{test}}=1000$ particles per class. For each configuration, we perform 50 independent Monte Carlo repetitions and report the mean classification accuracy together with its sample standard deviation.


\paragraph*{Results for Scenario A}

The numerical results are reported in Tables~\ref{tab:compare_with_NN-2labels-regime1} and~\ref{tab:compare_with_NN-2labels-regime2}. As expected from the stronger overlap between the two classes, Regime~A(ii) ($\theta_0=0.5$) is more difficult than Regime~A(i) ($\theta_0=1$). For our plug-in classifier, the mean accuracy ranges from $0.803$ to $0.807$ in Regime~A(i), compared with $0.679$ to $0.692$ in Regime~A(ii). The corresponding standard deviations remain small across all configurations, ranging from $0.009$ to $0.010$ in Regime~A(i) and from $0.011$ to $0.013$ in Regime~A(ii).

Increasing the number of training particles from $N_{\mathrm{train}}=100$ to $N_{\mathrm{train}}=1000$ generally improves the classification accuracy. In Regime~A(i), the improvement is approximately $0.002$--$0.004$, depending on the observation grid, while in Regime~A(ii) it is approximately $0.008$--$0.010$. By comparison, increasing $M_{\mathrm{train}}$ from $10$ to $100$ has a smaller and less systematic effect. 

We next compare the proposed procedure with the direct neural-network classifier. The plug-in classifier achieves a higher mean accuracy in all eight configurations. In Regime~A(i), the direct NN attains mean accuracies between $0.748$ and $0.801$, compared with $0.803$--$0.807$ for the plug-in method. In Regime~A(ii), the direct NN achieves mean accuracies between $0.617$ and $0.679$, compared with $0.679$--$0.692$ for the plug-in classifier.

The direct NN also exhibits greater variability across Monte Carlo repetitions. Its standard deviation ranges from $0.012$ to $0.094$ in Regime~A(i) and from $0.029$ to $0.067$ in Regime~A(ii), whereas the standard deviation of the plug-in classifier remains between $0.009$ and $0.013$. The difference in variability is particularly pronounced for the smaller training sample size $N_{\mathrm{train}}=100$. In particular, for $(M_{\mathrm{train}},N_{\mathrm{train}})=(10,100)$, the direct NN exhibits class collapse in $5$ out of the $50$ repetitions in both regimes. These occasional collapse events contribute to the larger variability observed for the direct NN in the small-sample setting.
\begin{table}[H]
\centering
\caption{Performance comparison between our proposed method and a direct neural network classifier in Regime~A(i) ($\theta = 1$).}
\label{tab:compare_with_NN-2labels-regime1}
\begin{tabular}{ccccccccc}
\hline 
$M_{\text{train}}$ & $N_{\text{train}}$ 
& \makecell{\vspace{-0.25cm}\\Our method\\mean accuracy}
 
  & \makecell{\vspace{-0.25cm}\\Direct NN\\mean accuracy} \\
\hline
10 & 100 & 0.803 (0.010) & 0.748 (0.094) \\
10 & 1000 & 0.807 (0.010) & 0.785 (0.073) \\
100 & 100 & 0.803 (0.010) & 0.749 (0.091) \\
100 & 1000 & 0.805 (0.009) & 0.801 (0.012) \\
\hline
\end{tabular}%
\end{table}

\begin{table}[H]
\centering
\caption{Performance comparison between our proposed method and a direct neural network classifier in Regime~A(ii) ($\theta = 0.5$).}
\label{tab:compare_with_NN-2labels-regime2}
\begin{tabular}{ccccccccc}
\hline 
$M_{\text{train}}$ & $N_{\text{train}}$ 
& \makecell{\vspace{-0.25cm}\\Our method\\mean accuracy (std)}
 
  & \makecell{\vspace{-0.25cm}\\Direct NN\\mean accuracy (std)} \\
\hline
10 & 100 & 0.679 (0.013) & 0.617 (0.067) \\
10 & 1000 & 0.687 (0.012) & 0.679 (0.029) \\
100 & 100 & 0.682 (0.011) & 0.618 (0.064) \\
100 & 1000 & 0.692 (0.011) & 0.672 (0.040) \\
\hline
\end{tabular}%
\end{table}

\paragraph*{Results for Scenario B}

For Scenario~B, we follow the same experimental protocol as in Scenario~A. The results are reported in Table~\ref{tab:comparison-with-NN-5labels}.

Overall, the conclusions are similar to those obtained for Scenario~A. The proposed plug-in classifier performs well across all four configurations. Increasing the number of training particles to $N_{\mathrm{train}}=1000$ raises the mean accuracies to $0.898$ and $0.906$ for $M_{\mathrm{train}}=10$ and $M_{\mathrm{train}}=100$, respectively. In particular, for $(M_{\mathrm{train}},N_{\mathrm{train}})=(100,1000)$, the accuracy of the plug-in classifier is essentially identical to that of the optimal classifier. The standard deviations remain small across all configurations, ranging from $0.004$ to $0.007$.

The proposed method also achieves a higher mean accuracy than the direct neural-network classifier in all four configurations. Moreover, the direct NN exhibits larger variability across Monte Carlo repetitions, with standard deviations ranging from $0.030$ to $0.191$, compared with $0.004$--$0.007$ for the proposed method.

\begin{table}[H]
\centering
\caption{Performance comparison between our proposed method and a direct neural network classifier for Scenario B.}
\label{tab:comparison-with-NN-5labels}
\begin{tabular}{ccccccccc}
\hline 
$M_{\text{train}}$ & $N_{\text{train}}$ 
& \makecell{\vspace{-0.25cm}\\Our method\\mean accuracy (std)}
 
  & \makecell{\vspace{-0.25cm}\\Direct NN\\mean accuracy (std)} \\
\hline
10 & 100 & 0.887 (0.007) & 0.815 (0.131) \\
10 & 1000 & 0.898 (0.004) & 0.882 (0.041) \\
100 & 100 & 0.894 (0.005) & 0.741 (0.191) \\
100 & 1000 & 0.906 (0.004) & 0.883 (0.030) \\
\hline
\end{tabular}%
\end{table}

\section{Conclusion and discussion}

In this paper, we studied the supervised classification problem where the features consist of discrete observations of particles evolving according to interacting particle systems associated with $K$ classes. Assuming that the classes are distinguished through the drift function of the underlying stochastic differential equation, we introduced the mean-field Bayes classifier based on the limiting McKean--Vlasov dynamics. Building upon this characterization, we proposed a plug-in classification procedure relying on a preliminary nonparametric estimator of the interacting drift function.

We established the consistency of the proposed classifier and derived a convergence rate for its excess risk with respect to the mean-field Bayes risk. Our analysis highlights the combined effect of the statistical estimation error and the finite-size approximation of the interacting particle system. Although drift estimation is only an intermediate step in the construction of the classifier, we also obtained an integrated convergence rate for the drift estimator itself, showing an overall rate of order $N^{-1/6}$. Finally, numerical experiments illustrate the effectiveness of the proposed approach and support the theoretical findings.

Throughout this work, we assumed that the class proportions are known. This assumption is motivated by our observational framework, in which only two interacting particle systems are available for each class, preventing a consistent estimation of the class distribution. Estimating the class prior would therefore require a richer sampling framework with repeated independent observations of the interacting particle systems. We also note that our results can be readily extended to weighted classification risks with known class-dependent misclassification costs.

For simplicity, both the methodology and the theoretical analysis were developed for one-dimensional particle systems. While the theoretical arguments naturally extend to finite-dimensional settings, the proposed nonparametric estimator relies on a two-dimensional $B$-spline basis, making both its construction and the associated analysis substantially more involved in higher dimensions. Developing computationally efficient estimation procedures together with corresponding theoretical guarantees in this setting remains an interesting direction for future work.

Finally, another promising perspective is the estimation of the diffusion coefficient within the classification procedure. While considerable progress has been made on drift estimation for McKean--Vlasov stochastic differential equations, the nonparametric estimation of the diffusion coefficient remains largely unexplored. Addressing this problem requires simultaneously controlling the statistical estimation error and the propagation-of-chaos error induced by the empirical approximation of the underlying distribution, making it a particularly challenging avenue for future research.


\section*{Acknowledgments}
This work is part of the 2022 DAE 103 EMERGENCE(S) - PROCECO project supported by Ville de Paris.

This paper was completed while C. Dion-Blanc was affiliated
with the Centre de Math\'ematiques Appliqu\'ees (CMAP) at Ecole Polytechnique. She is grateful to Sylvie 
M\'el\'eard for the opportunity to join her team, supported by the European Union (ERC, SINGER, 101054787).

Yating Liu acknowledges the financial support from the CNRS through the MITI interdisciplinary programs and support from the French National Research Agency (ANR) under the France 2030 program (ANR-23-EXMA-0011, MIRTE Project). 

The authors also thank Marc Hoffmann for fruitful discussions. 

\newpage
\section{Proofs}

\subsection{Technical results}\label{proofs:technical}
The following result is taken from \citet[Lemma 20]{hoffmann2021}. For the reader’s convenience, we restate it below using our notation.

\begin{lemma}\label{lem:hoffmann2021-lem20}
Assume that Assumption \ref{ass:bsup_et_lip_in_x_and_y} holds.  For any fixed $k \in \mathcal{Y}$, let $\xi^{[k]} = (\xi^{[k]}_t)_{t \in [0,1]}$ denote the unique strong solution to \eqref{eq:mkv} given $Y=k$. Then, for every $p\geq 1$, it holds 
\begin{equation}\label{eq:hoffmann2021-lem20}
\sup_{t\in[0,1]}\E\big[|\xi_t^{[k]}|^{2p}\big]\leq p\,! \, \mathfrak{C}^p.
\end{equation}
\end{lemma}

The following proposition is adapted from \citet[Lemma~3.2]{lacker2018mean}, \citet[Propositions~7.6 and~7.7]{pages-numerical-proba}, and \citet[Proposition~2.1]{liu2020functional}.
\begin{proposition}\label{prop:pre-result1}
Grant Assumption \ref{ass:bsup_et_lip_in_x_and_y}. Let $\xi^{[k]}=(\xi^{[k]}_t)_{t\in[0,1]}$ denote the unique strong solution of the McKean-Vlasov equation~\eqref{eq:mkv} given $Y=k$ and let $(\mu^{[k]}_t)_{t\in[0,1]},\,k\in\mathcal{Y}$ denote their respective marginal distributions. Moreover, let $(X_t^{[k], 1}, ..., X_t^{[k], N})_{t\in[0,1]}$ be the particles  defined by \eqref{eq:sys}. 
\begin{enumerate}[label=(\alph*)]
\item For every $p\geq 1$ and for every $k\in\mathcal{Y}$, we have  
\begin{equation}
\sup_{t\in[0,1]}\mathcal{W}_p(\mu_{t}^{[k]}, \delta_0)\leq \Big\Vert \sup_{t\in[0,1]}\big|\xi_t^{[k]}\big|\Big\Vert_p\leq {C_{p,\mathfrak{C}}}.
\end{equation}
\item For every $k\in\mathcal{Y}$ and for every $s,t\in[0,1], s\leq t,$ we have 
\[\mathcal{W}_2(\mu_t^{[k]}, \mu_s^{[k]})\leq \Vert \xi_t^{[k]}- \xi_s^{[k]}\Vert_2\leq \mathfrak{C}\sqrt{t-s}.\]
\item {For every $p\geq 1$, $k\in\cY$, $n\in \{1, ..., N\}$ and} $t,s\in[0,1]$, we have 
\begin{align}\label{eq:particle-module-continuity}
&\qquad \Big\Vert \sup_{t\in[0,T]}\big|X_t^{[k], n}\big|\Big\Vert_p\leq {C_{p,\mathfrak{C}}}\quad\text{and}\quad \left\Vert X_t^{[k], n}-X_s^{[k], n}\right\Vert_p\leq {C_{p,\mathfrak{C}}}\,|t-s|^{1/2}.
\end{align}
\end{enumerate}
\end{proposition}
We also refer to \cite{AMORINOMKV} for similar results Lemma 5.1.

\begin{lemma}\label{lem:ineq-concentration-Xt}
Assume that Assumption \ref{ass:bsup_et_lip_in_x_and_y} holds.  
\begin{enumerate}
\item[(1)] For any fixed $k \in \mathcal{Y}$, let $\xi^{[k]} = (\xi^{[k]}_t)_{t \in [0,1]}$ denote the unique strong solution to \eqref{eq:mkv}. Then, for any $t \in [0,1]$, we have
\[
\P\left(\, \big|\xi^{[k]}_t\big| > A_N \,\right) \leq 2\exp\left(-\mathfrak{C}{A_N^2}\right).\]
\item[(2)] Let $X^{[k],1} = (X^{[k],1}_t)_{t \in [0,1]}$ denote the first particle in the particle system \eqref{eq:sys}. Then, for any $t \in [0,1]$, we have
\[
\P\left( |X^{[k],1}_t| > A_N \right) \leq 2\exp\left(-\mathfrak{C} A_N^2\right).
\]
\end{enumerate}
\end{lemma}

The proof of Lemma \ref{lem:ineq-concentration-Xt} is standard. For completeness, we provide the main steps below.

\begin{proof}[Proof of Lemma \ref{lem:ineq-concentration-Xt}]
(1) By \citet[Lemma 20]{hoffmann2021}, for every integer $p\geq1$, every $k\in\cY$ and every $t\in[0,1]$, we have $\E \big[|\xi_t^{[k]}|^{2p}\big]\leq p!\,\mathfrak{C}^p$. Therefore, $\E\big[\exp (\tfrac{1}{2\mathfrak{C}}|\xi_t^{[k]}|^{2})\big]\leq 2$. Hence, by Markov's inequality, 
\begin{align}
    \P \big(\big|\xi_t^{[k]}\big|>A_N\big)&=\P\Big(\exp (\tfrac{1}{2\mathfrak{C}}\big|\xi_t^{[k]}\big|^{2}\big)> \exp \big(\tfrac{1}{2\mathfrak{C}}A_N^2\big)\Big)\leq \exp \big(-\tfrac{1}{2\mathfrak{C}}A_N^2\big)\E \big[\exp \big(\tfrac{1}{2\mathfrak{C}}|\xi_t^{[k]}|^{2}\big)\big]\nonumber\\
    &\leq 2 \exp \big(-\tfrac{1}{2\mathfrak{C}}A_N^2\big). \nonumber
\end{align}
Up to a change in the value of the generic constant $\mathfrak C$, this yields $\P\big(\, \big|\xi^{[k]}_t\big| > A_N \,\big) \leq 2\exp\big(-\mathfrak{C}{A_N^2}\big).$\smallskip

\noindent(2) By the definition of the particle system in \eqref{eq:sys} and Assumption \ref{ass:bsup_et_lip_in_x_and_y}, for every $k\in\cY$ and every $t\in[0, 1]$, we have $|X_t^{[k],1}|\leq |x_0|+b_{\max}^{*}+|W_t^1|$. Moreover, as $W_t^1\sim\mathcal{N}(0,t)$, we have \[\E \Big[\big|W_t^1\big|^{2p}\Big]=t^p\frac{(2p)!}{2^p p!}\leq\frac{(2p)!}{2^p p!}=\binom{2p}{p}\frac{p!}{2^p}\leq 2^p p!.\] Therefore, $\E \big[|X_t^{[k],1}|^{2p}\big]\leq p! \,\mathfrak{C}^p$ and the remaining argument is identical to that used in the proof of part (1).
\end{proof}


\subsection{Proofs of Section \ref{sec:maths}}\label{subsec:proof-sec-2}

\subsubsection{Proof of Proposition \ref{prop:chaos}}


\paragraph{Point $(a)$.} For each label $k \in \mathcal{Y}$, we define a non-interacting particle system $(\xi^{[k], 1}, ..., \xi^{[k], N})=(\xi_t^{[k], 1}, ..., \xi_t^{[k], N})_{t\in[0,1]}$ given by the following equations:
\begin{equation}\label{eq:sys-without-interaction}
\dd \xi_t^{[k], n}=b_k^{*}\big[\xi_t^{[k], n}, \mu_t^{[k]}\big] \dd t+  \dd W_t^{ n}\quad\text{with}\quad\xi_0^{[k], n}=x_0,\quad 1 \leq n \leq N,\\
\end{equation}
where $(\mu_t^{[k]})_{t\in[0,1]}$ are marginal distributions of the unique solution $(\xi_t^{[k]})_{t\in[0,1]}$ defined by \eqref{eq:mkv} given label $Y=k$, and $W^{ n}=(W^{ n}_t)_{t\in[0,1]}, \,1\leq k\leq K, 1\leq n\leq N$, are the same i.i.d. Brownian motions as defined in \eqref{eq:sys}. Additionally, we define the following empirical measures
\begin{equation}\label{eq:empirical-xi}
    \nu_{t}^{[k],N}\coloneqq \frac{1}{N}\sum_{n=1}^N\delta_{\xi_t^{[k], n}}, \;t\in[0,1].
\end{equation}
Remark that the particles $(\xi^{[k], 1}, ..., \xi^{[k], N})$ defined by \eqref{eq:sys-without-interaction} are i.i.d. having the same distribution as  the unique solution $(\xi_t^{[k]})_{t\in[0,1]}$ defined by \eqref{eq:mkv} given label $Y=k$.
It follows that, for every $p\geq 1$, $t\in[0,1]$,
\begin{align}
&\left\Vert  \sup_{s\in[0,t]}\left|X_s^{[k], n} - \xi_s^{[k], n}\right|\right\Vert_p
\leq \left\Vert  \sup_{s\in[0,t]}\left|\int_{0}^s\left[ b_k^{*}\big[X_u^{[k], n}, \mu_u^{[k], N}\big] -b_k^{*}\big[\xi_u^{[k], n}, \mu_u^{[k]}\big] \right]\dd u\right|\right\Vert_p\nonumber\\
&\quad \leq \left\Vert \int_{0}^t\left| b_k^{*}\big[X_u^{[k], n}, \mu_u^{[k], N}\big] -b_k^{*}\big[\xi_u^{[k], n}, \mu_u^{[k]}\big] \right|\dd u\right\Vert_p\leq \int_{0}^t\left\Vert  b_k^{*}\big[X_u^{[k], n}, \mu_u^{[k], N}\big] -b_k^{*}\big[\xi_u^{[k], n}, \mu_u^{[k]}\big] \right\Vert_p\dd u\nonumber\\
&\quad \leq L\int_{0}^{t}\left[ \left\Vert X_u^{[k], n}- \xi_u^{[k], n}\right\Vert_p +\left\Vert\mathcal{W}_1\left(\mu_u^{[k], N},  \mu_u^{[k]}\right)\right\Vert_p \right]\dd u \nonumber\\
&\quad \leq  L\int_{0}^{t}\left\Vert \sup_{v\in[0,u]} \left|X_v^{[k], n}- \xi_v^{[k], n}\right|\right\Vert_p \dd u + L \int_{0}^{t}\left\Vert\mathcal{W}_1\left(\mu_u^{[k], N},  \mu_u^{[k]}\right)\right\Vert_p \dd u\nonumber
\end{align}
where the third inequality follows from the  generalized Minkowski inequality (see e.g. \cite[Section 7.8.1]{pages-numerical-proba}). Hence, by applying the Gronwall's inequality, we get 
\begin{equation}\label{eq:first-gronwall}
    \left\Vert  \sup_{s\in[0,t]}\left|X_s^{[k], n} - \xi_s^{[k], n}\right|\right\Vert_p\leq L\exp(L)\int_{0}^{t}\left\Vert\mathcal{W}_1\left(\mu_u^{[k], N},  \mu_u^{[k]}\right)\right\Vert_p \dd u
\end{equation}
Moreover, remark that $\frac{1}{N}\sum_{i=1}^N\delta_{\big(X_t^{[k], n}, \,\xi_t^{[k], n}\big)}$ is a coupling of $\nu_{t}^{[k],N}$ defined by \eqref{eq:empirical-xi} and $\mu_t^{[k], N}$ defined by \eqref{eq:sys}. 
Hence, for every $p\geq 1$ and for every $t\in[0,1]$,
\begin{align}\label{eq:inequality-coupling}
\mathcal{W}_p^p\left( \mu_t^{[k], N}, \nu_{t}^{[k],N}\right)\leq \frac{1}{N}\sum_{i=1}^N\left| X_t^{[k], i}-\xi_t^{[k], i}\right|^p,\quad \text{a.s.},
\end{align}
and consequently, 
\begin{align}\label{eq:inequality-coupling-2}
\left\Vert\mathcal{W}_p\left( \mu_t^{[k], N}, \nu_t^{[k], N}\right)\right\Vert_p\leq \left\Vert X_t^{[k], 1}-\xi_t^{[k], 1}\right\Vert_p. 
\end{align}
Therefore, for every $p\geq 1$, $t\in[0,1]$, 
\begin{align}
\sup_{s\in[0,t]}&\left \Vert  \mathcal{W}_p\left(\mu_s^{[k], N}, \mu_s^{[k]} \right)\right\Vert_p \leq \sup_{s\in[0,t]}\left \Vert  \mathcal{W}_p\left(\mu_s^{[k], N}, \nu_s^{[k], N}\right)\right\Vert_p + \sup_{s\in[0,t]}\left \Vert  \mathcal{W}_p\left(\nu_s^{[k], N}, \mu_s^{[k]} \right)\right\Vert_p\nonumber\\
&\leq \sup_{s\in[0,t]}\left \Vert  X_s^{[k], 1}-\xi_s^{[k], 1}\right\Vert_p + \sup_{s\in[0,t]}\left \Vert  \mathcal{W}_p\left(\nu_s^{[k], N}, \mu_s^{[k]} \right)\right\Vert_p\nonumber\\
&\leq \left \Vert \sup_{s\in[0,t]}\left| X_s^{[k], 1}-\xi_s^{[k], 1}\right|\right\Vert_p + \sup_{s\in[0,t]}\left \Vert  \mathcal{W}_p\left(\nu_s^{[k], N}, \mu_s^{[k]} \right)\right\Vert_p\nonumber\\
&\leq
L\exp(L)\int_{0}^{t}\sup_{v\in[0,u]}\left\Vert\mathcal{W}_p\left(\mu_v^{[k], N},  \mu_v^{[k]}\right)\right\Vert_p \dd u+\sup_{s\in[0,t]}\left \Vert  \mathcal{W}_p\left(\nu_s^{[k], N}, \mu_s^{[k]} \right)\right\Vert_p,\nonumber
\end{align}
where the second inequality comes from \eqref{eq:inequality-coupling-2} and the last inequality follows from \eqref{eq:first-gronwall} and the fact that $\mathcal{W}_1(\mu, \nu)\leq \mathcal{W}_p(\mu, \nu)$ for every $\mu, \nu\in\mathcal{P}_p(\R^d)$ (see e.g. \cite[Remark 6.6]{villani2008optimal}). By applying again Gronwall's inequality, we get for every $p\geq 1$, 
\begin{align}
\sup_{s\in[0,1]}&\left \Vert  \mathcal{W}_p\left(\mu_s^{[k], N}, \mu_s^{[k]} \right)\right\Vert_p \leq \exp(L\exp(L))\sup_{s\in[0,1]}\left \Vert  \mathcal{W}_p\left(\nu_s^{[k], N}, \mu_s^{[k]} \right)\right\Vert_p.
\end{align}
Finally, as $\nu_{t}^{[k], N}$ is the empirical measure generated by the i.i.d. sample of random variable having the same distribution as $\mu_t^{[k]}$, and $\left\Vert \xi_t^{k}\right\Vert_q\leq \mathfrak{C}$ for every $q\geq1$ (see Proposition \ref{prop:pre-result1}), one can apply \citet[Theorem 1]{fournier2015on} with $d=1$ and $p\geq 1$,  and get the following inequality for every fixed $N\geq1$,
\begin{align}\label{eq:wp-cvg-chaos}
\sup_{s\in[0,1]}\E \left[ \mathcal{W}_p^p\left(\mu_s^{[k], N}, \mu_s^{[k]} \right)\right]\leq C N^{-\frac{1}{2}} , 
\end{align}
where the constant $C$ above depends on the constant $\mathfrak{C}$ in Proposition \ref{prop:pre-result1} and $p$. Finally, combine \eqref{eq:wp-cvg-chaos} with \eqref{eq:first-gronwall}, we get 
\begin{equation}
    \E\left[ \sup_{s\in[0,1]}\left|X_s^{[k], n} - \xi_s^{[k], n}\right|^p\right]\leq \wt{C} N^{-\frac{1}{2}},
\end{equation}
where the constant $\wt{C}$ above depends on the constant $\mathfrak{C}$ in Proposition \ref{prop:pre-result1} and $p$.

\paragraph{Point $(b)$.}


For a fixed class \(Y=k\) and fixed number of particles $N$, the particle system \eqref{eq:sys} can be rewritten as
\begin{align}\label{eq:sys-to-lacker23}
\dd X_t^{[k],n}&=\frac{1}{N}\sum_{n'=1}^N 
b_k^{*}\!\left(X_t^{[k],n},X_t^{[k],n'}\right)\dd t
+ \dd W_t^{n}\nonumber\\
&=\frac{1}{N-1}\sum_{n'=1, n'\neq n}^N 
\frac{N-1}{N} b_k^{*}\!\left(X_t^{[k],n},X_t^{[k],n'}\right)\dd t + \frac{1}{N}b_k^{*}\!\left(X_t^{[k],n},X_t^{[k],n}\right)\dd t+ \dd W_t^{n}.
\end{align}
Define $\mathfrak{b}_0(x)=\frac{1}{N}b_k^*(x,x)$ and  $\mathfrak{b}(x,y)=\frac{N-1}{N}b_k^*(x,y)$.  Then $\mathfrak{b}_0$ and $\mathfrak{b}$ are bounded, since  the functions $b_k^*, k\in\mathcal{Y}$ are assumed to be bounded under Assumption \ref{ass:bsup_et_lip_in_x_and_y}. Hence, one can apply \citet[Theorem 2.14]{Lacker2023Hierarchies} and  Pinsker’s inequality to obtain 
\begin{align}
&d_{\mathrm{TV}}\big(\mathrm{Law}(X^{[k], 1}, \ldots, X^{[k], l}), ({\mu^{[k]}})^{\otimes l}\big) \leq \mathfrak{C}\frac{l}{N}. 
\end{align}

\begin{lemma}\label{lem:application-TV-distance-and-Bernstein}
Assume that Assumption~\ref{ass:bsup_et_lip_in_x_and_y} holds. Given \( Y = k \in \mathcal{Y} \), let \( \xi^{[k]} = (\xi_t^{[k]})_{t \in [0,1]} \) be the unique solution to the McKean--Vlasov Equation~\eqref{eq:mkv}, and let \( \mu^{[k]} \) denote its law, with marginal distributions \( (\mu_t^{[k]})_{t \in [0,1]} \). Let \( (X_t^{[k],1}, \ldots, X_t^{[k],N})_{t \in [0,1]} \) be the particle system defined by~\eqref{eq:sys} and let $\mu_t^{[k], N}=\frac{1}{N}\sum_{n=1}^{N}\delta_{X_t^{[k],n}}$, $t\in[0,1]$. Let $\phi:\R\rightarrow\R$ be a bounded measurable function. Then, for every fixed $k\in\mathcal{Y}$ and for every $t\in[0,1]$, the following estimates hold,
\begin{enumerate}
    \item[$(a)$] 
$\displaystyle \left|\E \left[  \frac{1}{N}\sum_{n=1}^{N} \phi(X_t^{[k],n})-\int_{\R}\phi(y)\mu_t^{[k]}(\dd y )\right]\right|\leq \mathfrak{C} \Vert \phi\Vert_{\sup} N^{-1}$;
\item[$(b)$]
$\displaystyle \E \left[  \left|\frac{1}{N}\sum_{n=1}^{N} \phi(X_t^{[k],n})-\int_{\R}\phi(y)\mu_t^{[k]}(\dd y )\right|\right]\leq \mathfrak{C} \Vert \phi\Vert_{\sup} N^{-\frac{1}{2}}$;
\item[$(c)$]
$\displaystyle \E \left[  \left|\frac{1}{N}\sum_{n=1}^{N} \phi(X_t^{[k],n})-\int_{\R}\phi(y)\mu_t^{[k]}(\dd y )\right|^2\right]\leq \mathfrak{C} \Vert \phi\Vert_{\sup}^2 N^{-1}$.
\end{enumerate}
\end{lemma}
\begin{proof}[Proof of Lemma \ref{lem:application-TV-distance-and-Bernstein}]
For $(a)$, we have from Proposition~\ref{prop:chaos} that
\begin{align}
 & \left|\E \left[ \frac{1}{N}\sum_{n=1}^{N} \phi(X_t^{[k],n})-\int_{\R}\phi(y)\mu_t^{[k]}(\dd y )\right] \right|= \left|\E \left[  \frac{1}{N}\sum_{n=1}^{N} \phi(X_t^{[k],n})-\frac{1}{N}\sum_{n=1}^{N}\int_{\R}\phi(y)\mu_t^{[k]}(\dd y )\right]\right|\nonumber\\
    &\leq \left|\frac{1}{N}\sum_{n=1}^{N} \left\{\E \left[  \phi(X_t^{[k],n})\right]-\int_{\R}\phi(y)\mu_t^{[k]}(\dd y )\right\}\right|\leq 2\Vert \phi\Vert_{\sup} d_{\mathrm{TV}}\Big(\mathrm{Law}(X_t^{[k],1}), \mu_t^{[k]}\Big)\leq 2\mathfrak{C}\Vert \phi\Vert_{\sup} N^{-1}. \nonumber
\end{align}
For $(b)$, by Theorem 18 in \cite{hoffmann2021}, there exist constants $\kappa_1,\kappa_2>0$ depending on $x_0$, $\|b\|_{\infty}$ such that for all $t\in[0,T]$ and all $k\in\mathcal{Y}$,
\begin{align}
    \mathbb{P}\left( \frac{1}{N}\sum_{n=1}^N\phi\big({X}_{t}^{[k], n}\big)-\int_{\R} \phi(y) \mu^{[k]}_{t}(dy)\geq x \right)\leq \kappa_1\exp\left( -\kappa_2 \frac{Nx^2}{|\phi|^2_{L^2({\mu_t^{[k]}})}+|\phi|_{\sup}x}\right). 
\end{align}
Therefore, 
\begin{align}
    &\mathbb{E}\left[\left|\frac{1}{N}\sum_{n=1}^N\phi\big({X}_{t}^{[k], n}\big)-\int_{\R} \phi(y) \mu^{[k]}_{t}(dy)\right|\right]=\int_{0}^{\infty}\mathbb{P}\left( \left|\frac{1}{N}\sum_{n=1}^N\phi\big({X}_{t}^{[k], n}\big)-\int_{\R} \phi(y) \mu^{[k]}_{t}(dy)\right|\geq x \right) \dd x\nonumber\\
    &\leq 2\int_{0}^{\infty}\kappa_1\exp\left( -\kappa_2 \frac{Nx^2}{\|\phi\|^2_{L^2(\mu_t)}+\|\phi\|_{\infty}x}\right) \dd x \leq \frac{\mathfrak{C}{\|\phi\|_{\infty}}}{\sqrt{N}}. 
\end{align}
Similarly, for $(c)$, we have
\begin{align}
    &\mathbb{E}\left[\left|\frac{1}{N}\sum_{n=1}^N\phi\big({X}_{t}^{[k], n}\big)-\int_{\R} \phi(y) \mu^{[k]}_{t}(dy)\right|^2\right]=\int_{0}^{\infty}\mathbb{P}\left( \left|\frac{1}{N}\sum_{n=1}^N\phi\big({X}_{t}^{[k], n}\big)-\int_{\R} \phi(y) \mu^{[k]}_{t}(dy)\right|^2\geq x \right) \dd x\nonumber\\
    &=\int_{0}^{\infty}\mathbb{P}\left( \left|\frac{1}{N}\sum_{n=1}^N\phi\big({X}_{t}^{[k], n}\big)-\int_{\R} \phi(y) \mu^{[k]}_{t}(dy)\right|\geq \sqrt{x} \right) \dd x\nonumber\\
    &\leq 2\int_{0}^{\infty}\kappa_1\exp\left( -\kappa_2 \frac{Nx}{\|\phi\|^2_{L^2(\mu_t)}+\|\phi\|_{\infty}\sqrt{x}}\right) \dd x \leq \frac{\mathfrak{C}{\|\phi\|_{\infty}^2}}{{N}}. \nonumber\hfill\qedhere
\end{align}
\end{proof}

\color{black}

\subsection{Proof of Section \ref{sec:classif}}\label{proofs:classif}

\subsubsection{Proof of Proposition \ref{prop:excessriskTerm13}} 

The goal is to bound the error term 
$\E[ \cR_{(\bar{Z}^{[Y]})}(\w{g})]- \E[ \cR_{(\bar{\xi}^{[Y]})}(\w{g})]$ following Equation~\eqref{eq:decomposition-error}. Recall that, for two probability measures $\mu$ and $\nu$  on $\mathcal{C}([0,1], \mathbb{R})$, equipped with the Borel $\sigma$-algebra induced by the supremum norm $\|\cdot\|_{\infty}$, the total variation distance $d_{\mathrm{TV}}(\mu,\nu)$ is defined in \eqref{eq:total-variation-distance}. If $X_1$ and $X_2$ are  random variables with respective distributions $\mu$ and $\nu$, we write, with a slight abuse of notation, $d_{\mathrm{TV}}(X_1, X_2)$ for $d_{\mathrm{TV}}(\mu,\nu)$.

Now we introduce a finite-dimensional projection $\mathrm{Proj}_{\Delta}:\mathcal{C}([0,1], \R) \rightarrow \R^{M+1}$, which maps a the continuous trajectory to its discrete-time observations on the grid associated with $\Delta$. 
The projection $\mathrm{Proj}_{\Delta}$ is defined as follows:
\begin{equation}\label{eq:def-projection-delta}
\alpha\in\mathcal{C}\big([0,1], \R\big)\mapsto \mathrm{Proj}_{\Delta}(\alpha)=(\alpha_0,\alpha_{\Delta}, \alpha_{2\Delta}, ..., \alpha_{1})\in\R^{M+1}.  
\end{equation} 
Let $g^{\Delta}: \R^{M+1}\rightarrow \cY$ be an arbitrary measurable classifier. 
It follows from the definition of the risk function in Section \ref{subsec:Stasetting} that
\begin{align}
    \cR_{(\bar{Z}^{[Y]})}(g^{\Delta}) &= \mathbb{P}\left(g^{\Delta}(\bar{Z}^{[Y]}) \neq Y\right)=\mathbb{P}\left(g^{\Delta}\big(\mathrm{Proj}_{\Delta}({Z}^{[Y]})\big) \neq Y\right)=\E \left[\E\left[\one_{\{g^{\Delta}(\mathrm{Proj}_{\Delta}({Z}^{[Y]})) \neq Y\}}\;\big|\;Y\right]\right]\nonumber\\
    &=\sum_{k=1}^{K}\frac{1}{K}\E\left[\one_{\{g^{\Delta}(\mathrm{Proj}_{\Delta}({Z}^{[k]})) \neq k\}}\;\big|\;Y=k\right]=\sum_{k=1}^{K}\frac{1}{K}\E\left[\one_{\{g^{\Delta}(\mathrm{Proj}_{\Delta}({Z}^{[k],1})) \neq k\}}\right]
\end{align}
where \( Z^{[k],1} \) denotes the first particle in the system defined by~\eqref{eq:sys}, conditioned on the label \( Y = k \), with system size \( \Nt\).
Similarly, 
\begin{align}
\cR_{(\bar{\xi}^{[Y]})}({g^{\Delta}}) =\frac{1}{K}\sum_{k=1}^{K} \E \left [\one_{\{g^{\Delta}(\mathrm{Proj}_\Delta(\xi^{[k]})) \neq k\}}\right].
\end{align}
where $\xi^{[k]}$ is a strong solution of \eqref{eq:mkv} given \( Y = k \). 
Hence, 
\begin{align}
\left |\mathcal{R}_{(\bar{Z}^{[Y]})}(g^{\Delta})-\mathcal{R}_{(\bar{\xi}^{[Y]})}(g^{\Delta})\right|&\leq \frac{1}{K} \sum_{k=1}^{K}   \Big|\E \left [\one_{\{g^{\Delta}(\mathrm{Proj}_\Delta(Z^{[k],1})) \neq k\}}\right]-\E \left [\one_{\{g^{\Delta}(\mathrm{Proj}_\Delta(\xi^{[k]}))\neq k\}}\right]\Big|\nonumber\\
&\leq 2 d_{\mathrm{TV}}(Z^{[k],1}, \xi^{[k]}), \nonumber
\end{align}
where the last inequality comes from the definition~\eqref{eq:total-variation-distance} by considering the function 
\[f:\alpha \in\mathcal{C}([0,1], \R)\mapsto f(\alpha)\coloneqq\one_{\{g^{\Delta}(\mathrm{Proj}_\Delta(\alpha))\neq k\}}\in\{0,1\}.\]
Applying the strong propagation of chaos property stated in Proposition~\ref{prop:chaos}-(b) with $l=1$, we obtain
\begin{equation}\label{eq:controlexcessT1}
  \left|\cR_{(\bar{Z}^{[Y]})}({g^{\Delta}})-  \cR_{(\bar{\xi}^{[Y]})}({g^{\Delta}})\right|\leq \mathfrak{C}\Nt^{-1},
\end{equation}
where the right hand side is independent of $\Delta$. 

Finally, since the test observations are independent of the training sample \(\mathcal D_N\) and conditionally on \(\mathcal D_N\), the classifier \(\widehat g\) is fixed. We can therefore apply the previous inequality \(\eqref{eq:controlexcessT1}\) to \(\widehat g\), which yields
\begin{align}
\E\big[ \cR_{(\bar{Z}^{[Y]})}(\w{g})\big]- \E\big[ \cR_{(\bar{\xi}^{[Y]})}(\w{g})\big]\leq \E\Big[ \E\Big[\,\Big|\cR_{(\bar{Z}^{[Y]})}(\w{g})-  \cR_{(\bar{\xi}^{[Y]})}(\w{g})\Big| \;\Big|\;\mathcal{D}_N\,\Big]\,\Big]\leq \mathfrak{C}\Nt^{-1}.
\end{align}

\begin{remark}[About the link between $\cR^*$ and $\cR^*_{\Nt}$]\label{rem:excessriskNtest}
First, consider an arbitrary measurable classifier $g: \mathcal{C}([0,1], \R)\rightarrow \cY$. By an argument similar to that used in the proof of Proposition \ref{prop:excessriskTerm13}, we have 
\begin{equation}\label{eq:continue-ntest-cvg}
 \left|\cR_{({Z}^{[Y]})}(g)-  \cR_{({\xi}^{[Y]})}({g})\right|\leq \mathfrak{C}\Nt^{-1}.   
\end{equation}
Indeed, for each $k\in\cY$, Proposition~\ref{prop:chaos}-(b) can be applied to the measurable function \[f:\alpha \in\mathcal{C}([0,1], \R)\mapsto f(\alpha)\coloneqq\one_{\{g(\alpha)\neq k\}}\in\{0,1\},\]
and the conclusion follows by conditioning on \(Y=k\) and summing over $k\in\cY$.

Now, let $g^*\in \argmin{g \in \mathcal{G}}\mathcal{R}_{(\xi^{[Y]})}(g)$  and $g^{*}_{\Nt} \in \argmin{g \in \mathcal{G}}\mathcal{R}_{({Z}^{[Y]})}(g)$. If $\mathcal{R}^*\geq\mathcal{R}^{*}_{\Nt} $, we have 
\begin{align}
    \left|\mathcal{R}^*-\mathcal{R}^*_ {\Nt}\right|&=\mathcal{R}^*-\mathcal{R}^*_{\Nt}=\mathcal{R}_{(\xi^{[Y]})}(g^{*})-\mathcal{R}_{({Z}^{[Y]})}(g^{*}_{\Nt})\nonumber\\
    &\leq \big|\mathcal{R}_{(\xi^{[Y]})}(g^{*}_{\Nt})-\mathcal{R}_{({Z}^{[Y]})}(g^*_{\Nt})\big|.\nonumber
\end{align}
Similarly, if $\mathcal{R}^*\leq\mathcal{R}^*_{\Nt}$,
\begin{align}
    \left|\cR^*-\cR^*_{\Nt}\right|&=\cR^*_{\Nt}-\cR^*=\cR_{({Z}^{[Y]})}(g^*_{\Nt})-\cR_{(\xi^{[Y]})}(g^{*})\nonumber\\
    &\leq \big|\cR_{(\xi^{[Y]})}(g^{*})-\cR_{({Z}^{[Y]})}(g^{*})\big|.\nonumber
\end{align}
We can then conclude by applying \eqref{eq:continue-ntest-cvg} to $g^*$ and $g_{N_{\mathrm{test}}^*}$, and obtain
$|\cR^{*}-\cR^{*}_{\Nt}| \leq \mathfrak{C} \Nt^{-1}.$
\end{remark}

\subsubsection{Proof to Proposition \ref{prop:bayes}}

The proof of Proposition \ref{prop:bayes} is divided into 2 steps: in the first step, we prove that there exists a reference probability measure $\P_0$ on the filtered probability space $(\Omega, \mathcal{F}, (\mathcal{F}_t)_{t\in[0,1]})$ and an $(\mathcal{F}_t)$-standard Brownian motion $(W^0_t)_{t\in[0,1]}$ under $\P_0$ such that the process 
\begin{equation}\label{eq:ref-process}
\wt{\xi}=(\wt{\xi}_t)_{t\in[0,1]}\coloneqq (x_0+W_t^0)_{t\in[0,1]}
\end{equation}  
is the unique strong solution to \eqref{eq:mkv} on $(\Omega, \mathcal{F}, (\mathcal{F}_t)_{t\in[0,1]}, \mathbb{P}_k)$; in the second step, we prove \eqref{eq:pi}, which is a reproduction from the proof of \cite[Proposition 1]{denis2020consistent} for reader's convenience. \\

\noindent {\sc Step 1.} Recall that Assumption \ref{ass:bsup_et_lip_in_x_and_y} implies the existence and strong uniqueness of the McKean-Vlasov equations \eqref{eq:mkv} given $Y=k\in\{1, ..., K\}$. For every fixed $k$,  let $(\mu_t^{[k]})_{t\in[0,1]}$ denote the marginal distributions of the unique solution $(\xi_t^{[k]})_{t\in[0,1]}$ of \eqref{eq:mkv} given $Y=k$, in the sense that $\mu_t^{[k]}=\P_k\circ (\xi_t^{[k]})^{-1},\;t\in[0,1]$. 

Consider a probability measure $\P_0$ on the filtered probability space $(\Omega, \mathcal{F}, (\mathcal{F}_t)_{t\in[0,1]})$ and an $(\mathcal{F}_t)$-standard Brownian motion $(W^0_t)_{t\in[0,1]}$ under $\P_0$. 
The process $(\Phi_{t}^{k})_{t\in[0,1]}$ defined by  
\begin{equation}\label{eq:def_Phi-func}
\Phi_{t}^{k}\coloneqq\mathcal{E}_t\Big(\int_{0}^{\cdot}b^*_k\big[\wt{\xi}_s, \mu_s^{[k]}\big]\dd W_s^0\;\Big), \;t\in[0,1],
\end{equation}
is a martingale \cite[Corollary 5.13]{karatzas1991brownian}, where $\mathcal{E}_t(M)\coloneqq\exp(M_t-\frac{1}{2}\langle M\rangle_t)$ for any continuous martingale $M$. For every $k\in\mathcal{Y}$, 
we define probability measures ${\P}_{k}$ on $(\Omega, \mathcal{F}, (\mathcal{F}_t)_{t\in[0,1]})$ as follows 
\begin{equation}\label{eq:defPk}
\frac{\dd {\P}_k}{\dd\P_0}\Big|_{\cF_t}\coloneqq \Phi_{t}^{k},\quad t\in[0,1].
\end{equation}
 By Girsanov's Theorem (see e.g. \cite[Theorem 5.1]{karatzas1991brownian}), the following process 
\begin{equation}\label{eq:def-change-BM}
W^{[k]}_t\coloneqq W_t^0-\int_{0}^{t} b^*_k\big[ \wt{\xi}_s, \mu_s^{[k]}\big] \dd s,\quad t\in[0,1]
\end{equation}
is a Brownian motion on  $(\Omega, \mathcal{F}, (\mathcal{F}_t)_{t\in[0,1]}, {\P}_k)$ and 
\begin{equation}
\wt{\xi}_t=x_0+W_t^{0}=x_0+\int_{0}^{t}b^*_k\big[\wt{\xi}_s, \mu_s^{[k]}\big]\dd s + W_t^{[k]}, \;t\in[0,1].
\end{equation}
Hence, $(\wt{\xi}_t)_{t\in[0,1]}$ is the unique solution to \eqref{eq:mkv} under $\P_k$ due to the strong uniqueness of the solution to this equation. Remark that the probability measures $\P_0$ and ${\P}_k$ are 
mutually absolutely continuous in the sense that 
\[\forall\, F\in\mathcal{F}_T, \quad {\P}_k(F)=\E_{\P_0}\big[\Phi_{1}^{k}\one_{F}\big] \quad \text{and}\quad \P_0(F)= \E_{{\P}_k}\big[(\Phi_{1}^{k})^{-1}\one_{F}\big] \]
since we consider a finite time horizon $T=1$ and the density $\frac{d{\P}_k|_{\cF_1}}{d\P_0|_{\cF_1}}$ is strictly positive. Consequently, the definition of $\P_0$ can also be determined from $\P_k$, see e.g. the classical textbook \cite[Section 5.5]{legall2016brownian}. Now let ${\P}(\cdot):=\sum_{k=1}^{K}\dfrac{1}{K} {\P}_k(\cdot)$. The probability measures $\P_0$ and ${\P}$ are also
mutually absolutely continuous with the Radon-Nikodym derivative 
    $$\frac{\dd{\P}}{\dd\P_0}\Big|_{\cF_t}=\sum_{k=1}^{K}\dfrac{1}{K} \Phi_{t}^{k} $$
    and using Equations \eqref{eq:def_Phi-func} and \eqref{eq:def-change-BM},
\begin{align}\label{eq:RNderi}
\frac{\dd{\P}}{\dd\P_0}\Big|_{\cF_t}
&=\sum_{k=1}^{K}\dfrac{1}{K} \exp\left( \int_{0}^{t}b^*_k\big[ \wt{\xi}_s, \mu_s^{[k]}\big]\dd W_s^0-\frac{1}{2}\int_{0}^{t}\left(b^*_k\big[ \wt{\xi}_s, \mu_s^{[k]}\big]\right)^2\dd s \right)\nonumber\\
&=\sum_{k=1}^{K}\frac{1}{K} \exp\left( \int_{0}^{t}b^*_k\big[ \wt{\xi}_s, \mu_s^{[k]}\big]\dd W_s^{[k]}+\frac{1}{2}\int_{0}^{t}\left(b^*_k\big[ \wt{\xi}_s, \mu_s^{[k]}\big]\right)^2\dd s \right)\qquad 
\end{align} which is also strictly positive. Remark that the Brownian motion $W=(W_t)_{t\in[0,1]}$ in \eqref{eq:mkv} under $\P$ is still a Brownian motion under $\P_k$ as we assume that $W$ and the label $Y$ are independent.  Consequently, given the probability measure $\P$ on the space $(\Omega, \mathcal{F}, (\mathcal{F}_t)_{t\in[0,1]})$, the standard Brownian motion $(W_t)_{t\in[0,1]}$ under $\P$ and the unique solution $\xi=({\xi}_t)_{t\in[0,1]}$ to \eqref{eq:mkv}, the reference probability $\P_0$ and the Brownian motion $W^0=(W_t^0)_{t\in[0,1]}$ discussed above can be identified by using \eqref{eq:RNderi} (see also e.g. \cite[Section 5.5]{legall2016brownian}). 
\\

\noindent {\sc Step 2.} 
 Now we define for every $k\in\cY$, 
\begin{equation}\label{eq:defpsi}
\Psi^{k}_t\coloneqq \frac{\dd \P_k|_{\cF_t}}{\dd\P|_{\cF_t}}=\frac{\Phi_t^k\;\dd \P_0}{\sum_{k'=1}^{K}(1/K)\Phi_t^{k'}\;\dd\P_0}=\frac{\Phi_t^k}{\sum_{k'=1}^{K}(1/K)\Phi_t^{k'}},\quad t\in[0,1]
\end{equation}
with  $$\Phi_t^k=\exp\left( \int_{0}^{t}b^*_k\big[ \wt{\xi}_s, \mu_s^{[k]}\big]\dd W_s^{[k]}+\frac{1}{2}\int_{0}^{t}\left(b^*_k\big[ \wt{\xi}_s, \mu_s^{[k]}\big]\right)^2\dd s \right).$$  
Consequently, $\Phi_1^k= \exp \big(F^*_k(\xi)\big).$
Let us denote 
$\mathcal{F}_t^\xi= \sigma(\xi_s, s \in [0,t])$, $t \in [0,1]$.
For a bounded function $h: \cY\rightarrow \R$ and an $\cF_t^\xi$-measurable bounded random variable $Z$, we have, 
\begin{align}
\E \big[h(Y)Z\big]&=\E \big[h(Y)\E[Z\;|\;Y]\big]=\E \left[ \sum_{k=1}^{K}h(k)\mathbf{1}_{\{Y=k\}}\E_k[Z]\right]=\E \left[ \sum_{k=1}^{K}h(k)\mathbf{1}_{\{Y=k\}}\E[Z\Psi_t^k]\right]\nonumber\\
&=\dfrac{1}{K}\sum_{k=1}^{K}h(k)\E[Z\Psi_t^k]=\dfrac{1}{K}\E \left[ \Big( \sum_{k=1}^{K}h(k)\Psi_t^k\Big)Z\right],
\end{align}
which implies 
\[\E\big[\,h(Y)\,\big|\,\mathcal{F}_1^\xi\,\big]=\sum_{k=1}^{K}(1/K) h(k)\Psi^{k}_1 = \frac{\sum_{k=1}^{K}(1/K) h(k)\Phi^{k}_1}{\sum_{k'=1}^{K}(1/K)\Phi_1^{k'}},\quad \P-\text{a.s.}\]
by using the definition of $\Psi_{t}^{k}$ in \eqref{eq:defpsi}. 
Then we conclude the proof by taking $h(x)=\one_{\{k\}}(x), \;k\in\cY$. 


\subsubsection{Proof of Theorem \ref{theo:excessrisk}}\label{subsec:proof-thm-4.1}

Recall that $\xi^{[Y]}$ denotes the stochastic process defined in \eqref{eq:mkv}. 
When no ambiguity arises, we omit the superscript $^{[Y]}$ and simply write $\xi$. 
We also recall the projection $\mathrm{Proj}_{\Delta}$ defined in~\eqref{eq:def-projection-delta}.  By a slight abuse of notation, for any $\alpha\in\mathcal C([0,1],\R)$, we write
$\widehat F_k(\alpha)$, $\widehat \pi_k(\alpha)$, and $\widehat g(\alpha)$
for $\widehat F_k(\mathrm{Proj}_{\Delta}(\alpha))$,
$\widehat \pi_k(\mathrm{Proj}_{\Delta}(\alpha))$, and
$\widehat g(\mathrm{Proj}_{\Delta}(\alpha))$, respectively.

We have the following inequality by \cite[Proposition 2]{denis2020consistent}
$$
\E\left[\mathcal{R}_{(\bar{\xi}^{[Y]})}(\w{g}) - {\mathcal{R}^*}\right]    
\leq 2 \sum_{k=1}^K \E \left[ \left|\w{\pi}_k(\xi^{[Y]})-\pi^*_k(\xi^{[Y]})\right|\right],
$$  
then, let us denote $\pi_k^*(\xi)= \phi_k(F(\xi)):={\exp(F_k(\xi))}/{\sum_{k'}\exp(F_{k'}(\xi))}$. Thus,
$$
\E\left[\mathcal{R}_{(\xi^{[Y]})}(\w{g}) - \mathcal{R}^*\right]   
\leq  2\sum_{k=1}^K \E \left[ \left|\phi_k({\w{ F}}(\xi^{[Y]}))-\phi_k({ F}^*(\xi^{[Y]}))\right|\right]
$$
Since the softmax function $\phi_k$ is $1$-Lipschitz, we have for $k \in \mathcal{Y}$
\begin{equation*}
\E \left|\phi_k({{\w{ F}}}(\xi^{[Y]}))-\phi_k({F}(\xi^{[Y]}))\right| \leq 
\sum_{k'=1}^K\E\left[\left|{\w{F}}_{k'}(\xi^{[Y]})-{F}^*_{k'}(\xi^{[Y]}) \right|\right].  
\end{equation*}
We remind the reader that using that 
$\displaystyle \int_0^1 b_k^*[\xi_s^{[Y]}, \mu^{[k]}_s]\dd \xi_s^{[Y]} = \int_0^1 \int_\R b^*_k(\xi^{[Y]}_s, y)\mu^{[k]}_s(\dd y)\dd \xi_s^{[Y]},$
it yields
\begin{eqnarray}\label{eq:Fstar}
F^*_k(\xi^{[Y]})&:=&\int_{0}^{ 1}b_k^{*}[ \xi^{[Y]}_s, \mu_s^{[k]}]\dd \xi^{[Y]}_s-\frac{1}{2}\int_{0}^{1}{b^{*2}_{k}[ \xi^{[Y]}_s, \mu_s^{[k]}]\dd s}\nonumber\\
&=&\int_{0}^{ 1}\int_\R b_k^{*}( \xi^{[Y]}_s, y)\mu^{[k]}_s(\dd y)\dd \xi^{[Y]}_s-\frac{1}{2}\int_{0}^{1}\left(\int_{\R}b^{*}_{k}(\xi^{[Y]}_s, y)\mu^{[k]}_s(\dd y)\right)^2\dd s.
\end{eqnarray}
Then, we define the empirical counterpart of $F^*_k$ for the discretization of the time interval,
\begin{eqnarray}\label{eq:Fstarbar}
\bar{F}^*_k(\xi^{[Y]})&:=&\sum_{m=0}^{M-1}b_{k}^{*}\left[ \xi^{[Y]}_{m\Delta}, \mu^{[k]}_{m\Delta} \right](\xi^{[Y]}_{(m+1)\Delta}-\xi^{[Y]}_{m\Delta})-\frac{\Delta}{2}\sum_{m=0}^{M-1}b^{*2}_{k}\left[\xi^{[Y]}_{m\Delta}, \mu^{[k]}_{m\Delta}\right]\\
&=&\sum_{m=0}^{M-1}\int_{\R} b_{k}^{*}( \xi^{[Y]}_{m\Delta}, y)\mu_{m\Delta}^{[k]}(\dd y)\big(\xi^{[Y]}_{(m+1)\Delta}-\xi^{[Y]}_{m\Delta}\big)-\frac{\Delta}{2}\sum_{m=0}^{M-1}\left(\int_{\R} b_{k}^{*}(\xi^{[Y]}_{m\Delta}, y)\mu_{m\Delta}^{[k]}(\dd y)\right)^2.\nonumber
\end{eqnarray}
Then, we also introduce 
\begin{equation}\label{eq:Fhat}
\wt{F}^*_k(\xi^{[Y]}):=\frac{1}{N}\sum_{n=1}^N\sum_{m=0}^{M-1} b_k^{*}( \xi_{m\Delta}^{[Y]}, \wt{X}_{m\Delta}^{[k],n})(\xi^{[Y]}_{(m+1)\Delta}-\xi^{[Y]}_{m\Delta})-\frac{\Delta}{2}\sum_{m=0}^{M-1}  \left(\frac{1}{N}\sum_{n=1}^N b^{*}_{k}(\xi^{[Y]}_{m\Delta}, \wt{X}^{[k],n}_{m\Delta})\right)^2.
\end{equation}

Finally, to approximate $\wt{F}_k^*$, we estimate $b^*_k$, and obtain {$\w{F}_k$} given in Equation \eqref{eq:Fhathat}.
We decompose the risk as follows,
\begin{eqnarray*}
\left|{\w{F}}_k(\xi^{[Y]})-{F}^*_k(\xi^{[Y]}) \right| &\leq&  \left|{\w{F}}_k(\xi^{[Y]})-\wt{F}^*_k(\xi^{[Y]}) \right|+
\left|\wt{F}^*_k(\xi^{[Y]})- \bar{F}_k^*(\xi^{[Y]}) \right| +\left|\bar{F}_k^*(\xi^{[Y]})- F_k^*(\xi^{[Y]}) \right| \\
&=:& T_1+T_2+T_3.
\end{eqnarray*}
The last term $T_3$ represents the cost of the discretization in time, and $T_2$ represents the cost of the approximation of the true marginal distribution  $\mu^{[k]}_{m\Delta}$ by the empirical one, and $T_1$ is the estimation cost of the function $b_k^*$.

\paragraph{Study of $T_1$.}
 \begin{eqnarray*}
T_1&=&\left|{\w{F}}_k(\xi^{[Y]})-\wt{F}^*_k(\xi^{[Y]}) \right| \leq  \left| \frac{1}{N}\sum_{n=1}^N \sum_{m=0}^{M-1} \left(\w{b}_k(\xi^{[Y]}_{m\Delta},\wt{X}^{[k],n}_{m\Delta})- b_k^{*}( \xi^{[Y]}_{m\Delta}, \wt{X}^{[k],n}_{m\Delta})\right)(\xi^{[Y]}_{(m+1)\Delta}-\xi^{[Y]}_{m\Delta})
\right|\\
&&+\frac{\Delta}{2}
\left|
\sum_{m=0}^{M-1}  \left(\frac{1}{N}\sum_{n=1}^N \w{b}_{k}(\xi^{[Y]}_{m\Delta}, \wt{X}^{[k],n}_{m\Delta})\right)^2
-\sum_{m=0}^{M-1}  \left(\frac{1}{N}\sum_{n=1}^N b^{*}_{k}(\xi^{[Y]}_{m\Delta}, \wt{X}^{[k],n}_{m\Delta})\right)^2
\right|
\end{eqnarray*}
Let use here that for any function $f$,
$$\sum_{m=0}^{M-1} f(\xi^{[Y]}_{m\Delta})(\xi^{[Y]}_{(m+1)\Delta}-\xi^{[Y]}_{m\Delta})= \int_0^1 f(\xi^{[Y]}_{\eta(s)})\dd \xi^{[Y]}_s, \quad 
\Delta\sum_{m=0}^{M-1} f(\xi^{[Y]}_{m\Delta})= \int_0^1 f(\xi^{[Y]}_{\eta (s)})\dd s
$$
with $\eta (s)=m\Delta$ if $m\Delta \leq s < (m+1)\Delta$. Then, 
\begin{align*}
&T_1\leq  \left| \frac{1}{N}\sum_{n=1}^N \int_0^1 \left(\w{b}_k(\xi^{[Y]}_{\eta (s)},\wt{X}_{\eta (s)}^{[k],n})- b_k^{*}( \xi^{[Y]}_{\eta (s)}, \wt{X}_{\eta (s)}^{[k],n})\right)\dd \xi^{[Y]}_s\right|\\
&\quad+\frac{1}{2}\left| \int_0^1 \left(\frac{1}{N} \sum_{n=1}^N {b^{*}_{k}(\xi^{[Y]}_{\eta (s)}, \wt{X}^{[k],n}_{\eta (s)})}-\w{b}_k(\xi^{[Y]}_{\eta (s)}, \wt{X}^{[k],n}_{\eta (s)}) \right) \left( \frac{1}{N} \sum_{n=1}^N {b^{*}_{k}(\xi^{[Y]}_{\eta (s)}, \wt{X}^{[k],n}_{\eta (s)})}+\w{b}_k(\xi^{[Y]}_{\eta (s)}, \wt{X}^{[k],n}_{\eta (s)})\right)\dd s \right|.
\end{align*}
Furthermore as $\xi$ has a dynamic described in Equation \eqref{eq:mkv}, 
%
\begin{align*}
&T_1\leq  \left| \frac{1}{N}\sum_{n=1}^N \int_0^1 \left(\w{b}_k(\xi^{[Y]}_{\eta (s)},\wt{X}^{[k],n}_{\eta (s)})- b_k^{*}( \xi^{[Y]}_{\eta (s)}, \wt{X}^{[k],n}_{\eta (s)})\right) {b_Y^*[\xi^{[Y]}_s, \mu_s^{[Y]}] \dd s}\right|\\
&\quad +\left| \frac{1}{N}\sum_{n=1}^N \int_0^1 \left(\w{b}_k(\xi^{[Y]}_{\eta (s)},\wt{X}^{[k],n}_{\eta (s)})- b_k^{*}( \xi^{[Y]}_{\eta (s)}, \wt{X}^{[k],n}_{\eta (s)})\right) \dd W_s\right|\\
&\quad +\frac{1}{2}\left| \int_0^1 \left(\frac{1}{N} \sum_{n=1}^N {b^{*}_{k}(\xi^{[Y]}_{\eta (s)}, \wt{X}^{[k],n}_{\eta (s)})}-\w{b}_k(\xi^{[Y]}_{\eta (s)}, \wt{X}^{[k],n}_{\eta (s)}) \right) \left( \frac{1}{N} \sum_{n=1}^N {b^{*}_{k}(\xi^{[Y]}_{\eta (s)}, \wt{X}^{[k],n}_{\eta (s)})}+\w{b}_k(\xi^{[Y]}_{\eta (s)}, \wt{X}^{[k],n}_{\eta (s)})\right)\dd s \right|.
\end{align*}
Then, 
\begin{eqnarray}\label{eq:eqT1Utile1}
\E[T_1]&\leq & \int_0^1 \E\left[ \left( \frac{1}{N}\sum_{n=1}^N  \w{b}_k(\xi^{[Y]}_{\eta (s)},\wt{X}_{\eta (s)}^{[k],n})- b_k^{*}( \xi^{[Y]}_{\eta (s)}, \wt{X}_{\eta (s)}^{[k],n})\right)^2\right]^{1/2} \E\left[ b_Y^{*2}[\xi^{[Y]}_s,\mu^{[Y]}_s]\right]^{1/2} \dd s \nonumber\\ 
&&+\E\left[\int_0^1 \left(\frac{1}{N}\sum_{n=1}^N \w{b}_k(\xi^{[Y]}_{\eta (s)},\wt{X}^{[k],n}_{\eta (s)})- b_k^{*}( \xi^{[Y]}_{\eta (s)}, \wt{X}^{[k],n}_{\eta (s)})\right)^2 \dd s\right]^{1/2} \nonumber\\ 
&&+\frac{1}{2}\E\left[\left| \int_0^1 \left(\frac{1}{N} \sum_{n=1}^N {b^{*}_{k}(\xi^{[Y]}_{\eta (s)}, \wt{X}^{[k],n}_{\eta (s)})}-\w{b}_k(\xi^{[Y]}_{\eta (s)}, \wt{X}^{[k],n}_{\eta (s)}) \right) \right.\right.\nonumber\\
&&\hspace{2cm}\left.\left.\times\left( \frac{1}{N} \sum_{n=1}^N {b^{*}_{k}(\xi^{[Y]}_{\eta (s)}, \wt{X}^{[k],n}_{\eta (s)})}+\w{b}_k(\xi^{[Y]}_{\eta (s)}, \wt{X}^{[k],n}_{\eta (s)})\right)\right| \dd s\right] \nonumber \\
&=& T_a+ T_b+\frac{1}{2}T_c.
\end{eqnarray}
Since $b^*_k$ is bounded for $k\in \cY$ and $\w{b}_k$ is truncated,
we then deduce that
$$T_a \leq {\mathfrak{C}}
\int_0^1 \E\left[ \left( \frac{1}{N}\sum_{n=1}^N  \w{b}_k(\xi^{[Y]}_{\eta (s)},\wt{X}^{[k],n}_{\eta (s)})- b_k^{*}( \xi^{[Y]}_{\eta(s)}, \wt{X}^{[k],n}_{\eta (s)})\right)^2\right]^{1/2}
\dd s,
$$
and 
$$T_b \leq \E\left[\int_0^1 \left(\frac{1}{N}\sum_{n=1}^N \w{b}_k(\xi^{[Y]}_{\eta (s)},\wt{X}^{[k],n}_{\eta (s)})- b_k^{*}( \xi^{[Y]}_{\eta (s)}, \wt{X}^{[k],n}_{\eta (s)})\right)^2 \dd s\right]^{1/2}.$$
Thus from the above inequality we deduce that 
\begin{equation}\label{eq:eqT1Utile2}
\E[T_a]\leq  {\mathfrak{C}}\, \E[\|b_k^*-\w{b}_k\|_{M,k}], \;\; {\rm and} \;\; \E[T_b]\leq  \E[\|b_k^*-\w{b}_k\|_{M,k}]. %
\end{equation}
Furthermore, since for $N$ large enough, and each $k \in \mathcal{Y}$, we have from Assumption~\ref{ass:bsup_et_lip_in_x_and_y} that $\left\|b_k^*\right\|_{\infty} \leq \sqrt{\log(N)}$, and $\left\|\w{b}_k\right\|_{\infty} \leq \sqrt{\log(N)}$ by definition in Equation~\eqref{eq:bhattronc}, we deduce from Equation~\eqref{eq:eqT1Utile1} that 
\begin{eqnarray*}
T_c 
&\leq & \log^{1/2}(N)\int_0^1 \E\left[ \left(\frac{1}{N} \sum_{n=1}^N {b^{*}_{k}(\xi^{[Y]}_{\eta (s)}, \wt{X}^{[k],n}_{\eta (s)})}-\w{b}_k(\xi^{[Y]}_{\eta (s)}, \wt{X}^{[k],n}_{\eta (s)}) \right)^2\right]^{1/2} \mathrm{d}s.
\end{eqnarray*}
Thus from the above inequality we deduce that 
\begin{equation}\label{eq:eqT1Utile3}
 \E[T_c]\leq {2}\log^{1/2}(N)  \E[\|b_k^*-\w{b}_k\|_{M,k}].
\end{equation}
Finally, from Equation~\eqref{eq:eqT1Utile1}, ~\eqref{eq:eqT1Utile2}, and~\eqref{eq:eqT1Utile3}, we get
\begin{equation*}
\mathbb{E}\left[T_1\right] \leq  {\mathfrak{C}}  \log^{1/2}(N) \E[\|b_k^*-\w{b}_k\|_{M,k}].  
\end{equation*}

\paragraph{Study of $T_2=\big|{\wt{F}}_k^{*}(\xi^{[Y]})-\bar{F}_k^*(\xi^{[Y]}) \big|$.}

The term $T_2$ corresponds to the error arising from approximating the measure arguments 
$\mu_{m\Delta}^{[k]}$, $k \in \mathcal{Y}$, by the empirical measures
\[
\widetilde{\mu}_{m\Delta}^{[k],N}
:=
\frac{1}{N}\sum_{n=1}^N \delta_{\widetilde{X}_{m\Delta}^{[k],n}},
\qquad k \in \mathcal{Y},
\]
where $\widetilde{X}_{m\Delta}^{[k],n}$, $k \in \mathcal{Y}$, $1 \leq n \leq N$, are defined by the particle system~\eqref{eq:sys}.
Recall that ${\wt{F}}_k^{ *}(\xi^{[Y]})$ is given in Equation~\eqref{eq:Fhat}, and we can rewrite it as follows:
\begin{equation}\label{eq:Fhat-rewrite}
\wt{F}^*_k(\xi^{[Y]}):=\sum_{m=0}^{M-1} b_k^{*} \left[\xi^{[Y]}_{m\Delta}, \widetilde{\mu}_{m\Delta}^{[k],N}\right](\xi^{[Y]}_{(m+1)\Delta}-\xi^{[Y]}_{m\Delta})-\frac{\Delta}{2}\sum_{m=0}^{M-1}  \left( b^{*}_{k}\left[\xi^{[Y]}_{m\Delta}, \widetilde{\mu}_{m\Delta}^{[k],N}\right]\right)^2.
\end{equation}
Hence, 
\begin{align}
&\left|\wt{F}_k^*(\xi^{[Y]})-\bar{F}_k^*(\xi^{[Y]})\right|\nonumber\\
&\quad \leq \left|\sum_{m=0}^{M-1}  \left (b_{k}^{*}\left[ \xi^{[Y]}_{m\Delta}, \widetilde{\mu}_{m\Delta}^{[k],N}\right]- b_{k}^{*}\left[ \xi^{[Y]}_{m\Delta}, \mu_{m\Delta}^{[k]}\right] \right)\big(\xi^{[Y]}_{(m+1)\Delta}-\xi^{[Y]}_{m\Delta}\big)\right|\nonumber\\
& \qquad +\frac{\Delta}{2}\sum_{m=0}^{M-1} \left|b_{k}^{*2}\left[ \xi^{[Y]}_{m\Delta}, \widetilde{\mu}_{m\Delta}^{[k],N}\right]-b_{k}^{*2}\left[\xi^{[Y]}_{m\Delta}, \mu_{m\Delta}^{[k]}\right]\right|\nonumber\\
&\quad \leq \left|\int_{0}^{1} \left (b_{k}^{*}\left[ \xi^{[Y]}_{\eta(s)}, \widetilde{\mu}_{\eta(s)}^{[k],N}\right]- b_{k}^{*}\left[ \xi^{[Y]}_{\eta(s)}, \mu_{\eta(s)}^{[k]}\right] \right)\dd \xi^{[Y]}_{s}\right|\label{eq:T2improve}\\
&\qquad+\frac{\Delta}{2}\sum_{m=0}^{M-1} \left|b_{k}^{*}\left[ \xi^{[Y]}_{m\Delta}, \widetilde{\mu}_{m\Delta}^{[k],N}\right]-b_{k}^{*}\left[ \xi^{[Y]}_{m\Delta}, \mu_{m\Delta}^{[k]}\right]\right|\left|b_{k}^{*}\left[ \xi^{[Y]}_{m\Delta}, \widetilde{\mu}_{m\Delta}^{[k],N}\right]+b_{k}^{*}\left[ \xi^{[Y]}_{m\Delta}, \mu_{m\Delta}^{[k]}\right]\right|,\nonumber
\end{align}
where $\eta(s)\coloneqq m\Delta$ if $s\in[m\Delta, (m+1)\Delta)$ for every $m\in\{0,..., M-1\}$. 
The expectation of the first term in \eqref{eq:T2improve} admits the following upper bound.
\begin{align}
&\E \left[\left|\int_{0}^{1} \left (b_{k}^{*}\left[ \xi^{[Y]}_{\eta(s)},\widetilde{\mu}_{\eta(s)}^{[k],N}\right]- b_{k}^{*}\left[ \xi^{[Y]}_{\eta(s)}, \mu_{\eta(s)}^{[k]}\right] \right)\dd \xi_{s}\right|\right]\nonumber\\
&\quad \leq \E \left[\left|\int_{0}^{1} \left (b_{k}^{*}\left[ \xi^{[Y]}_{\eta(s)}, \widetilde{\mu}_{\eta(s)}^{[k],N}\right]- b_{k}^{*}\left[ \xi^{[Y]}_{\eta(s)}, \mu_{\eta(s)}^{[k]}\right] \right)b\left[\xi^{[Y]}_s, \mu_s^{[Y]}\right]\dd s \right|\right]\nonumber\\
&\qquad + \E \left[\left|\int_{0}^{1} \left (b_{k}^{*}\left[ \xi^{[Y]}_{\eta(s)}, \widetilde{\mu}_{\eta(s)}^{[k],N}\right]- b_{k}^{*}\left[ \xi^{[Y]}_{\eta(s)}, \mu_{\eta(s)}^{[k]}\right] \right)\dd W_s\right|\right]\nonumber\\
&\quad \leq b^*_{\max} \int_{0}^{1} \E\left[\left| b_{k}^{*}\left[ \xi^{[Y]}_{\eta(s)}, \widetilde{\mu}_{\eta(s)}^{[k],N}\right]- b_{k}^{*}\left[ \xi^{[Y]}_{\eta(s)}, \mu_{\eta(s)}^{[k]}\right] \right|\right]\dd s \nonumber\\
&\qquad + \left\Vert\int_{0}^{1} \left (b_{k}^{*}\left[ \xi^{[Y]}_{\eta(s)}, \widetilde{\mu}_{\eta(s)}^{[k],N}\right]- b_{k}^{*}\left[ \xi^{[Y]}_{\eta(s)}, \mu_{\eta(s)}^{[k]}\right] \right)\dd W_s\right\Vert_2\nonumber\\
&\quad \leq b^*_{\max}L\int_{0}^{1} \E\left[\mathcal{W}_1\left(\widetilde{\mu}_{\eta(s)}^{[k],N}, \mu_{\eta(s)}^{[k]}\right) \right]\dd s \nonumber\\
& \qquad + C\left\{\E \left[ \int_{0}^{1}\left (b_{k}^{*}\left[ \xi^{[Y]}_{\eta(s)}, \widetilde{\mu}_{\eta(s)}^{[k],N}\right]- b_{k}^{*}\left[ \xi^{[Y]}_{\eta(s)}, \mu_{\eta(s)}^{[k]}\right] \right)^2\dd s\right]\right\}^{\frac{1}{2}},\nonumber
\end{align}
where the second term in the third inequality above follows from the Burkholder-Davis-Gundy inequality (see, e.g., Section 7.8.1  and (7.53) in \cite{pages-numerical-proba}).

For the term $\E \left[ \int_{0}^{1}\left (b_{k}^{*}\left[ \xi^{[Y]}_{\eta(s)}, \widetilde{\mu}_{\eta(s)}^{[k],N}\right]- b_{k}^{*}\left[ \xi^{[Y]}_{\eta(s)}, \mu_{\eta(s)}^{[k]}\right] \right)^2\dd s\right]$ in the above inequality, we have 
\begin{align}
&\E \left[ \int_{0}^{1}\left (b_{k}^{*}\left[ \xi^{[Y]}_{\eta(s)}, \widetilde{\mu}_{\eta(s)}^{[k],N}\right]- b_{k}^{*}\left[ \xi^{[Y]}_{\eta(s)}, \mu_{\eta(s)}^{[k]}\right] \right)^2\dd s\right]=\int_{0}^{1}\E \left[ \left (b_{k}^{*}\left[ \xi^{[Y]}_{\eta(s)}, \widetilde{\mu}_{\eta(s)}^{[k],N}\right]- b_{k}^{*}\left[ \xi^{[Y]}_{\eta(s)}, \mu_{\eta(s)}^{[k]}\right] \right)^2\right]\dd s\nonumber\\
&= \int_{0}^{1}\E \left[ \left (\frac{1}{N}\sum_{n=1}^Nb_{k}^{*}\left( \xi^{[Y]}_{\eta(s)}, \widetilde{X}_{\eta(s)}^{[k],n}\right)-\int_{\R} b_{k}^{*}\left( \xi^{[Y]}_{\eta(s)}, y\right) \mu_{\eta(s)}^{[k]}(\dd y)\right)^2\right]\dd s\nonumber\\
&=\int_{0}^{1}\E \left[ \E \left[ \left (\frac{1}{N}\sum_{n=1}^Nb_{k}^{*}\left( \xi^{[Y]}_{\eta(s)}, \widetilde{X}_{\eta(s)}^{[k],n}\right)-\int_{\R} b_{k}^{*}\left( \xi^{[Y]}_{\eta(s)}, y\right) \mu_{\eta(s)}^{[k]}(\dd y)\right)^2 \;\Bigg|\;  \xi^{[Y]}_{\eta(s)}\right]\right]\dd s\nonumber\\
&\leq \mathfrak{C}N^{-1},
\end{align}
where the last inequality comes from Lemma \ref{lem:application-TV-distance-and-Bernstein}-$(c)$. Combining with Proposition \ref{prop:chaos}-(a)for the term $\E\left[\mathcal{W}_1\left(\widetilde{\mu}_{\eta(s)}^{[k],N}, \mu_{\eta(s)}^{[k]}\right) \right]$, we obtain, 
\begin{align}
&\E \left[\left|\int_{0}^{1} \left (b_{k}^{*}\left[ \xi^{[Y]}_{\eta(s)},\widetilde{\mu}_{\eta(s)}^{[k],N}\right]- b_{k}^{*}\left[ \xi^{[Y]}_{\eta(s)}, \mu_{\eta(s)}^{[k]}\right] \right)\dd \xi_{s}\right|\right]\leq \mathfrak{C} N^{-\frac{1}{2}}.
\end{align}
Moreover, by applying Proposition~\ref{prop:chaos}, we obtain the following upper bound for the expectation of the second term in~\eqref{eq:T2improve},
\begin{align}
&\E\left[ \frac{\Delta}{2}\sum_{m=0}^{M-1} \left|b_{k}^{*}\left[ \xi^{[Y]}_{m\Delta},\widetilde{\mu}_{m\Delta}^{[k],N}\right]-b_{k}^{*}\left[ \xi^{[Y]}_{m\Delta}, \mu_{m\Delta}^{[k]}\right]\right|\left|b_{k}^{*}\left[ \xi^{[Y]}_{m\Delta}, \widetilde{\mu}_{m\Delta}^{[k],N}\right]+b_{k}^{*}\left[ \xi^{[Y]}_{m\Delta}, \mu_{m\Delta}^{[k]}\right]\right|\right]\nonumber\\
&\leq\frac{\Delta}{2}\sum_{m=0}^{M-1}  L\E\left[\mathcal{W}_1\Big(\widetilde{\mu}_{m\Delta}^{[k],N}, \mu_{m\Delta}^{[k]} \Big)\right]\cdot 2 b^*_{\max}\leq L b^*_{\max}\mathfrak{C} N^{-\frac{1}{2}}.
\end{align}
Finally, 
\begin{align}
\E[T_2]=\E \left[\left|\wt{F}_k^*(\xi)-\bar{F}_k^*(\xi)\right| \right]\leq \mathfrak{C}N^{-\frac{1}{2}}.
\end{align}


\paragraph{Study of $T_3=\big|\bar{F}_k^*(\xi^{[Y]})-F_k^*(\xi^{[Y]})\big|  $.}

First, we have
\begin{eqnarray}\label{maj1}
\left|F_k^{*}(\xi^{[Y]})-\bar{F}_k^*(\xi^{[Y]})\right|
&\leq &\left| \sum_{m=0}^{M-1}\int_{m\Delta}^{(m+1)\Delta} b_{k}^{*}\left[\xi^{[Y]}_s, \mu_{s}^{[k]}\right]\dd \xi^{[Y]}_s- \sum_{m=0}^{M-1} b_{k}^{*}\left[\xi^{[Y]}_{m\Delta}, \mu_{m\Delta}^{[k]}\right]\big(\xi^{[Y]}_{(m+1)\Delta}-\xi^{[Y]}_{m\Delta}\big)\right|\nonumber\\
&& + \frac{1}{2}\left| \int_{0}^{1}(b_{k}^{*})^2\left[\xi^{[Y]}_s,\mu_{s}^{[k]}\right]\dd s
- \sum_{m=0}^{M-1}\Delta (b_{k}^{*})^2\left[\xi^{[Y]}_{m\Delta}, \mu_{m\Delta}^{[k]}\right]\right|.
\end{eqnarray}
Recall that for every $s\in[m\Delta, (m+1)\Delta)$, $0\leq m\leq M-1$, we denote $\eta(s)\coloneqq m\Delta$. 
The first term in~\eqref{maj1} can be upper-bounded as follows:
\begin{align*}
&\E\left[
\left|
\sum_{m=0}^{M-1}\int_{m\Delta}^{(m+1)\Delta}
b_k^*\left[\xi_s^{[Y]},\mu_s^{[k]}\right]\dd\xi_s^{[Y]}
-
\sum_{m=0}^{M-1}
b_k^*\left[\xi_{m\Delta}^{[Y]},\mu_{m\Delta}^{[k]}\right]
\left(\xi_{(m+1)\Delta}^{[Y]}-\xi_{m\Delta}^{[Y]}\right)
\right|
\right]
\\
&=
\E\left[
\left|
\int_0^1
\left(
b_k^*\left[\xi_s^{[Y]},\mu_s^{[k]}\right]
-
b_k^*\left[\xi_{\eta(s)}^{[Y]},\mu_{\eta(s)}^{[k]}\right]
\right)
\dd\xi_s^{[Y]}
\right|
\right].
\end{align*}
Since
\[
\dd\xi_s^{[Y]}
=
b_Y^*\left[\xi_s^{[Y]},\mu_s^{[Y]}\right]\dd s+\dd W_s,
\]
we obtain
\begin{align*}
&\E\left[
\left|
\int_0^1
\left(
b_k^*\left[\xi_s^{[Y]},\mu_s^{[k]}\right]
-
b_k^*\left[\xi_{\eta(s)}^{[Y]},\mu_{\eta(s)}^{[k]}\right]
\right)
\dd\xi_s^{[Y]}
\right|
\right]
\\
&\leq
b_{\max}^*
\int_0^1
\E\left[
\left|
b_k^*\left[\xi_s^{[Y]},\mu_s^{[k]}\right]
-
b_k^*\left[\xi_{\eta(s)}^{[Y]},\mu_{\eta(s)}^{[k]}\right]
\right|
\right]\dd s +
\E\left[
\left|
\int_0^1
\left(
b_k^*\left[\xi_s^{[Y]},\mu_s^{[k]}\right]
-
b_k^*\left[\xi_{\eta(s)}^{[Y]},\mu_{\eta(s)}^{[k]}\right]
\right)
\dd W_s
\right|
\right].
\end{align*}
By the Lipschitz continuity of $b_k^*$ and It\^o's isometry,
\begin{align*}
&\E\left[
\left|
\int_0^1
\left(
b_k^*\left[\xi_s^{[Y]},\mu_s^{[k]}\right]
-
b_k^*\left[\xi_{\eta(s)}^{[Y]},\mu_{\eta(s)}^{[k]}\right]
\right)
\dd\xi_s^{[Y]}
\right|
\right]
\\
&\leq
Lb_{\max}^*
\int_0^1
\left(
\E\left[
\left|\xi_s^{[Y]}-\xi_{\eta(s)}^{[Y]}\right|
\right]
+
\mathcal W_1\left(
\mu_s^{[k]},\mu_{\eta(s)}^{[k]}
\right)
\right)\dd s
\\
&\quad+
L\left\{
\int_0^1
\E\left[
\left(
\left|\xi_s^{[Y]}-\xi_{\eta(s)}^{[Y]}\right|
+
\mathcal W_1\left(
\mu_s^{[k]},\mu_{\eta(s)}^{[k]}
\right)
\right)^2
\right]\dd s
\right\}^{1/2}.
\end{align*}

Similarly, the second term in~\eqref{maj1} can be upper-bounded as follows, 
\begin{align}
&\left| \int_{0}^{1}(b_{k}^{*})^2\left[\xi^{[Y]}_s, \mu_{s}^{[k]}\right]\dd s
- \sum_{m=0}^{M-1}\Delta (b_{k}^{*})^2\left[\xi^{[Y]}_{m\Delta}, \mu_{m\Delta}^{[k]}\right]\right|\nonumber\\
&\leq \sum_{m=0}^{M-1}\int_{m\Delta}^{(m+1)\Delta}\left|(b_{k}^{*})^2\left[\xi^{[Y]}_s, \mu_{s}^{[k]}\right]-(b_{k}^{*})^2\left[\xi^{[Y]}_{m\Delta}, \mu_{m\Delta}^{[k]}\right]\right|\dd s\nonumber\\
&\leq\sum_{m=0}^{M-1}\int_{m\Delta}^{(m+1)\Delta}\left|b_{k}^{*}\left[\xi^{[Y]}_s, \mu_{s}^{[k]}\right]-b_{k}^{*}\left[\xi^{[Y]}_{m\Delta}, \mu_{m\Delta}^{[k]}\right]\right|\cdot \left|b_{k}^{*}\left[\xi^{[Y]}_s, \mu_{s}^{[k]}\right]+b_{k}^{*}\left[\xi^{[Y]}_{m\Delta}, \mu_{m\Delta}^{[k]}\right]\right|\dd s\nonumber\\
&\leq \sum_{m=0}^{M-1}\int_{m\Delta}^{(m+1)\Delta}L \left[\big|\xi^{[Y]}_s-\xi^{[Y]}_{m\Delta}\big|+\mathcal{W}_1( \mu_{s}^{[k]}, \mu_{m\Delta}^{[k]})\right]
\cdot \left|b_{k}^{*}\left[\xi^{[Y]}_s, \mu_{s}^{[k]}\right]+b_{k}^{*}\left[\xi^{[Y]}_{m\Delta}, \mu_{m\Delta}^{[k]}\right]\right|\dd s\nonumber\\
&\leq 2b_{\max}^* \sum_{m=0}^{M-1}\int_{m\Delta}^{(m+1)\Delta}L \left[\big|\xi^{[Y]}_s-\xi^{[Y]}_{m\Delta}\big|+\mathcal{W}_1( \mu_{s}^{[k]}, \mu_{m\Delta}^{[k]})\right]\dd s.\nonumber
\end{align}
Proposition \ref{prop:pre-result1} implies that 
\begin{align}
\sup_{s\in[0,1]}\mathcal{W}_1( \mu_{s}^{[k]}, \mu_{\eta(s)}^{[k]})\leq \sup_{s\in[0,1]}\big\Vert \xi_s^{[k]}-\xi_{\eta(s)}^{[k]}\big\Vert_2\leq \mathfrak{C}\sqrt{\Delta} \text{ and }\sup_{s\in[0,1]}\Vert \xi^{[k]}_s\Vert_2\leq \mathfrak{C}. \nonumber
\end{align}
Consequently, we have
$$
\E[T_3]=\E\left[\left|F_k^{*}(\xi)-\bar{F}_k^*(\xi)\right|\right]\leq \mathfrak{C} \sqrt{\Delta}.
$$

\subsection{Proofs of Section \ref{sec:inference}: study of a drift estimator }\label{proof:secdrift}
In this section we work with observations on class $k$ only. 
Let recall that $\left\{\left(X^{[k],n}_{m\Delta}\right)_{m = 0,\ldots,M-1}\right\}$ and $\left\{\left(\wt{X}^{[k],n}_{m\Delta}\right)_{m=0, \ldots,M-1}\right\}$ are two independent samples from the same class $k$ drawn from Equation~\eqref{eq:sys}.
We denote $\mu_{N}$ and $\wt{\mu}_N$ the two associated empirical measures. 
For a function $h$, we then have that
\begin{equation*}
\frac{1}{N} \sum_{n'=1}^N h(X_{m\Delta}^{[k],n}, {\wt{X}_{m\Delta}^{[k],n^{'}}}) = h[X^{[k],n}_{m\Delta}, \wt{\mu}^{[k],N}_{m\Delta}] \;\; {\rm and} \;\; \frac{1}{N} \sum_{n'=1}^N h(X^{[k],n}_{m\Delta}, {{X}_{m\Delta}^{[k], n^{'}}}) = h[X^{[k],n}_{m\Delta}, {\mu}^{[k],N}_{m\Delta}].
\end{equation*}

Let us introduce the following empirical norm, naturaly associated with our empirical contrast,
\begin{equation}\label{eq:def-norm-NMK}
\left\|h\right\|^2_{N,M,k} := \dfrac{1}{NM} \sum_{m=0}^{M-1}\sum_{n = 1}^N  \left(\frac{1}{N} \sum_{n' = 1}^N   h(X^{[k],n}_{m\Delta}, \wtX_{m\Delta}^{[k],n'})\right)^2.
\end{equation}
We remind the reader that $\w{b}_k$ denote the final truncated estimator of $b^*_k$ given in Equation~\eqref{eq:estimatorbhatfinal}, whereas, $\wt{b}_k$ denote the minimizer of the least-squares contrast before truncation given in Equation~\eqref{eq:btilde}.
Besides, for each $k \in \mathcal{Y}$, we also recall that the restriction of $b_k^*$ to $[-A_N,A_N]^2$ denoted by  
$\bar{b}^*_k$ is defined as
\begin{equation}\label{eq:bbarstar}
 \bar{b}^*_k(x,y) = b_k^*(x,y) \one_{\left\{(x,y) \in [-A_N,A_N]^2\right\}}.   
 \end{equation}
Now, we established the following results on $\bar{b}^*_k$.

This proof section is organized as follows. First, we prove the order of the approximation error, given in Proposition~\ref{prop:splineapprox}. Then, we first control the order of the error $\E[\|\bar{b}_k^*-b_k^* \|^2_{M,k,k}$ in Proposition \ref{prop:bkBar}. Then, in Proposition \ref{prop:estimbresult} we deal with the empirical error $\E[\|\wt{b}_k^*-b_k^* \|^2_{N,M,k}$ first. Before finally controlling $\E[\|\bar{b}_k^*-b_k^* \|^2_{M,k,k}$ which is the object of Theorem~\ref{theo:riskbhatMkk}.

\subsubsection{Proof of Proposition~\ref{prop:splineapprox}}\label{proof:propsplineapprox}

Following a similar path to the proof of Lemma C.1 in \cite{denis2021ridge}, we construct the approximation $\mathfrak{b}_k$ as
$$
    \mathfrak{b}_k(x,y) = \sum_{j_1=-R}^{D-1} \sum_{j_2=-R}^{D-1} b_k^*(u_{j_1}, u_{j_2}) B_{j_1,j_2}(x, y).
$$
Let $(x,y) \in [-A,A)^2$. By definition of the knot sequence $\mathbf{u}$, there exist indices $i_1, i_2 \in \{0, \dots, D-1\}$ such that $x \in [u_{i_1}, u_{i_1+1})$ and $y \in [u_{i_2}, u_{i_2+1})$.
Recall that the univariate B-spline $B_{j}$ has a local support such that $B_j(z) = 0$ for $z \notin [u_j, u_{j+R+1})$. Therefore, $B_{j_1}(x) = 0$ for all $j_1 \notin \{i_1-R, \dots, i_1\}$ and $B_{j_2}(y) = 0$ for all $j_2 \notin \{i_2-R, \dots, i_2\}$. The double sum then reduces to these local indices:
\begin{equation*}
    \mathfrak{b}_k(x,y) = \sum_{j_1=i_1-R}^{i_1} \sum_{j_2=i_2-R}^{i_2} b_k^*(u_{j_1}, u_{j_2}) B_{j_1,j_2}(x,y).
\end{equation*}
Since $\sum_{j_1,j_2} B_{j_1,j_2}(x,y) = 1$ for all $(x,y)$ we can write:
\begin{align*}
    \left| \mathfrak{b}_k(x,y) - b^*_k(x,y) \right| &= \left| \sum_{j_1=i_1-R}^{i_1} \sum_{j_2=i_2-R}^{i_2} \left( b_k^*(u_{j_1}, u_{j_2}) - b_k^*(x,y) \right) B_{j_1,j_2}(x,y) \right| \\
    &\leq \max_{\substack{i_1-R \leq j_1 \leq i_1 \\ i_2-R \leq j_2 \leq i_2}} \left| b_k^*(u_{j_1}, u_{j_2}) - b_k^*(x,y) \right| \sum_{j_1,j_2} B_{j_1,j_2}(x,y) \\
    &= \max_{\substack{i_1-R \leq j_1 \leq i_1 \\ i_2-R \leq j_2 \leq i_2}} \left| b_k^*(u_{j_1}, u_{j_2}) - b_k^*(x,y) \right|.
\end{align*}
By Assumption \ref{ass:bsup_et_lip_in_x_and_y}, the function $b_k^*$ is $L$-Lipschitz. Thus, for any $(j_1, j_2)$ within the local support:
$$
    \left| b_k^*(u_{j_1}, u_{j_2}) - b_k^*(x,y) \right| \leq L \left( |u_{j_1} - x| + |u_{j_2} - y| \right).
$$
Since the distance between two consecutive knots is $2A/D$, we have:
$$
    |u_{j_1} - x| \leq u_{i_1+1} - u_{i_1-R} = (R+1)\frac{2A}{D},
$$
and similarly, $|u_{j_2} - y| \leq (R+1)\frac{2A}{D}$. Thus,
$$
    \left| \mathfrak{b}_k(x,y) - b^*_k(x,y) \right| \leq L \left( (R+1)\frac{2A}{D} + (R+1)\frac{2A}{D} \right) = \frac{4A(R+1)}{D}L,
$$
which establishes the bound \eqref{eq:biais}.\\
Finally, we must verify that $\mathfrak{b}_k \in \mathcal{S}_{D,R}$. This requires the coefficients vector $\mathbf{a} = (a_{(j_1,j_2)})$ with $a_{(j_1,j_2)} = b_k^*(u_{j_1}, u_{j_2})$ to satisfy the $\ell^2$-norm constraint $\|\mathbf{a}\|_2^2 \leq (D+R)^2 A^2 \log(N)$ defined in \eqref{eq:SDR}. 
Since the function $b_k^*$ is $L$  Lipschitz, we have that, for all $(x,y) \in (-A,A)^2$,
\begin{equation*}
\left|b_k^*(x,y) \right| \leq \left|b_k^*(x,y)-b_k^*(0,0) \right|+\left|b_k^*(0,0)\right| \leq L(|x|+|y|) +  \left|b_k^*(0,0)\right| \leq A\sqrt{\log(N)},  
\end{equation*}
for $N$ large enough.
Therefore,
\begin{equation*}
    \|\mathbf{a}\|_2^2 = \sum_{j_1,j_2} |b_k^*(u_{j_1}, u_{j_2})|^2 \leq (D+R)^2 A^2 \log(N),
\end{equation*}
ensuring that $\mathfrak{b}_k \in \mathcal{S}_{D,R}$. This concludes the proof.


\begin{proposition}\label{prop:bkBar}
For each $k \in \mathcal{Y}$. For $A_N = \log(N)$, the following holds for $\bar{b}_k^*$ given in Equation~\eqref{eq:bbarstar},
\begin{equation*}
\E\left[\left\|\bar{b}^*_k-b_k^*\right\|^2_{M,k,k} \right] \leq \frac{\mathfrak{C}}{N}.
\end{equation*}
\end{proposition}
\begin{proof}
Let $k \in \mathcal{Y}$. First, using the definition of the pseudo norm $\left\|\cdot\right\|_{M,k,k}$ and the definition of $\bar{b}_k$, we have 
$$\E[ \| \bar{b}_k^*- b_k^*\|_{M,k,k}^2]= 
\E \left[
\frac{1}{M}\sum_{m=0}^{M-1} \left(
\frac{1}{N}\sum_{n=1}^N b_k^*(\xi_{m\Delta}^{k'}, \tilde{X}^{[k],n}_{m\Delta})\one_{\left\{ |\xi_{m\Delta}^{k'}|>A_N \bigcup|\tilde{X}^{[k],n}_{m\Delta}|>A_N
\right\}}
\right)^2\right].
$$
Now, let us notice that, applying Jensen's inequality to the convex function $x \mapsto x^2 $ taking the expectation under the uniform distribution on the discrete set $\{1, \ldots, n\}$, we have that 
 $$
 \left(
\frac{1}{N}\sum_{n=1}^N b_k^*(\xi_{m\Delta}^{k'}, \tilde{X}^{[k],n}_{m\Delta})\one_{\left\{ |\xi_{m\Delta}^{k'}|>A_N \bigcup|\tilde{X}^{[k],n}_{m\Delta}|>A_N
\right\}}
\right)^2
\leq 
\frac{1}{N}\sum_{n=1}^N \left(b_k^*(\xi_{m\Delta}^{k'}, \tilde{X}^{[k],n}_{m\Delta})\one_{\left\{ |\xi_{m\Delta}^{k'}|>A_N \bigcup|\tilde{X}^{[k],n}_{m\Delta}|>A_N
\right\}}
\right)^2.
$$
Then, under Assumption~\ref{ass:bsup_et_lip_in_x_and_y} 
$b_k^*$ is bounded and the particles $\widetilde{X}^{[k],n}$ have the same law for all $n$, thus, we finally have
\begin{eqnarray*}
\E\left[ \| \bar{b}_k^*- b_k^*\|_{M,k,k}^2\right]
&\leq & 
b_{\max}\E \left[
\frac{1}{M}\sum_{m=0}^{M-1}
\one_{\left\{ |\xi_{m\Delta}^{k'}|>A_N \bigcup|\tilde{X}^{[k],1}_{m\Delta}|>A_N
\right\}}
\right]\\
& \leq & b_{\max} \left(\sup_{t \in [0,1]} \P( | \xi^{k'}_t |> A_N)+ \sup_{t \in [0,1]} \P( | \widetilde{X}^{[k]}_t |> A_N)\right).
\end{eqnarray*}
We are now able to use Lemma \ref{lem:ineq-concentration-Xt}
to obtain that 
$$\E\left[ \| \bar{b}_k^*- b_k^*\|_{M,k,k}^2\right]
\leq  4 b_{\max} \exp(-\mathfrak{C} A_N^2).
$$
Finally, the choice $A_N = \log(N)$ lead, for $N$ large enough, to
$$
\E\left[\left\|\bar{b}^*_k-b_k^*\right\|^2_{M,k,k} \right] \leq \frac{\mathfrak{C}}{N}.
$$
\end{proof}

Now, let us control the error for the empirical norm of the untruncated estimator of the drifts. 
\begin{proposition}\label{prop:estimbresult}
Let $k \in \cY$, the estimation $\wt{b}_k$ of $b_k^*$ before truncation, given in Equation~\eqref{eq:btilde}, satisfies
\begin{equation*}
\E[\|\wt{b}_k-{b}_k^*\|^2_{N,M,k}] \leq C_A \left( \frac{1}{D_N^2} + \sqrt{\frac{D_N^2\log(N)}{N}}\right) +C\left(\Delta+ \frac{1}{N}\right)
\end{equation*}
where $C_A$ depend only on $A$ and $C$ depends on $x_0, L$. 
\end{proposition}

\begin{proof}
Let us study the following increment of the process
\begin{eqnarray*}
U_{m\Delta}^{[k],n}&:=& \frac{X_{(m+1)\Delta}^{[k],n}-X_{m\Delta}^{[k],n}}{\Delta}\\
    &=&
{b}^*_k[X_{m\Delta}^{[k],n}, \mu^{[k],N}_{m\Delta}] +\frac{W_{(m+1)\Delta}^{[k],n}-W_{m\Delta}^{[k],n}}{\Delta}  
    \\
    &&+
   \frac{1}{\Delta}\int_{m\Delta}^{(m+1)\Delta} \left(\frac{1}{N}\sum_{n'=1}^N {b}^*_k(X_{s}^{[k],n}, X_{s}^{[k],n'})-\frac{1}{N}\sum_{n'=1}^N {b}^*_k(X_{m\Delta}^{[k],n}, X_{m\Delta}^{[k],n'})\right) \dd s\\
   &=& b_k^*[X_{m\Delta}^{[k],n}, \wt{\mu}^{[k],N}_{m\Delta}]
     + \frac{W_{(m+1)\Delta}^{[k],n}-W_{m\Delta}^{[k],n}}{\Delta}  
    \\
    &&+
   \frac{1}{\Delta}\int_{m\Delta}^{(m+1)\Delta} \left(\frac{1}{N}\sum_{n'=1}^N {b}^*_k(X_{s}^{[k],n}, X_{s}^{[k],n'})-\frac{1}{N}\sum_{n'=1}^N {b}^*_k(X_{m\Delta}^{[k],n}, X_{m\Delta}^{[k],n'})\right) \dd s\\
   &&+ \left(b_k^*[X_{m\Delta}^{[k],n}, {\mu}^{[k],N}_{m\Delta}] - b_k^*[X_{m\Delta}^{[k],n}, \wt{\mu}^{[k],N}_{m\Delta}]\right).
\end{eqnarray*}
Hence, 
\begin{equation*}
U_{m\Delta}^{[k],n} =:  b_k^*[X_{m\Delta}^{[k],n}, \wt{\mu}^{[k],N}_{m\Delta}] + \Sigma^{[k]}_{m,n} + R^{[k]}_{m,n} + \Delta_{\mu,m,n},  
\end{equation*}
with 
\begin{eqnarray*}
\Sigma_{m,n}^{[k]}&:=&\frac{W_{(m+1)\Delta}^{[k],n}-W_{m\Delta}^{[k],n}}{\Delta}    \nonumber\\
R_{m,n}^{[k]}&:=&\frac{1}{\Delta}\int_{m\Delta}^{(m+1)\Delta} \left(\frac{1}{N}\sum_{n'=1}^N {b}^*_k(X_{s}^{[k],n}, X_{s}^{[k],n'})-\frac{1}{N}\sum_{n'=1}^N {b}^*_k(X_{m\Delta}^{[k],n}, X_{m\Delta}^{[k],n'})\right) \dd s\\
\Delta_{\mu,m,n} & := &  \left(b_k^*[X_{m\Delta}^{[k],n}, {\mu}^{[k],N}_{m\Delta}] - b_k^*[X_{m\Delta}^{[k],n}, \wt{\mu}^{[k],N}_{m\Delta}]\right).
\end{eqnarray*}
For all $h$, we then have that
\begin{multline*}
\gamma^{[k]}_{N,M}(h) = \gamma^{[k]}_{N,M}(b) + \left\|h-b\right\|^2_{N,M,k} + 2 \nu_{N,M}(b_k^*-h)\\ + \dfrac{2}{NM}\sum_{m=0}^{M-1} \sum_{n=1}^N ({b}^*_k-h)\left[X_{m\Delta}^{[k],n}, \wt{\mu}^{[k],N}_{m\Delta}\right] R_{m,n}^{[k]}  \\ +\dfrac{2}{NM}  \sum_{m=0}^{M-1}\sum_{n=1}^N({b}^*_k-h)\left[X_{m\Delta}^{[k],n},\wt{\mu}^{[k],N}_{m\Delta}\right] \Delta_{\mu,m,n},
\end{multline*}
with 
\begin{equation*}
\nu_{N,M}(h):= \dfrac{1}{NM}  \sum_{m=0}^{M-1} \sum_{n=1}^N \Sigma_{m\Delta}^{[k],n} \left(\frac{1}{N}\sum_{n'=1}^N h(X_{m\Delta}^{[k],n},\wt{X}^{[k],n'}_{m\Delta})\right).    
\end{equation*}
By definition of $\wt{b}_k$, we deduce that for all $h$
\begin{multline*}
\left\|\wt{b}_k-b_k^*\right\|^2_{N,M,k} \leq
\left\|h-b_k^*\right\|^2_{N,M,k} + 2 \nu_{N,M}(\wt{b}_k-h) + \dfrac{2}{NM}\sum_{m=0}^{M-1} \sum_{n=1}^N (\wt{b}_k-h)[X_{m\Delta}^{[k],n}, \wt{\mu}^{[k],N}_{m\Delta}] R_{m,n}^{[k]}  \\
+\dfrac{2}{NM}  \sum_{m=0}^{M-1}\sum_{n=1}^N(\wt{b}_k-h)[X_{m\Delta}^{[k],n},\wt{\mu}^{[k],N}_{m\Delta}] \Delta_{\mu,m,n}.
\end{multline*}
Now we use the fact that for $a >0$, 
$2xy \leq \frac{1}{a}x^2+ay^2$, that yields
\begin{equation*}
\dfrac{2}{NM}\sum_{m=0}^{M-1} \sum_{n=1}^N (\wt{b}_k-h)[X_{m\Delta}^{[k],n}, \wt{\mu}^{[k],N}_{m\Delta}] R_{m,n}^{[k]} \leq \dfrac{1}{a} \left\| \wt{b}_k - h \right\|^{2}_{N,M,k}   +  \dfrac{a}{NM}\sum_{m = 0}^{M-1}\sum_{n = 1}^N  \left(R_{m,n}^{[k]}\right)^2,
\end{equation*}
and 
\begin{equation*}
\dfrac{2}{NM}\sum_{m=0}^{M-1} \sum_{n=1}^N (\wt{b}_k-h)[X_{m\Delta}^{[k],n}, \wt{\mu}^{[k],N}_{m\Delta}]  \Delta_{\mu,m,n} \leq \dfrac{1}{a} \left\| \wt{b}_k - h \right\|^{2}_{N,M,k} + \dfrac{a}{NM}\sum_{m = 0}^{M-1}\sum_{n = 1}^N
\Delta_{\mu,m,n}^2.
\end{equation*}
It follows from \cite[Section~7.8.1]{pages-numerical-proba} that
$$\E\left[\left(R_{m,n}^{[k]}\right)^2\right] \leq C \Delta.$$ 
Then, we can decompose
\begin{eqnarray*}
\E\left[\Delta_{\mu,m,n}^2\right] &\leq & \E\left[2\left(b_k^*[X_{m\Delta}^{[k],n}, {\mu}^{[k],N}_{m\Delta}]- b_k^*[X_{m\Delta}^{[k],n}, {\mu}_{m\Delta}^{[k]}]\right)^2+ \left(b_k^*[X_{m\Delta}^{[k],n}, \wt{\mu}^{[k],N}_{m\Delta}] - b_k^*[X_{m\Delta}^{[k],n}, {\mu}^{[k]}_{m\Delta}]\right)^2\right]\\
&\leq & C_L\E\left[\left(\mathcal{W}^2_1\left(\mu_{m\Delta}^{[k],N},  \mu^{[k]}_{m\Delta}\right) + \mathcal{W}^2_1\left(\wt{\mu}_{m\Delta}^{[k],N}, \mu^{[k]}_{m\Delta}\right)\right)\right] ,
\end{eqnarray*}
thus, from Proposition \ref{prop:chaos} and that $\mathcal{W}_{1}(\mu, \nu) \leq \mathcal{W}_2(\mu, \nu)$ we have the following control
$$\E\left[\Delta_{\mu,m,n}^2\right] \leq C_L N^{-1}.
$$
Finally, we get 
\begin{equation*}
\E\left[\left\|\wt{b}_k-b_k^*\right\|^2_{N,M,k} \right]\leq
\E\left[\left\|h-b_k^*\right\|^2_{N,M,k}\right] + 2 \nu_{N,M}(\wt{b}_k-h) + \dfrac{2}{a}  \E\left[\left\| \wt{b}_k - h \right\|^{2}_{N,M,k}\right] + Ca\left(\Delta +N^{-1}\right).
\end{equation*}
Now, since $\left\| \wt{b}_k - h \right\|^{2}_{N,M,k} \leq 2\left\| \wt{b}_k - b_k^* \right\|^{2}_{N,M,k} + 2
\left\|h - b_k^* \right\|^{2}_{N,M,k}$, 
we obtain for all $h$
\begin{equation*}
(1-4/a)\E\left[\left\|\wt{b}_k-b_k^*\right\|^2_{N,M,k} \right]\leq (1+4/a)\E\left[\left\|h-b_k^*\right\|^2_{N,M,k}\right] + 2\nu_{N,M}(\wt{b}_k-h) +C( \Delta + N^{-1}).   
\end{equation*}

Now, we focus on the term $\nu_{N,M}(\wt{b}_k-h)$. 
Since the functions $h$ and $\wt{b}_k$ are both in $S_{D_N,R}$, that is 
$h=\sum_{j_1,j_2=-R}^{D-1} h_{j_1,j_2} B_{j_1}B_{j_2}$,
we observe that by linearity, 
$$
\nu_{N, M}\left(\wt{b}_{k}-h\right)=\sum_{j_1=-R}^{D_N}\sum_{j_2=-R}^{D_N}\left(\w{a}_{j_1, j_2}-h_{j_1, j_2}\right) \nu_{N, M}\left(B_{j_1}B_{j_2}\right) $$
and using Cauchy-Schwarz inequality it comes
$$\nu_{N, M}\left(\wt{b}_{k}-h\right) \leq 2 {A_N}\sqrt{(D_N+R)^2 \log(N)} \sqrt{\sum_{j_1=-R}^{D_N}\sum_{j_2=-R}^{D_N} \nu_{N, M}^2\left(B_{j_1}B_{j_2}\right)}.
$$
Therefore, we get
$$
\begin{aligned}
\mathbb{E}\left[\nu_{N, M}\left(\wt{b}_{k}-h\right)\right] & \leq \sqrt{\E\left[\nu_{N, M}^2\left(\wt{b}_{N, M}-h\right)\right]} \\
& \leq  2 {A_N}\sqrt{(D_N+R)^2 \log(N)}  \sqrt{\mathbb{E}\left[\sum_{j_1=-R}^{D_N}\sum_{j_2=-R}^{D_N} \nu_{N, M}^2\left(B_{j_1}B_{j_2}\right)\right]}.
\end{aligned}
$$
Let us study the right hand side of the previous equation. 
\begin{lemma}\label{lem:contrast}
\begin{eqnarray*}\label{eq:sumproc}
\E\left[\sum_{j_1=-R}^{D_N}\sum_{j_2=-R}^{D_N} \nu^2_{N,M}(B_{j_1},B_{j_2})\right]
\leq\frac{1}{N}.
\end{eqnarray*}
\end{lemma}
\begin{proof}[Proof of Lemma \ref{lem:contrast}]
Let us write 
$$\E[\nu^2_{N,M}(B_{j_1}B_{j_2})]=\E\left[\left(\frac{1}{NM}\sum_{n=1}^N  Z_n \right)^2\right]$$
with $$Z_n:=\sum_{m=1}^M \Sigma_{m\Delta}^{[k],n} B_{j_1}(X_{m\Delta}^{[k],n}) \left(\frac{1}{N}\sum_{n'=1}^N B_{j_2}( \wtX_{m\Delta}^{[k],n^{'}})\right)$$
Even if the $Z_n$ random variables are not independent, $\E[Z_nZ_n']=0$ for $n\neq n'$. Indeed, for $m<m'$ (an similarly of $m>m')$ conditioning on $\mathcal{F}_{m'\Delta}$ gives $\E[ \Sigma_{m\Delta}^{[k],n} \Sigma_{m'\Delta}^{[k],n'}]=0$, and for $m=m'$, as the Brownian motions are independent, and $\E[ \Sigma_{m\Delta}^{[k],n} \Sigma_{m\Delta}^{[k],n'}]=0$. Thus the exchangeability of the particles gives
$$\E[\nu^2_{N,M}(B_{j_1}B_{j_2})] 
= \frac{N}{N^2M^2}\E[Z_1^2].$$
Then,
$$
\E[Z_1^2]= \E\left[\sum_{m=1}^M (\Sigma_{m\Delta}^{[k],1})^2 B^2_{j_1}(X_{m\Delta}^{[k],1}) \left(\frac{1}{N}\sum_{n'=1}^N B_{j_2}( \wtX_{m\Delta}^{[k],n^{'}})\right)^2\right]
$$
using conditioning technique and that $\E[\Sigma_{m\Delta}^{[k],1}|\mathcal{F}_{m\Delta}]=0$.
We have now,
$$\E[\nu^2_{N,M}(B_{j_1}B_{j_2})] =\frac{N}{N^2M^2}\sum_{m=1}^M
\E \left[ \left(\Sigma_{m\Delta}^{[k],1}\right)^2 B^2_{j_1}(X_{m\Delta}^{[k],1}) \left(\frac{1}{N}\sum_{n'=1}^N B_{j_2}( \wtX_{m\Delta}^{[k],n^{'}})\right)^2\right].
$$
Note that 
$$\E\left[ \left(\Sigma_{m,n}^{[k]}\right)^2 | \mathcal{F}_{m \Delta}\right]= \frac{1}{\Delta}= M 
$$
thus,
$$\E[\nu^2_{N,M}(B_{j_1}B_{j_2})] =\frac{NM}{N^2M^2}\sum_{m=1}^M
\E \left[ B^2_{j_1}(X_{m\Delta}^{[k],1}) \left(\frac{1}{N}\sum_{n'=1}^N B_{j_2}( \wtX_{m\Delta}^{[k],n^{'}})\right)^2\right].
$$
We use now Cauchy-Schwartz inequality to obtain
$$\E[\nu^2_{N,M}(B_{j_1}B_{j_2})] \leq \frac{NM}{N^2M^2 N}\sum_{m=1}^M
\E \left[ B^2_{j_1}(X_{m\Delta}^{[k],1}) \sum_{n'=1}^N B^2_{j_2}( \wtX_{m\Delta}^{[k],n^{'}})\right].
$$
Since $\sum_{j_1,j_2} B^2_{j_1}(x)B^2_{j_2}(y) \leq 1$ for all $(x,y)$, we obtain that 
\begin{eqnarray*}\label{eq:sumproc}
\E\left[\sum_{j_1=-R}^{D_N}\sum_{j_2=-R}^{D_N} \nu^2_{N,M}(B_{j_1},B_{j_2})\right]
& \leq  & \frac{1}{N^2M}\E \left[\sum_{m=1}^M\sum_{n'=1}^N \sum_{j_1=-R}^{D_N}\sum_{j_2=-R}^{D_N}  B^2_{j_1}(X_{m\Delta}^{[k],1})B^2_{j_2}( \wtX_{m\Delta}^{[k],n^{'}}) \right]\\
&\leq & \frac{1}{N}
\end{eqnarray*}
\end{proof}
We finally have obtained that 
$$\E[\nu_{N,M}(\wt{b}_k-h)] \leq 2{A_N} \sqrt{\frac{(D_N+R)^2 \log(N)}{N}}.$$
Finally, using both Proposition~\ref{prop:splineapprox}, and~\ref{prop:bkBar}, we deduce that for $A_N = \log(N)$, and $N$ large enough
\begin{equation}\label{eq:infbiais}
\inf_{h \in \mathcal{S}_{{D_N},R}}
\E\left[\left\|h-{b}_k^*\right\|_{N,M,k}^{2}\right] \leq \frac{C_L A_N^2}{{D_N}^2} + \E\left[\left\|\bar{b}^*_k-b_k^*\right\|^2_{M,k,k} \right] \leq \dfrac{C_L\log^2(N)}{D_N^2}+ \dfrac{1}{N}.
\end{equation}
It yields 
\begin{equation*}
\E\left[\left\|\wt{b}_k-b_k^*\right\|^2_{N,M,k} \right]\leq C_L\left(\frac{\log^2(N)}{{D_N}^2} +\log^{3/2}(N)\sqrt{\frac{(D_N+R)^2}{N}}\right) +C( \Delta + N^{-1}).   
\end{equation*}
\end{proof}
We are now ready to prove the main result of the section.

\subsubsection{Proof of Theorem~\ref{theo:riskbhatMkk}}

From Proposition \ref{prop:estimbresult},
since for $N$ large enough, we have that under Assumption~\ref{ass:bsup_et_lip_in_x_and_y} that 
\begin{equation*}
\left\|b_k^*\right\|_\infty \leq \sqrt{\log(N)},
\end{equation*}    
we deduce that
\begin{equation}
\label{eq:eqMkk1}
\mathbb{E}\left[\left\|\widehat{b}_k-b_k^*\right\|^2_{N,M,k}\right] \leq \mathbb{E}\left[\left\|\widetilde{b}_k-b_k^*\right\|^2_{N,M,k}\right] \leq C\left(\dfrac{{\log^2(N)}}{D_N^2} +{\log^{3/2}(N)}\sqrt{\dfrac{(D_N+R)^2}{N}}+ \Delta\right),  
\end{equation}
where $C>0$ depends on 
$x_0, L$. 
For each $h \in \mathcal{S}_{D_N,R}$, we denote
we denote by $\bar{h}$ its thresholded counterpart
\begin{equation*}
\bar{h}(\cdot):=h(\cdot) \one_{|h(\cdot)| \leq \log^{1/2}(N)} + \text{sgn}(h) \log^{1/2}(N) \one_{\{|h(\cdot)| > \log^{1/2}(N)\}}.
\end{equation*}
We also introduce the set 
$$\mathcal{H}_{D_N,R} := \left\{\bar{h}, h \in \mathcal{S}_{D_N,R}\right\}.$$
From Equation~\eqref{eq:eqMkk1}, and the following decomposition
\begin{equation}
\label{eq:eqDecompRisk}
\E\left[\left\|\w{b}_k-{b}_k^*\right\|_{M,k,k}^2\right]=\E\left[\left\|\w{b}_k-{b}^*_k\right\|_{M,k,k}^2\right]-2 \E\left[\left\|\w{b}_k-b^*_k\right\|_{N,M,k}^2\right]+2\E\left[\left\|\w{b}_k-{b}^*_k\right\|_{N,M,k}^2\right],
\end{equation}
we get
\begin{eqnarray}
\label{eq:eqMkk4}
\E\left[\left\|\w{b}_k-{b}_k^*\right\|^2_{M,k,k}\right] \leq C\left(\dfrac{{\log^2(N)}}{D_N^2} + {\log^{3/2}(N)}\sqrt{\dfrac{(D_N+R)^2}{N}}+ \Delta\right) \nonumber\\
+
\mathbb{E}\left[\sup _{\bar{h} \in \mathcal{H}_{D_N,R}}\left\|\bar{h}-{b}_k^*\right\|_{M, k,k}^2-2\left\|\bar{h}-{b}_k^*\right\|_{M, N,k}^2\right].
\end{eqnarray}
Hence to prove the result, it remains to control the second term in {\it r.h.s.} of the above equation.
To this end, we introduce the set
\begin{equation*}
\mathcal{G}_{D_N,R}  := \left\{g : x \mapsto \left(\dfrac{1}{N}\sum_{n'=1}^N(\bar{h}-b_k^*)\left(x, \tilde{X}_{m\Delta}^{k, [n']}\right)\right)^2, \;\; \bar{h} \in \mathcal{H}_{D_N,R}\right\}.    
\end{equation*}

Then, we observe that
\begin{multline}
\label{eq:eqMkk3}
\mathbb{E}\left[\sup _{\bar{h} \in \mathcal{H}_{D_N,R}}\left\|\bar{h}-{b}_k^*\right\|_{M, k,k}^2-2\left\|\bar{h}-{b}_k^*\right\|_{M, N,k}^2\right] \leq\\  \E\left[\sup _{g \in \mathcal{G}_{D_N,R}}\left\{\dfrac{1}{M}\sum_{m=0}^{M-1}\mathbb{E}\left[g(\xi^{[k]}_{m\Delta})\right] - \dfrac{2}{MN}\sum_{m=0}^{M-1}\sum_{n=1}^Ng(X_{m\Delta}^{[k],n} ) \right\}\right].
\end{multline}
Now, conditional on $\tilde{X}$, we apply Theorem~18
in~\cite{hoffmann2021} with $\phi(t,X_t) = g(X_t)$, and $\rho$ the uniform distribution on $(m\Delta)_{m \in \left\{0, \ldots, M-1\right\}}$.

Let $g \in \mathcal{G}_{D_N,R}$.
Since, for $N$ large enough $\left\|b_k^*\right\|_{\infty} \leq \sqrt{\log(N)} $, and by definition $\left\|\bar{h}\right\|_{\infty} \leq \sqrt{\log(N)}$, for $\bar{h} \in \mathcal{H}_{D_N,R}$, we deduce that there exists $\bar{h} \in \mathcal{H}_{D_N,R}$ such that
\begin{equation*}
\left\|g\right\|_{\infty} \leq \left\|(\bar{h}-b^*_k)^2\right\|_{\infty} \leq 4 \log(N). 
\end{equation*}
Furthermore, since $g(x) \geq 0$, for all $x \in \mathbb{R}$,
\begin{equation*}
 \E\left[\frac{1}{M}\sum_{m=0}^{M-1}g^2(\xi_{m\Delta})\right] \leq 
\left\|g\right\|_{\infty} \dfrac{1}{M}\sum_{m=0}^{M-1}\mathbb{E}\left[g(\xi^{[k]}_{m\Delta})\right] \leq 4\log(N)\dfrac{1}{M}\sum_{m=0}^{M-1}\mathbb{E}\left[g(\xi^{[k]}_{m\Delta})\right],
\end{equation*}    
Applying Theorem~18 in~\cite{hoffmann2021}, we get that for each $g \in \mathcal{G}_{D_N,R}$ conditional on $\wt{X}$ that for all $x > 0$
\begin{multline*}
\mathbb{P}\left(\dfrac{1}{M}\sum_{m=0}^{M-1}\mathbb{E}\left[g(\xi^{[k]}_{m\Delta})\right] - \dfrac{2}{MN}\sum_{m=0}^{M-1}\sum_{n=1}^Ng(X_{m\Delta}^{[k],n} ) \geq x\right) \leq \\  \mathbb{P}\left(\dfrac{1}{M}\sum_{m=0}^{M-1}\mathbb{E}\left[g(\xi^{[k]}_{m\Delta})\right] - \dfrac{1}{MN}\sum_{m=0}^{M-1}\sum_{n=1}^Ng(X_{m\Delta}^{[k],n} ) \geq \left(\dfrac{x + \dfrac{1}{M}\sum_{m=0}^{M-1}\mathbb{E}\left[g(\xi^{[k]}_{m\Delta})\right]}{2} \right)\right) \\ \leq \exp\left(-CN\dfrac{\left(x + \dfrac{1}{M}\sum_{m=0}^{M-1}\mathbb{E}\left[g(\xi^{[k]}_{m\Delta})\right]\right)^2}{\E\left[\frac{1}{M}\sum_{m=0}^{M-1}g^2(\xi_{m\Delta})\right] + \left(x + \dfrac{1}{M}\sum_{m=0}^{M-1}\mathbb{E}\left[g(\xi^{[k]}_{m\Delta})\right]\right)\left\|g\right\|_{\infty}}\right)
\end{multline*}
where $C>0$ depends on $b^*_k$.
Now, we observe that
\begin{equation*}
\E\left[\frac{1}{M}\sum_{m=0}^{M-1}g^2(\xi_{m\Delta})\right] + \left(x + \dfrac{1}{M}\sum_{m=0}^{M-1}\mathbb{E}\left[g(\xi^{[k]}_{m\Delta})\right]\right)\left\|g\right\|_{\infty} \leq 8\log(N)\left(x + \dfrac{1}{M}\sum_{m=0}^{M-1}\mathbb{E}\left[g(\xi^{[k]}_{m\Delta})\right]\right),
\end{equation*}
which implies that
Conditional on $\wt{X}$, for all $g \in \mathcal{G}_{D_N,R}$, and $x > 0$
\begin{align}
\label{eq:eqMkk2}
\mathbb{P}\left(\dfrac{1}{M}\sum_{m=0}^{M-1}\mathbb{E}\left[g(\xi^{[k]}_{m\Delta})\right] - \dfrac{2}{MN}\sum_{m=0}^{M-1}\sum_{n=1}^Ng(X_{m\Delta}^{[k],n} ) \geq x\right) \nonumber \\
\leq \exp\left(-C\dfrac{N}{\log(N)}\left(x+\dfrac{1}{M}\sum_{m=0}^{M-1}\mathbb{E}\left[g(\xi^{[k]}_{m\Delta})\right]\right)\right)\nonumber \\
\leq \exp\left(-\dfrac{CNx}{\log(N)}\right).
\end{align}

To conclude the proof, for $\varepsilon > 0$, we consider $\mathcal{G}_{\varepsilon}$ an $\varepsilon$-net of $\mathcal{G}_{D_N,R}$ {\it w.r.t.} the $\left\|.\right\|_{\infty}$ norm and we denote by $\mathcal{N}_{\infty}\left(\mathcal{G}_{D_N,R}, \varepsilon\right)$ its cardinality.
We have that for each $g \in \mathcal{G}_{D_N,R}$, there exists $g_{\varepsilon} \in \mathcal{G}_{\varepsilon}$ such that
\begin{equation*}
\left\|g-g_{\varepsilon}\right\|_{\infty} \leq \varepsilon.    
\end{equation*}
Therefore, we have
\begin{multline*}
\dfrac{1}{M}\sum_{m=0}^{M-1}\mathbb{E}\left[g(\xi^{[k]}_{m\Delta})\right] - \dfrac{2}{MN}\sum_{m=0}^{M-1}\sum_{n=1}^Ng(X_{m\Delta}^{[k],n}) = \dfrac{1}{M}\sum_{m=0}^{M-1}\mathbb{E}\left[g(\xi^{[k]}_{m\Delta})\right] - \dfrac{2}{MN}\sum_{m=0}^{M-1}\sum_{n=1}^Ng(X_{m\Delta}^{[k],n})- \\ \left( \dfrac{1}{M}\sum_{m=0}^{M-1}\mathbb{E}\left[g_{\varepsilon}(\xi^{[k]}_{m\Delta})\right] - \dfrac{2}{MN}\sum_{m=0}^{M-1}\sum_{n=1}^Ng_{\varepsilon}(X_{m\Delta}^{[k],n})\right)\\+\dfrac{1}{M}\sum_{m=0}^{M-1}\mathbb{E}\left[g_{\varepsilon}(\xi^{[k]}_{m\Delta})\right] - \dfrac{2}{MN}\sum_{m=0}^{M-1}\sum_{n=1}^Ng_{\varepsilon}(X_{m\Delta}^{[k],n}).
\end{multline*}
From the above inequality, we deduce that
\begin{equation*}
\dfrac{1}{M}\sum_{m=0}^{M-1}\mathbb{E}\left[g(\xi^{[k]}_{m\Delta})\right] - \dfrac{2}{MN}\sum_{m=0}^{M-1}\sum_{n=1}^Ng(X_{m\Delta}^{[k],n}) \leq 3\varepsilon + \dfrac{1}{M}\sum_{m=0}^{M-1}\mathbb{E}\left[g_{\varepsilon}(\xi^{[k]}_{m\Delta})\right] - \dfrac{2}{MN}\sum_{m=0}^{M-1}\sum_{n=1}^Ng_{\varepsilon}(X_{m\Delta}^{[k],n}),
\end{equation*}
which yields
\begin{multline*}
\mathbb{E}\left[\sup _{g \in \mathcal{G}_{D_N,R}}\left\{\dfrac{1}{M}\sum_{m=0}^{M-1}\mathbb{E}\left[g(\xi^{[k]}_{m\Delta})\right] - \dfrac{2}{MN}\sum_{m=0}^{M-1}\sum_{n=1}^Ng(X_{m\Delta}^{[k],n} ) \right\}\right]
\leq 3 \varepsilon + \\\mathbb{E}\left[\sup _{g \in \mathcal{G}_{\varepsilon}}\left\{\dfrac{1}{M}\sum_{m=0}^{M-1}\mathbb{E}\left[g(\xi^{[k]}_{m\Delta})\right] - \dfrac{2}{MN}\sum_{m=0}^{M-1}\sum_{n=1}^Ng(X_{m\Delta}^{[k],n} ) \right\}\right].
\end{multline*}
Now, we bound the second in the {\it r.h.s.} of the above equation.
For each $u \geq 0$, we have that 
\begin{multline*}
\mathbb{E}\left[\sup _{g \in \mathcal{G}_{\varepsilon}}\left\{\dfrac{1}{M}\sum_{m=0}^{M-1}\mathbb{E}\left[g(\xi^{[k]}_{m\Delta})\right] - \dfrac{2}{MN}\sum_{m=0}^{M-1}\sum_{n=1}^Ng(X_{m\Delta}^{[k],n} ) \right\}\right] \leq \\  u+ \int_{x \geq u} \mathbb{P}\left(\sup_{g \in \mathcal{G}_{\varepsilon}} \dfrac{1}{M}\sum_{m=0}^{M-1}\mathbb{E}\left[g(\xi^{[k]}_{m\Delta})\right] - \dfrac{2}{MN}\sum_{m=0}^{M-1}\sum_{n=1}^Ng(X_{m\Delta}^{[k],n} ) \geq x\right)   {\rm d} x.
\end{multline*}
Hence, from Equation~\eqref{eq:eqMkk2}, we deduce that
\begin{multline*}
\mathbb{E}\left[\sup _{g \in \mathcal{G}_{\varepsilon}}\left\{\dfrac{1}{M}\sum_{m=0}^{M-1}\mathbb{E}\left[g(\xi^{[k]}_{m\Delta})\right] - \dfrac{2}{MN}\sum_{m=0}^{M-1}\sum_{n=1}^Ng(X_{m\Delta}^{[k],n} ) \right\}\right] \leq \\u + \mathcal{N}_{\infty}\left(\mathcal{G}_{D_N,R}, \varepsilon\right)\int_{x \geq u} \exp\left(-C\dfrac{Nx}{\log(N)}\right) {\rm d}x,   
\end{multline*}
which yields, with $u=C \dfrac{\log(N)\log(\mathcal{N}_{\infty}\left(\mathcal{G}_{D_N,R}, \varepsilon\right))}{N}$,
\begin{equation*}
\mathbb{E}\left[\sup _{g \in \mathcal{G}_{\varepsilon}}\left\{\dfrac{1}{M}\sum_{m=0}^{M-1}\mathbb{E}\left[g(\xi^{[k]}_{m\Delta})\right] - \dfrac{2}{MN}\sum_{m=0}^{M-1}\sum_{n=1}^Ng(X_{m\Delta}^{[k],n} ) \right\}\right] \leq C\dfrac{\log(N)\log\left(\mathcal{N}_{\infty}\left(\mathcal{G}_{D_N,R}, \varepsilon\right)\right)}{N}
\end{equation*}
Hence from the above equation, and Equation~\eqref{eq:eqMkk4} and ~\eqref{eq:eqMkk3}, we get
\begin{equation}
\label{eq:eqMkk5}
\E\left[\left\|\w{b}_k-{b}_k^*\right\|^2_{M,k,k}\right] \leq C\left(\dfrac{{\log^2(N)}}{D_N^2} + {\log^{3/2}(N)}\sqrt{\dfrac{(D_N+R)^2}{N}}+ \Delta + \dfrac{\log(N)\log\left(\mathcal{N}_{\infty}\left(\mathcal{G}_{D_N,R}, \varepsilon\right)\right)}{N}+3\varepsilon\right).
\end{equation}
Therefore, to conclude the proof, it remains to upper-bound 
$\mathcal{N}_{\infty}\left(\mathcal{G}_{D_N,R}, \varepsilon\right)$.
To this extent, we observe that since for each $h \in \mathcal{S}_{D_N,R}$ with $h = \sum_{j_1,j_2} a_{j_1,j_2} B_{j_1,j_2}$ we have 
$\left\|{\bf a}\right\|_2 \leq (D_N+R){A_N}\log^{1/2}(N)$, we deduce that
\begin{equation}\label{eq:epsnetS}
\mathcal{N}_{\infty}\left(\mathcal{S}_{D_N,R}, \varepsilon\right) \leq \left(\dfrac{3(D_N+R){A_N}\log^{1/2}(N)}{\varepsilon}\right)^{(D_N+R)^2}.   
\end{equation}
Now, let $g, g_{\varepsilon} \in \mathcal{G}_{D_N,R}$, and $h, h_{\varepsilon} \in \mathcal{S}_{D_N,R}$ such that
\begin{equation*}
g(x) = \left(\dfrac{1}{N}\sum_{n=1}^N(\bar{h}-b_k^*)(x, \tilde{X}_{m\Delta}^{[k],n})\right)^2, \;\; {\rm and}, \;\;     g_{\varepsilon}(x) = \left(\dfrac{1}{N}\sum_{n=1}^N(\bar{h}_{\varepsilon}-b_k^*)(x, \tilde{X}_{m\Delta}^{[k],n})\right)^2,
\end{equation*}
such that 
\begin{equation*}
\left\|h-h_{\varepsilon}\right\|_{\infty} \leq \varepsilon.    
\end{equation*}
Since we have $\left\|\bar{h}\right\| _{\infty} \leq \log^{1/2}(N)$, and$ \left\|\bar{h}_{\varepsilon}\right\| _{\infty} \leq \log^{1/2}(N)$,
we deduce that
\begin{equation*}
\left|g(x)-g_{\varepsilon}(x)\right|\leq \left(\left\|\bar{h}\right\|_{\infty}+\left\|\bar{h}_{\varepsilon}\right\|_{\infty} + 2\left\|b_k^*\right\|_{\infty}\right) \left(\left\|\bar{h}-\bar{h}_{\varepsilon}\right\|_{\infty}\right) \leq 4\log^ {1/2}(N)\left\|{h}-{h}_{\varepsilon}\right\|_{\infty} \leq 4\log^ {1/2}(N)\varepsilon.
\end{equation*}
Hence, we deduce that
\begin{equation}\label{eq:NinftyG}
\mathcal{N}_{\infty}\left(\mathcal{G}_{D_N,R}, \varepsilon\right) \leq    \mathcal{N}_{\infty}\left(\mathcal{S}_{D_N,R}, \dfrac{\varepsilon}{4\log^{1/2}(N)}\right) \leq  \left(\dfrac{12(D_N+R){A_N}\log(N)}{\varepsilon}\right)^{(D_N+R)^2}. 
\end{equation}
Then, for $N$ large enough,  we deduce from the above inequality, and Equation~\eqref{eq:eqMkk5}, with $\varepsilon = \dfrac{12(D_N+R){A_N}\log(N)}{N}$, and {$A_N = \log(N)$} we obtain 
\begin{equation*}
\E\left[\left\|\w{b}_k-{b}_k^*\right\|^2_{M,k,k}\right] \leq C\left(\dfrac{{\log^2(N)}}{D_N^2} + {\log^{3/2}(N)}\sqrt{\dfrac{(D_N+R)^2}{N}}+ \Delta + \dfrac{(D_N+R)^2{\log(N)}}{N}\right),   
\end{equation*}
which yields the desired result and ends the proof.

\subsection{Proof of Lemma~\ref{lem:boundIntegrated}}
Recall that
\begin{equation}\label{eq:def-norm-mkk-prime}
\Vert h \Vert_{M, k, k'}^2\coloneqq \E \left[\frac{1}{M} \sum_{m=0}^{M-1}\left( \frac{1}{N}\sum_{n'=1}^{N}h\left(\xi_{m\Delta}^{[k']}, \, \widetilde{X}_{m\Delta}^{[k],n'}\right)\right)^2\right].
\end{equation}
First of all, since $\w{b}_k$, and $b_k^*$ are bounded (for $N$ large enough) by $\sqrt{\log(N)}$, we have that
\begin{multline*}
\left|\| \w{b}_k-b_k^* \|_{M, k, k', \star}^2 - \| \w{b}_k-b_k^* \|_{M, k, k'}^2\right|
\leq \\ \dfrac{2\sqrt{\log(N)}}{M}\sum_{m=0}^M
\mathbb{E}\left[\left|\int_{\R} (\w{b}_k-b_k^*)\big(\xi_{m\Delta}^{[k']}, y\big)\mu_{m\Delta}^{[k]}(dy)-\frac{1}{N}\sum_{n'=1}^{N} (\w{b}_k-b_k^*)\left( \xi_{m\Delta}^{[k']}, \, \widetilde{X}_{m\Delta}^{[k],n'}\right)\right|\right].
\end{multline*}
Now, for each $m \in \{0, \ldots, M-1\}$, we provide an upper-bound for the term
\begin{multline*}
\mathbb{E}\left[\left|\int_{\R} (\w{b}_k-b_k^*)\big(\xi_{m\Delta}^{[k']}, y\big)\mu_{m\Delta}^{[k]}(dy)-\frac{1}{N}\sum_{n'=1}^{N} (\w{b}_k-b_k^*)\left( \xi_{m\Delta}^{[k']}, \, \widetilde{X}_{m\Delta}^{[k],n'}\right)\right|\right].   
\end{multline*}

Let $m \in \left\{0, \ldots, M-1\right\}$ be fixed. 
We have that 
\begin{multline*}
\left|\int_{\R} (\w{b}_k-b_k^*)\big(\xi_{m\Delta}^{[k']}, y\big)\mu_{m\Delta}^{[k]}(dy)-\frac{1}{N}\sum_{n'=1}^{N} (\w{b}_k-b_k^*)\left( \xi_{m\Delta}^{[k']}, \, \widetilde{X}_{m\Delta}^{[k],n'}\right)\right| \\\leq \sup_{\bar{h} \in \mathcal{H}_{D_N,R}} \left|\int_{\R} (\bar{h}-b_k^*)\big(\xi_{m\Delta}^{[k']}, y\big)\mu_{m\Delta}^{[k]}(dy)-\frac{1}{N}\sum_{n'=1}^{N} (\bar{h}-b_k^*)\left( \xi_{m\Delta}^{[k']}, \, \widetilde{X}_{m\Delta}^{[k],n'}\right)\right|.
\end{multline*}
 Now, for $\varepsilon > 0$, we consider $\mathcal{H}_{\varepsilon}$ an $\varepsilon$-net of $\mathcal{H}_{D_N,R}$ with respect to the $\left\|\cdot\right\|_{\infty}$ and denote by  $\mathcal{N}_\varepsilon:= \mathcal{N}_{\infty}\left(\mathcal{H}_{D_N,R}, \varepsilon\right) $ its cardinality.
Then, we have that 
\begin{multline*}
\left|\int_{\R} (\w{b}_k-b_k^*)\big(\xi_{m\Delta}^{[k']}, y\big)\mu_{m\Delta}^{[k]}(dy)-\frac{1}{N}\sum_{n'=1}^{N} (\w{b}_k-b_k^*)\left( \xi_{m\Delta}^{[k']}, \, \widetilde{X}_{m\Delta}^{[k],n'}\right)\right| \leq \\ 2\varepsilon +  \sup_{\bar{h} \in \mathcal{H_{\varepsilon}}} \left|\int_{\R} (\bar{h}-b_k^*)\big(\xi_{m\Delta}^{[k']}, y\big)\mu_{m\Delta}^{[k]}(dy)-\frac{1}{N}\sum_{n'=1}^{N} (\bar{h}-b_k^*)\left( \xi_{m\Delta}^{[k']}, \, \widetilde{X}_{m\Delta}^{[k],n'}\right)\right|.   
\end{multline*}

Let us denote 
$$\P(\cdot| \xi^{[k']}_{m\Delta})= \mathbb{P}_{\xi^{[k']}_{m\Delta}}(\cdot).$$
Let $\bar{h} \in \mathcal{H}_{\varepsilon}$.
From Theorem~18 in~\cite{hoffmann2021}, we deduce  that there exist $\kappa_1, \kappa_2 > 0$
such that, for all $x > 0 $
\begin{multline*}
\mathbb{P}_{\xi^{[k']}_{m\Delta}}\left(\left|\int_{\R} (\bar{h}-b_k^*)\big(\xi_{m\Delta}^{[k']}, y\big)\mu_{m\Delta}^{[k]}(dy)-\frac{1}{N}\sum_{n'=1}^{N} (\bar{h}-b_k^*)\left( \xi_{m\Delta}^{[k']}, \, \widetilde{X}_{m\Delta}^{[k],n'}\right)\right| \geq x\right) \\\leq 2 \kappa_1 \exp\left(-\kappa_2\dfrac{Nx^2}{\left\|\bar{h}-{b}^*_k\right\|^2_{L^2(\mu_{m\Delta})} +\left\|\bar{h}-{b}^*_k\right\|_{\infty}x } \right).  
\end{multline*}
Let us notice that
\begin{equation*}
\left\|\bar{h}-{b}^*_k\right\|^2_{L^2(\mu_{m\Delta})} \leq 4\log(N), \;\; {\rm and}, \;\;  \left\|\bar{h}-{b}^*_k\right\|_{\infty} \leq 2\sqrt{\log(N)}. 
\end{equation*}
Then, for all $x>0$, 
\begin{multline*}
\P\left(\sup_{\bar{h}\in \mathcal{H}_\varepsilon}\left|\int_{\R} (\bar{h}-b_k^*)\big(\xi_{m\Delta}^{[k']}, y\big)\mu_{m\Delta}^{[k]}(dy)-\frac{1}{N}\sum_{n'=1}^{N} (\bar{h}-b_k^*)\left( \xi_{m\Delta}^{[k']}, \, \widetilde{X}_{m\Delta}^{[k],n'}\right)\right| \geq x\right) \\
\leq \min\left(1, \cN_{\varepsilon} 2 \kappa_1 \exp\left(-\kappa_2\dfrac{Nx^2}{\left\|\bar{h}-{b}^*_k\right\|^2_{L^2(\mu_{m\Delta})} +\left\|\bar{h}-{b}^*_k\right\|_{\infty}x } \right)\right).
\end{multline*}
Now, we integrate the previous equation to obtain the control for the expectation, 
\begin{multline*}
\E\left[\sup_{\bar{h}\in \mathcal{H}_\varepsilon}\left|\int_{\R} (\bar{h}-b_k^*)\big(\xi_{m\Delta}^{[k']}, y\big)\mu_{m\Delta}^{[k]}(dy)-\frac{1}{N}\sum_{n'=1}^{N} (\bar{h}-b_k^*)\left( \xi_{m\Delta}^{[k']}, \, \widetilde{X}_{m\Delta}^{[k],n'}\right)\right| \right] \\
\leq \int_0^\infty \min\left(1, \cN_{\varepsilon} 2 \kappa_1 \exp\left(-\kappa_2\dfrac{Nx^2}{\left\|\bar{h}-{b}^*_k\right\|^2_{L^2(\mu_{m\Delta})} +\left\|\bar{h}-{b}^*_k\right\|_{\infty}x } \right)\right)dx\\
:=I_1+I_2,
\end{multline*}
where we split the integral according to whether $x \leq C_h$ ($I_1$) or $x > C_h$ ($I_2$)
with $C_h:=\|\bar{h}-{b}^*_k\|^2_{L^2(\mu_{m\Delta})}/\|\bar{h}-{b}^*_k\|_\infty$.
First,  we bound $I_1$. We have that 
$$I_1 = \int_0^{C_h}  \min\left(1, \cN_{\varepsilon} 2 \kappa_1\exp\left(-\kappa_2\dfrac{Nx^2}{\left\|\bar{h}-{b}^*_k\right\|^2_{L^2(\mu_{m\Delta})} +\left\|\bar{h}-{b}^*_k\right\|_{\infty}x } \right)\right) dx.$$
Since for $x \leq \|\bar{h}-{b}^*_k\|^2_{L^2(\mu_{m\Delta})}/\|\bar{h}-{b}^*_k\|_\infty=C_h$, the denominator in the exponential is smaller than $8\log(N)$, we obtain 
$$I_1 \leq C\int_0^\infty 
\exp\left(-\left(\kappa_2\dfrac{Nx^2}{8\log(N)}- \log(2\kappa_1\cN_{\varepsilon})\right)_+\right)dx.$$
Now, we observe that
\begin{equation*}
I_1 \leq C \left(\int_{0}^{t_0} dx + \int_{t_0}^{+\infty} \exp\left(-\left(\kappa_2\dfrac{Nx^2}{8\log(N)}- \log(2\kappa_1\cN_{\varepsilon})\right)\right)dx  \right),   
\end{equation*}
with $t_0:=\left(\log(2\kappa_1\cN_{\varepsilon}) 8 \log(N) / \kappa_2N\right)^{1/2}$. Therefore, using that for all $t > 0$
\begin{equation*}
\int_t^\infty e^{-cx^2} dx \leq \dfrac{e^{-ct^2}}{2tc},
\end{equation*}
we deduce that 
\begin{equation*}
I_1 \leq C\sqrt{\frac{\log(2 \kappa_1\cN_{\varepsilon}) \log(N)}{N}}.    
\end{equation*}
Similarly, for  $x> C_h=\|\bar{h}-{b}^*_k\|^2_{L^2(\mu_{m\Delta})}/\|\bar{h}-{b}^*_k\|_\infty$,  we get
$$I_2\leq C \int_{0}^\infty \exp\left(-\left(\kappa_2\dfrac{Nx}{4\sqrt{\log(N)} } -\log(2\kappa_1\mathcal{N}_\varepsilon)\right)_+\right) dx  \leq C \frac{\log(2\kappa_1\cN_\varepsilon) \sqrt{\log(N)}}{N}.$$
Therefore, we deduce that
\begin{equation*}
\mathbb{E}\left[\sup_{\bar{h} \in \mathcal{H_{\varepsilon}}} \left|\int_{\R} (\bar{h}-b_k^*)\big(\xi_{m\Delta}^{[k']}, y\big)\mu_{m\Delta}^{[k]}(dy)-\frac{1}{N}\sum_{n'=1}^{N} (\bar{h}-b_k^*)\left( \xi_{m\Delta}^{[k']}, \, \widetilde{X}_{m\Delta}^{[k],n'}\right)\right|\right] 
\leq \mathfrak{C} \sqrt{\log(2\kappa_1\mathcal{N}_{\varepsilon}) }\sqrt{\frac{\log(N)}{N}}.
\end{equation*}

Now, we use that 
$T(x) := x \one_{\{|x| \leq \log^{1/2}(N)\}} + \text{sgn}(x) \log^{1/2}(N) \one_{\{|x| > \log^{1/2}(N)\}}$,the truncation operator is $1$-Lipschitz, thus for any $h_1, h_2 \in \mathcal{S}_{D_N,R}$:
$$\|\bar{h}_1 - \bar{h}_2\|_\infty \leq \|h_1 - h_2\|_\infty.$$
Because distances do not increase under truncation, the image of any $\varepsilon$-net of $\mathcal{S}_{D_N,R}$ under $T$ automatically forms an $\varepsilon$-net for $\mathcal{H}_{D_N,R}$. Consequently, from Equation~\eqref{eq:epsnetS}, we deduce that
\begin{equation*}
   \mathcal{N}_\varepsilon= \mathcal{N}_{\infty}\left(\mathcal{H}_{D_N,R}, \varepsilon\right) \leq \mathcal{N}_{\infty}\left(\mathcal{S}_{D_N,R}, \varepsilon\right) \leq \left(\dfrac{3(D_N+R){A_N}\log^{1/2}(N)}{\varepsilon}\right)^{(D_N+R)^2}.
\end{equation*}
Finally, with $\varepsilon = \dfrac{12(D_N+R){A_N}\log(N)}{N}$, we get
\begin{equation*}
\log(\mathcal{N}_{\varepsilon}) \leq C D_N^2\log(N),  
\end{equation*}
and the desired result.
$\square$

\subsubsection{Proof of Corollary \ref{coro:RateIntegratedNorm}}
As we have obtained in Theorem~\ref{theo:riskbhatMkk} that 
\begin{equation*}
\E\left[ \|\w{b}_k-{b}^*_k\|^2_{M,k,k}\right] 
\leq 
\mathfrak{C} \left( \frac{\log^2(N)}{D_N^2} + \log^{3/2}(N)\sqrt{\frac{(D_N+R)^2}{N}}+ \Delta  \right).
\end{equation*}
Then, 
we write
$$\E\left[ \|\w{b}_k-{b}^*_k\|^2_{M,k,k,*}\right]\leq \E\left[ \|\w{b}_k-{b}^*_k\|^2_{M,k,k}\right]+\left( \E\left[ \|\w{b}_k-{b}^*_k\|^2_{M,k,k,*}\right]-\E\left[ \|\w{b}_k-{b}^*_k\|^2_{M,k,k}\right]\right)$$
and from Lemma~\ref{lem:boundIntegrated} we have 
$$\left|\E\left[ \|\w{b}_k-{b}^*_k\|^2_{M,k,k,*}\right]-\E\left[ \|\w{b}_k-{b}^*_k\|^2_{M,k,k}\right]\right|\leq \mathfrak{C}\frac{\log(N)D_N}{\sqrt{N}}.
$$
Then, taking $\Delta=O(1/N)$, $D_N\propto N^{1/6}$ yields
$$\E\left[ \|\w{b}_k-{b}^*_k\|^2_{M,k,k,*}\right] \propto N^{-1/3},$$
which is the attended result.

\subsection{Proofs of Section \ref{sec:classiffinal}}

\subsubsection{Proof of Proposition~\ref{prop:changementnorme}}\label{subsubsec:proof-changementnorme}
We recall that the integrated norm is defined for a function $h$ as 
\begin{equation}\label{eq:def-norm-mkkprimestar}
\Vert h \Vert_{M, k, k', \star}^2\coloneqq \E\left[ \frac{1}{M}\sum_{m=0}^{M-1} \left(\int_{\R} h(\xi_{m\Delta}^{[k']}, y)\mu_{m\Delta}^{[k]}(dy)\right)^2\right]. 
\end{equation}

Recall that the first step in the proof of Proposition \ref{prop:bayes} establishes the existence of a reference probability measure $\P_0$ and a Brownian motion $W^0$ under $\P_0$ such that the unique solution $\xi=(\xi_t)_{t\in[0,1]}$ of the McKean-Vlasov Equation~\eqref{eq:mkv} corresponds to the process $\widetilde{\xi}=(\widetilde{\xi}_t)_{t\in[0,1]}\coloneqq (x_0+W_t^0)_{t\in[0,1]}$ as described in Equation~\eqref{eq:ref-process}, and we have the following Radon–Nikodym derivative for every $t\in[0,1]$ and for every $k\in\mathcal{Y}$
\begin{eqnarray}\label{eq:dPk-derivative-}
\frac{\dd {\P}_k}{\dd\P_0}\Big|_{\cF_t}=\Phi_t^k&=&\exp\Big(\int_{0}^{t}b^*_k\big[ \widetilde{\xi}_s, \mu_s^{[k]}\big]\dd W_s^0-\frac{1}{2}\;\int_{0}^{t}\left(b^*_k\big[ \widetilde{\xi}_s, \mu_s^{[k]}\big]\right)^2\dd s\Big)\nonumber\\
&=&\exp\Big(\int_{0}^{t}b^*_k\big[ {\xi}_s, \mu_s^{[k]}\big]\dd W_s+\frac{1}{2}\;\int_{0}^{t}\left(b^*_k\big[ {\xi}_s, \mu_s^{[k]}\big]\right)^2\dd s\Big)
\end{eqnarray}
since $W_t^0=W_t+\int_{0}^{t}b^*_k\big[ {\xi}_s, \mu_s^{[k]}\big]\dd s$. Hence, the above equality \eqref{eq:dPk-derivative-} implies that  for every $k,k'\in\mathcal{Y}$ such that $k\neq k'$, and for every $t\in[0,1]$, 
\begin{align}
\frac{\dd {\P}_k}{\dd\P_{k'}}\Big|_{\cF_t}&=\exp\left(\int_0^t \left( b_k^*[{\xi}_s, \mu_s^{[k]}-b_{k'}^*[{\xi}_s, \mu_s^{[k']}]\right)\dd W_s +\frac{1}{2}\int_0^t \left[ \big(b_k^*[{\xi}_s, \mu_s^{[k]}] \big)^2- \big(b_{k'}^*[{\xi}_s, \mu_s^{[k']}]\big)^2\right]\dd s\right)\nonumber\\
&\leq \exp\left(\int_0^t \left( b_{k}^*[{\xi}_s, \mu_s^{[k]}]-b_{k'}^*[{\xi}_s, \mu_s^{[k']}]\right)\dd W_s + t b^*_{\max}\right)\eqqcolon C_{ b^*_{\max}} \exp\big(\mathcal{M}_t^{k,k'}\big),\nonumber
\end{align}
with $C_{b^*_{\max}}= \exp\big(b^*_{\max}\big)$ (and can change in the following) and 
\begin{equation}\label{eq:Mij}
\mathcal{M}_t^{k,k'}\coloneqq \int_0^t \left( b_k^*\left[{\xi}_s, \mu_s^{[k]}\right]-b_{k'}^*\left[{\xi}_s, \mu_s^{[k']}\right]\right)\dd W_s,\quad t\in[0,1]. 
\end{equation}
Consequently, we can get the following inequality
\begin{align}
&\Big\Vert \w{b}_k-{b}_k^* \Big\Vert_{M, k, k', \star}^2=\frac{1}{M}\sum_{m=0}^{M-1} \E_{k'}\left[ \left(\int_{\R} \left(\w{b}_k-{b}_k^*\right)(\xi_{m\Delta}, y)\mu_{m\Delta}^{[k]}(dy)\right)^2\right]\nonumber\\
&\quad \leq \frac{1}{M}\sum_{m=0}^{M-1} \E_{k}\left[ \left(\int_{\R} \left(\w{b}_k-{b}_k^*\right)(\xi_{m\Delta}, y)\mu_{m\Delta}^{[k]}(dy)\right)^2\frac{\dd\P_{k'}}{\dd \P_{k}}\Big|_{\cF_{m\Delta}}\right]\nonumber\\
&\quad \leq \frac{C_{b^*_{\max}}}{M} \sum_{m=0}^{M-1} \E_{k}\left[ \left(\int_{\R} \left(\w{b}_k-{b}_k^*\right)(\xi_{m\Delta}, y)\mu_{m\Delta}^{[k]}(dy)\right)^2\exp\big(\mathcal{M}_{m\Delta}^{k',k}\big)\right]\label{tentative}\\
&\quad \leq C_{ b^*_{\max}}\,\exp(a)\left\| \w{b}_k-{b}_k^* \right\|^2_{M, k,k,\star}+4C_{b^*_{\max}}\frac{1}{M}\sum_{m=0}^{M-1}\E_{k}\left[\exp\big(\mathcal{M}_{m\Delta}^{k',k}\big)\one_{\left\{\mathcal{M}_{m\Delta}^{k',k}>a\right\}}\right].\nonumber
\end{align}
For the second term in the above inequality, note that Lemma 2.1 in \cite{vandeGeer1995} implies that there exists a constant $c$ depending on $b_{\max}$ such that  $$\P(\mathcal{M}_{t}^{k,k'}>a)\leq \exp(-a^2/c)$$
and that 
\begin{equation}\label{eq:sup-m-ij-cochet}
    \sup_{t\in[0,1]}\langle \mathcal{M}^{k,k'}\rangle_t\leq 4b^{*2}_{\max},
\end{equation}
since the drift functions $b_k^*, \,k\in\cY$ are bounded by the constant $b^*_{\max}$. 
Hence, for every $t\in[0,1]$, we have 
\begin{align}
\E\left[ \exp\Big(\mathcal{M}_{t}^{k,k'}\Big)\one_{\big\{\mathcal{M}_{t}^{k,k'}>a\big\}}\right]&\leq \sqrt{\P\left(\mathcal{M}_{t}^{k,k'}>a\right)}\sqrt{\E\left[\exp\Big(2\mathcal{M}_{t}^{i,j}\Big)\right]}\nonumber\\
&\leq \exp\left(-\frac{a^2}{2c}\right)\sqrt{\E\left[\exp\Big(2\mathcal{M}_{t}^{k,k'}-2\langle \mathcal{M}^{k,k'}\rangle_t\Big)\exp\big(2\langle \mathcal{M}^{k,k'}\rangle_t\big)\right]}\nonumber\\
&\leq \exp\left(-\frac{a^2}{2c}+4b^*_{\max}\right)\sqrt{\E\left[\exp\Big(2\mathcal{M}_{t}^{k,k'}-2\langle \mathcal{M}^{k,k'}\rangle_t\Big)\right]}\nonumber\\
&\leq \exp\left(-\frac{a^2}{2c}+2b^*_{\max}\right)=C_{ b^*_{\max}}\exp\left(-\frac{a^2}{2c}\right),\nonumber
\end{align}
where the first inequality above follows from Cauchy-Schwarz inequality, the third inequality comes from \eqref{eq:sup-m-ij-cochet}, and the last inequality is derived from the fact that the process $\Big(\!\exp\big(2\mathcal{M}_{t}^{k,k'}-2\langle \mathcal{M}^{k,k'}\rangle_t\big)\Big)_{t\in[0,1]}$ is a martingale by applying \eqref{eq:sup-m-ij-cochet} and the Novikov Theorem (see e.g. \cite[Theorem 5.23]{legall2016brownian}). 
Consequently, by setting $a=\sqrt{2c\log(N)}$, we obtain

\begin{align}\label{eq:eqnormstar}
&\E \left[\Big\Vert \w{b}_k-{b}_k^* \Big\Vert_{M, k, k', \star}^2\right] \leq C_{ b^*_{\max}}\,\exp(a)\E \left[\left\Vert \w{b}_k-{b}_k^* \right\Vert^2_{M, k, k,\star}\right]+C_{b^*_{\max}}\exp\left(-\frac{a^2}{2c}\right)\nonumber\\
&\quad \leq C_{b^*_{\max}}\,\exp(\sqrt{2c\log(N)})\E \left[\left\Vert \w{b}_k-{b}_k^* \right\Vert^2_{M, k, k,\star}\right]+\frac{C_{b^*_{\max}}}{\sqrt{N}}.
\end{align}

Therefore, applying Lemma~\ref{lem:boundIntegrated} 
to $\w{b}_k-{b}_k^*$, we obtain 
\begin{align*}
\E \left[\Big\Vert \widehat{b}_k-{b}_k^* \Big\Vert_{M, k, k'}^2\right] & \leq C_{ b^*_{\max}}\,\exp(\sqrt{2c\log(N)})\E \left[\left\Vert \widehat{b}_k-{b}_k^* \right\Vert^2_{M, k, k}\right]+C_{ b^*_{\max}}\log(N)\frac{\exp(\sqrt{c\log(N)})}{N}\nonumber\\
&\qquad\quad +
\frac{C_{b^*_{\max}}}{\sqrt{N}}.
\end{align*}

\newpage
\bibliographystyle{ScandJStat}
\bibliography{biblioMKV}

\newpage

\appendix

\begin{center}
    \Large{\bf Appendix}
\end{center}
\section{Numerical results of drift estimation}

In this section, we evaluate the numerical performance of the drift estimator. For each class $k$, the quantity of interest is $\E\left[|\widehat b_k-b_k^*|_{M,k,k}^2\right]$ introduced in~\eqref{eq:normMkk}. We approximate this quantity using independent auxiliary particle systems. First, for each repetition $r\in\{1, 2, \dots, 50\}$, we compute a finite-particle auxiliary empirical MSE
\begin{equation}\label{eq:empirical_estimation_error-b_k}
\widehat{\mathcal E}_{k}^{(r)}
=
\frac{1}{N_{\mathrm{aux}} M}
\sum_{n=1}^{N_{\mathrm{aux}}}
\sum_{m=0}^{M-1}
\left[
\frac{1}{N_{\mathrm{aux}}}
\sum_{n'=1}^{N_{\mathrm{aux}}}
\left(
\widehat b_k^{(r)}-b_k^*
\right)
\left(
X_{m\Delta}^{[k],n,r,\mathrm{aux}},
\widetilde X_{m\Delta}^{[k],n',r,\mathrm{aux}}
\right)
\right]^2.
\end{equation}
Here, $X^{[k],n,r,\mathrm{aux}},{1\leq n\leq N_{\mathrm{aux}}}$ and $\widetilde X^{[k],n',r,\mathrm{aux}}, {1\leq n'\leq N_{\mathrm{aux}}}$ are two auxiliary interacting particle systems that are independent of the training data and independent of each other. In the experiments, we set $N_{\mathrm{aux}}=1000$ particles per class. The auxiliary systems are observed on the same time grid as the corresponding training experiment, with $M\in{10,100}$ and $\Delta=1/M$.
Then we report the average mean for each class $k\in\cY$, computed by $\bar{\mathcal{E}}_k=\frac{1}{50}\sum_{r=1}^{50}\widehat{\mathcal{E}}_{k}^{(r)}$ 
and its associated standard deviation. 

\paragraph{Drift estimation results for Scenario A}

For Scenario~A, the tensor-product B-spline basis is constructed on the rectangular domain
$[-6,6]\times[-6,6]$. We set the spline resolution parameter to $D=2$ and use B-splines of polynomial degree $2$. The choice of $D$ is further investigated in Section~\ref{subsubsec:selection_b_spline_dim}.

Tables~\ref{tab:estimation-error-2labels-regime1} and~\ref{tab:estimation-error-2labels-regime2} report the empirical drift MSE for the two classes under Regimes~A(i) and~A(ii), respectively.



\begin{table}[H]
\centering
\caption{Empirical drift MSE in Regime~A(i) ($\theta_0=1$).}
\label{tab:estimation-error-2labels-regime1}
\begin{tabular}{ccccccccc}
\hline 
$M_{\text{train}}$ & $N_{\text{train}}$ 
& \makecell{\vspace{-0.25cm}\\Mean\\accuracy}
& \makecell{\vspace{-0.25cm}\\Class 1\\accuracy}
 & \makecell{\vspace{-0.25cm}\\Class 2\\accuracy} \vspace{0.2cm} \\
\hline
10  & 100  & 0.116 & 0.085 (0.044) & 0.148 (0.085)\\
10  & 1000 & 0.042 & 0.023 (0.005) & 0.061 (0.011)\\
100 & 100  & 0.132 & 0.090 (0.052) & 0.174 (0.092)\\
100 & 1000 & 0.045 & 0.021 (0.005) & 0.068 (0.009)\\
\hline
\end{tabular}%
\end{table}



\begin{table}[H]
\centering
\caption{Empirical drift MSE in Regime~A(ii) ($\theta_0=0.5$).}
\label{tab:estimation-error-2labels-regime2}
\begin{tabular}{cccccccccc}
\hline 
$M_{\text{train}}$ & $N_{\text{train}}$ 
& \makecell{\vspace{-0.25cm}\\Mean\\accuracy}
& \makecell{\vspace{-0.25cm}\\Class 1\\accuracy}
 & \makecell{\vspace{-0.25cm}\\Class 2\\accuracy}  \vspace{0.2cm}\\
\hline
10  & 100  & 0.113 & 0.072 (0.044) & 0.153 (0.071)\\
10  & 1000 & 0.037 & 0.012 (0.004) & 0.062 (0.014)\\
100 & 100  & 0.114 & 0.077 (0.041) & 0.152 (0.090)\\
100 & 1000 & 0.040 & 0.011 (0.004) & 0.069 (0.009)\\
\hline
\end{tabular}%
\end{table}

The results show that increasing the number of training particles from $N_{\mathrm{train}}=100$ to $N_{\mathrm{train}}=1000$ reduces the drift estimation error in both regimes. By contrast, increasing the number of observation times from $M_{\mathrm{train}}=10$ to $M_{\mathrm{train}}=100$ does not systematically improve the estimation error in Scenario~A. This is consistent with the numerical classification results, for which the number of training particles has a more visible effect than the observation frequency.


\paragraph{Drift estimation results for Scenario B}

For Scenario~B, the tensor-product B-spline basis is constructed on the rectangular domain
$[-12,12]\times[-12,12]$. We set the spline resolution parameter to $D=20$ and again use B-splines of polynomial degree $2$. Table~\ref{tab:drift-error} reports the empirical drift MSE for the five classes.

\begin{table}[H]
\centering
\caption{Estimation error of the drift functions $b_k, 1\leq k\leq 5$ for different training configurations. 
}
\label{tab:drift-error}
\begin{tabular}{ccccccccc}
\hline 
$M_{\text{train}}$ & $N_{\text{train}}$ 
& \makecell{\vspace{-0.25cm}\\Mean\\ error}
& \makecell{\vspace{-0.25cm}\\Class 1\\ error}
 & \makecell{\vspace{-0.25cm}\\Class 2\\ error} & \makecell{\vspace{-0.25cm}\\Class 3\\ error}& \makecell{\vspace{-0.25cm}\\Class 4\\ error}  & \makecell{\vspace{-0.25cm}\\Class 5\\ error}\vspace{0.2cm} \\
\hline
10  & 100
& 1.874
& 2.942 (0.851)
& 1.210 (0.635)
& 1.292 (1.166)
& 1.033 (0.703)
& 2.895 (0.807)\\
10  & 1000
& 0.905
& 1.746 (0.327)
& 0.247 (0.095)
& 0.564 (0.762)
& 0.266 (0.095)
& 1.699 (0.215)\\
100 & 100
& 1.320
& 1.504 (0.881)
& 1.046 (0.709)
& 1.619 (1.459)
& 0.961 (0.583)
& 1.469 (0.776)\\
100 & 1000
& 0.159
& 0.195 (0.107)
& 0.097 (0.059)
& 0.227 (0.250)
& 0.091 (0.052)
& 0.183 (0.071)\\
\hline
\end{tabular}%
\end{table}

Similar to Scenario~A, increasing the number of training particles reduces the drift estimation error. For $N_{\mathrm{train}}=1000$, the mean error across the five classes decreases from $0.905$ for $M_{\mathrm{train}}=10$ to $0.159$ for $M_{\mathrm{train}}=100$ and the smallest drift estimation error is obtained for the largest training configuration $(M_{\mathrm{train}},N_{\mathrm{train}})=(100,1000)$.

\section{Selection of the B-spline basis dimension}\label{subsubsec:selection_b_spline_dim}

In this section, we study the choice of the B-spline dimension parameter $D$ for Scenario~A. Since the particles within a given system are dependent, we compare two ways of constructing a validation sample: either by holding out particles from the same system or by using particles from an independently simulated system.

We consider $D\in\{2,5,10,15,20\}$. For each class and each Monte Carlo repetition, particle systems of size $N_{\mathrm{total}}\in\{100,1000\}$ are simulated. As in the drift estimation procedure, independent systems are used for the path sample and for the empirical measure. In each system, $60\%$ of the particles are used for fitting and the remaining $40\%$ are reserved for validation.

In the within-system setting, the validation particles are taken from the same interacting systems as the fitting particles. They are therefore not independent of the fitting sample. In the independent-system setting, new particle systems of the same size $N_{\mathrm{total}}$ are simulated independently, and a validation sample of size $N_{\mathrm{val}}=40\%N_{\mathrm{total}}$ is drawn from these systems. 

For each candidate value of $D$, we compute the validation classification accuracy and the mean drift MSE over the two classes. In each repetition, $D$ is selected by maximizing the classification accuracy or minimizing the drift MSE, with ties broken in favor of the smaller value. The tables below report the corresponding selection frequencies over 50 repetitions.

\paragraph{Independent-system validation.} The selection frequencies based on classification accuracy are reported below.

\begin{table}[H]
\centering
\caption{Selection of the B-spline basis dimension under independent-system validation based on classification accuracy. }
\label{tab:bestD-classif-accuracy}
\begin{tabular}{ccccccccc}
\hline
Regime & $M$ & $N_{\mathrm{total}}$ 
& $D=2$ & $D=5$ & $D=10$ & $D=15$ & $D=20$ \\
\hline
A(i)	&10	&100	&\textbf{0.58}&	0.14	&0.16	&0.02&	0.10\\
A(i)&	10	&1000&	\textbf{0.30}&	0.24	&0.24&	0.10&	0.12\\
A(i)&	100&	100&	\textbf{0.44}&	0.34	&0.14&	0.04	&0.04\\
A(i)	&100&	1000&	\textbf{0.30}&	0.24&	0.18&	0.10&	0.18\\
A(ii)&	10	&100	&\textbf{0.44}&0.22&	0.16&	0.08&	0.10\\
A(ii)&	10&	1000&\textbf{0.36}&	0.20&	0.20&	0.18&	0.06\\
A(ii)&	100&	100&	\textbf{0.54}&	0.20&	0.14&	0.06&	0.06\\
A(ii)&	100&	1000&	0.28&	0.14&	\textbf{0.30}&	0.16&	0.12\\
\hline
\end{tabular}%
\end{table}

The corresponding selection frequencies based on the drift MSE are:

\begin{table}[H]
\centering
\caption{Selection of the B-spline basis dimension under independent-system validation based on drift estimation MSE.}
\label{tab:bestD-drift-mse}
\begin{tabular}{ccccccccc}
\hline
Regime & $M$ & $N_{\mathrm{total}}$ 
& $D=2$ & $D=5$ & $D=10$ & $D=15$ & $D=20$ \\
\hline
A(i)	&10	&100	&\textbf{0.96}&	0.04	&0.00&	0.00	&0.00 \\
A(i)	&10&	1000&	\textbf{0.96}&	0.02&	0.02&	0.00	&0.00 \\
A(i)	&100&	100	&\textbf{0.94}	&0.06&	0.00&	0.00&	0.00 \\
A(i)	&100&	1000&	\textbf{0.96}&	0.00&	0.04&	0.00&	0.00 \\
A(ii)&	10&	100&	\textbf{0.94}&	0.06	&0.00&	0.00&	0.00 \\
A(ii)	&10&	1000&\textbf{	0.98}&	0.00&	0.02&	0.00&	0.00 \\
A(ii)&	100	&100	&\textbf{0.98}&	0.02&	0.00&	0.00&	0.00 \\
A(ii)	&100	&1000&\textbf{	0.96}&	0.00&	0.04&	0.00&	0.00 \\
\hline
\end{tabular}%
\end{table}

\paragraph{Within-system validation.} For validation based on held-out particles from the same interacting systems, the classification-based selection frequencies are:

\begin{table}[H]
\centering
\caption{Selection of the B-spline basis dimension under within-system validation based on classification accuracy. }
\label{tab:bestD-classif-accuracy-within-sys}
\begin{tabular}{ccccccccc}
\hline
Regime & $M$ & $N_{\mathrm{total}}$ 
& $D=2$ & $D=5$ & $D=10$ & $D=15$ & $D=20$ \\
\hline
A(i)&	10&	100&	\textbf{0.44}&	0.26&	0.08&	0.18&	0.04\\
A(i)&	10&	1000&	\textbf{0.32}&	0.20&	0.22&	0.16&	0.10\\
A(i)&	100&	100&	\textbf{0.52}&	0.16&	0.14&	0.12&	0.06\\
A(i)&	100&	1000&	\textbf{0.30}&	0.22&	0.26&	0.12&	0.10\\
A(ii)&	10&	100&	\textbf{0.48}&	0.22&	0.16&	0.06&	0.08\\
A(ii)&	10&	1000&\textbf{	0.42}&	0.12&	0.28&	0.12&	0.06\\
A(ii)&	100&	100&	\textbf{0.50}&	0.20&	0.12&	0.12	&0.06\\
A(ii)&	100&	1000&\textbf{	0.42}	&0.10	&0.26&	0.10&	0.12\\
\hline
\end{tabular}%
\end{table}

The selection frequencies based on the drift MSE are:

\begin{table}[H]
\centering
\caption{Selection of the B-spline basis dimension under within-system validation based on drift estimation MSE. }
\label{tab:bestD-drift-MSE-within-sys}
\begin{tabular}{ccccccccc}
\hline
Regime & $M$ & $N_{\mathrm{total}}$ 
& $D=2$ & $D=5$ & $D=10$ & $D=15$ & $D=20$ \\
\hline
A(i)&	10	&100	&\textbf{0.98}&	0.02&	0.00&	0.00&	0.00\\
A(i)&	10	&1000&	\textbf{0.98}&	0.02	&0.00&	0.00&	0.00\\
A(i)&	100	&100	&\textbf{0.94}	&0.06&	0.00&	0.00&	0.00\\
A(i)&	100	&1000	&\textbf{0.92}&	0.00&	0.08&	0.00&	0.00\\
A(ii)&	10	&100	&\textbf{0.92}&	0.08&	0.00&	0.00&	0.00\\
A(ii)&	10	&1000	&\textbf{0.98}&	0.00&	0.02&	0.00&	0.00\\
A(ii)&	100&	100&	\textbf{0.98}&	0.02&	0.00&	0.00&	0.00\\
A(ii)&	100&	1000	&\textbf{1.00}&	0.00&	0.00&	0.00&	0.00\\
\hline
\end{tabular}%
\end{table}

The two validation strategies give similar results overall. $D=2$ is the most frequently selected value in nearly all cases. The only exception under independent-system validation occurs for Regime~A(ii) with $(M,N_{\mathrm{total}})=(100,1000)$, where $D=10$ is selected in $30\%$ of the repetitions and $D=2$ in $28\%$. The drift-MSE criterion provides a considerably more stable selection. Under both validation strategies, $D=2$ is selected in between $92\%$ and 100\% of the repetitions, while values $D\geq 15 $are never selected.

Overall, using held-out particles from the same interacting systems leads to conclusions close to those obtained with independent validation systems. This also supports the choice $D=2$ in the numerical experiments for Scenario~A.

\end{document}